\documentclass[a4paper]{article}
\usepackage{amsmath,amsthm,amssymb,graphicx}

\usepackage{color}
\usepackage{version}
\newenvironment{NB}{
\color{red}{\bf NB}. \footnotesize
}{}
\excludeversion{NB}
\newenvironment{NB2}{
\color{blue}{\bf NB}. \footnotesize
}{}
\excludeversion{NB2}

\newenvironment{NB3}{
\color{green}{\bf NB}. \footnotesize
}{}
\excludeversion{NB3}

\newenvironment{NB4}{
\color{red}x{\bf NB}. \footnotesize
}{}
\excludeversion{NB4}

\usepackage[all]{xy}

\newtheorem{thm}{Theorem}[section]
\newtheorem{defn}[thm]{Definition}
\newtheorem{ex}[thm]{Example}
\newtheorem{prop}[thm]{Proposition}
\newtheorem{cor}[thm]{Corollary}
\newtheorem{lem}[thm]{Lemma}

\newtheorem{rem}[thm]{Remark}

\newcommand{\mf}[1]{{\mathfrak{#1}}}
\newcommand{\mr}[1]{{\mathrm{#1}}}
\newcommand{\mb}[1]{{\mathbf{#1}}}
\newcommand{\bb}[1]{{\mathbb{#1}}}
\newcommand{\mca}[1]{{\mathcal{#1}}}

\newcommand{\Hom}{\mr{Hom}}
\newcommand{\Ext}{\mr{Ext}}
\newcommand{\homd}{\mr{hom}}
\newcommand{\ext}{\mr{ext}}
\newcommand{\End}{\mr{End}}
\newcommand{\ghom}{\mca{H}om}

\newcommand{\gend}{\mca{E}nd}

\newcommand{\im}{\mr{im}}
\newcommand{\coker}{\mr{coker}}
\newcommand{\dimv}{\underline{\dim}\,}
\newcommand{\rk}{r}

\newcommand{\Z}{\bb{Z}}
\newcommand{\C}{\bb{C}}
\newcommand{\CP}{\bb{P}}
\newcommand{\OO}{\mca{O}}

\newcommand{\R}{\bb{R}}
\newcommand{\LL}{\bb{L}}

\newcommand{\Q}{\tilde{Q}}
\newcommand{\A}{\tilde{A}}
\newcommand{\V}{\tilde{V}}
\newcommand{\F}{\tilde{F}}
\newcommand{\tz}{\tilde{\zeta}}
\newcommand{\vv}{\mb{v}}
\newcommand{\q}{\mb{q}}

\newcommand{\perv}{\mr{Per}(Y/X)}
\newcommand{\pervc}{\mr{Per}_c(Y/X)}
\newcommand{\tperv}{\overline{\mr{Per}}(Y/X)}
\newcommand{\tpervc}{\overline{\mr{Per}}_c(Y/X)}
\newcommand{\fperv}{\mr{Per}(Y^+/X)}

\newcommand{\ftpervc}{\overline{{}^0\mr{Per}}_c(Y^+/X)}
\newcommand{\coh}{\mr{Coh}(Y)}
\newcommand{\cohc}{\mr{Coh}_c(Y)}
\newcommand{\fcoh}{\mr{Coh}(Y^+)}
\newcommand{\fcohc}{\mr{Coh}_c(Y^+)}
\newcommand{\amod}{A\text{-}\mr{Mod}}
\newcommand{\afmod}{A\text{-}\mr{mod}}
\newcommand{\hamod}{\tilde{A}\text{-}\mr{Mod}}
\newcommand{\hafmod}{\tilde{A}\text{-}\mr{mod}}
\newcommand{\mA}{\mca{A}}
\newcommand{\mamod}{\mr{Coh}(\mA)}
\newcommand{\macmod}{\mr{Coh}_c(\mA)}
\newcommand{\hmamod}{\mr{Coh}(\tilde{\mA})}
\newcommand{\hmacmod}{\mr{Coh}_c(\tilde{\mA})}

\newcommand{\M}[2]{\mf{M}_{#1}(#2)}
\newcommand{\Ms}[2]{\mf{M}^{\mathrm{s}}_{#1}(#2)}

\newcommand{\HN}{the Harder-Narasimhan filtration }
\newcommand{\JH}{a Jordan-H\"older filtration }
\newcommand{\pc}{perverse coherent }

\newcommand{\shfO}{\mathcal O}

\usepackage{pst-all}
\renewcommand{\labelenumi}{\textup{(\theenumi)}}%

\numberwithin{equation}{section}

\title{Counting invariant of perverse coherent sheaves and its wall-crossing}
\author{Kentaro Nagao and Hiraku Nakajima \\
Department of Mathematics, Kyoto University\\
  Kyoto 606-8502, Japan
}

\begin{document}

\maketitle

\begin{abstract}
  We introduce moduli spaces of stable perverse coherent systems on
  small crepant resolutions of Calabi-Yau $3$-folds and consider their
  Donaldson-Thomas type counting invariants. The stability depends on
  the choice of a component (= a chamber) in the complement of
  finitely many lines (= walls) in the plane. We determine all walls
  and compute generating functions of the invariants for all choices of
  chambers when the Calabi-Yau is the resolved conifold. For suitable
  choices of chambers, our invariants are specialized to
  Donaldson-Thomas, Pandharipande-Thomas and Szendroi invariants.
\end{abstract}

\begin{NB2}
If you write something in 
\verb+\begin{NB}+ and
\verb+\end{NB}+, it appears in the red. But if you define
\verb+\excludeversion{NB}+ for a final draft, it will disappear.

It will be no problem if we use this after a paragraph, but
if you use this in a sentence, put \verb+%+ after
\verb+\end{NB}+ like \verb+\end{NB}%+
to avoid a wrong spacing.
\end{NB2}

\section*{Introduction}

In this paper we study {\it variants\/} of Donaldson-Thomas (DT in
short) invariants \cite{thomas-dt} for the crepant resolution $f\colon
Y \to X = \{ xy - zw = 0\}$ of the conifold where $Y$ is the total
space of the vector bundle $\OO_{\CP^1}(-1)\oplus\OO_{\CP^1}(-1)$. The
ordinary DT invariants are defined by the virtually counting 
\begin{NB} countings? \end{NB}%
of moduli
spaces of ideal sheaves of curves. 
A variant has been introduced by Pandharipande-Thomas (PT in short)
recently \cite{pt1}. 
PT invariants are defined by the virtual counting of
moduli spaces of stable coherent systems, i.e., pairs of
$1$-dimensional sheaves $F$ and homomorphisms $s\colon \shfO_Y \to F$.
Both DT and PT invariants are defined for arbitrary Calabi-Yau
$3$-folds.
For the resolved conifold $Y$, yet another variant was introduced by
Szendroi \cite{szendroi-ncdt} (see also \cite{young-conifold}). 
His invariants are defined by the virtual counting of moduli spaces of
representations of a certain noncommutative algebra. This
noncommutative algebra has its origin in the celebrated works of
Bridgeland~\cite{bridgeland-flop} and
Van~den~Bergh~\cite{vandenbergh-3d}. In particular, we can interpret the invariants as virtual counting of moduli spaces of {\it perverse
  ideal sheaves\/}, originally introduced in order to describe the
flop $f^+\colon Y^+\to X$ of $Y$ as a moduli space.
Those three classes of invariants have been computed for the resolved
conifold and their generating functions are given by infinite products.  

In this paper we introduce more variants by using moduli spaces of
stable {\it perverse coherent systems\/}, i.e., pairs of
$1$-dimensional perverse coherent sheaves $F$ and homomorphisms
$s\colon \shfO_Y \to F$. The stability condition is determined by a
choice of parameter $\zeta = (\zeta_0,\zeta_1)$ in the complement of
finitely many lines in $\R^2$. (The number of lines increases when the Hilbert
polynomial of $F$ becomes large.) The invariants depend only on the
{\it chamber\/} containing $\zeta$. When we choose the chamber
appropriately, our new invariants recover DT, PT invariants for $Y$
and the flop $Y^+$, as well as Szendroi's invariants. See
Figure~\ref{fig:zeta}.
Our main results are
\begin{enumerate}
  \renewcommand{\theenumi}{\alph{enumi}}%
  \renewcommand{\labelenumi}{(\theenumi)}%
\item the determination of all walls, and
\item the computation of invariants for any choice of a
chamber. 
\end{enumerate}
The generating function of invariants is given by an infinite
product, dropping certain factors from the generating function of
Szendroi's invariants. The stability parameter determines which
factors we should drop in a very simple way (see
Theorems~\ref{wallcrossing}, \ref{thm-main}).
Finally in the chamber $\zeta_0,\zeta_1>0$ the generating function is simply $1$. 

\begin{figure}[htbp]
\def\JPicScale{.8}
  \centering
  \includegraphics{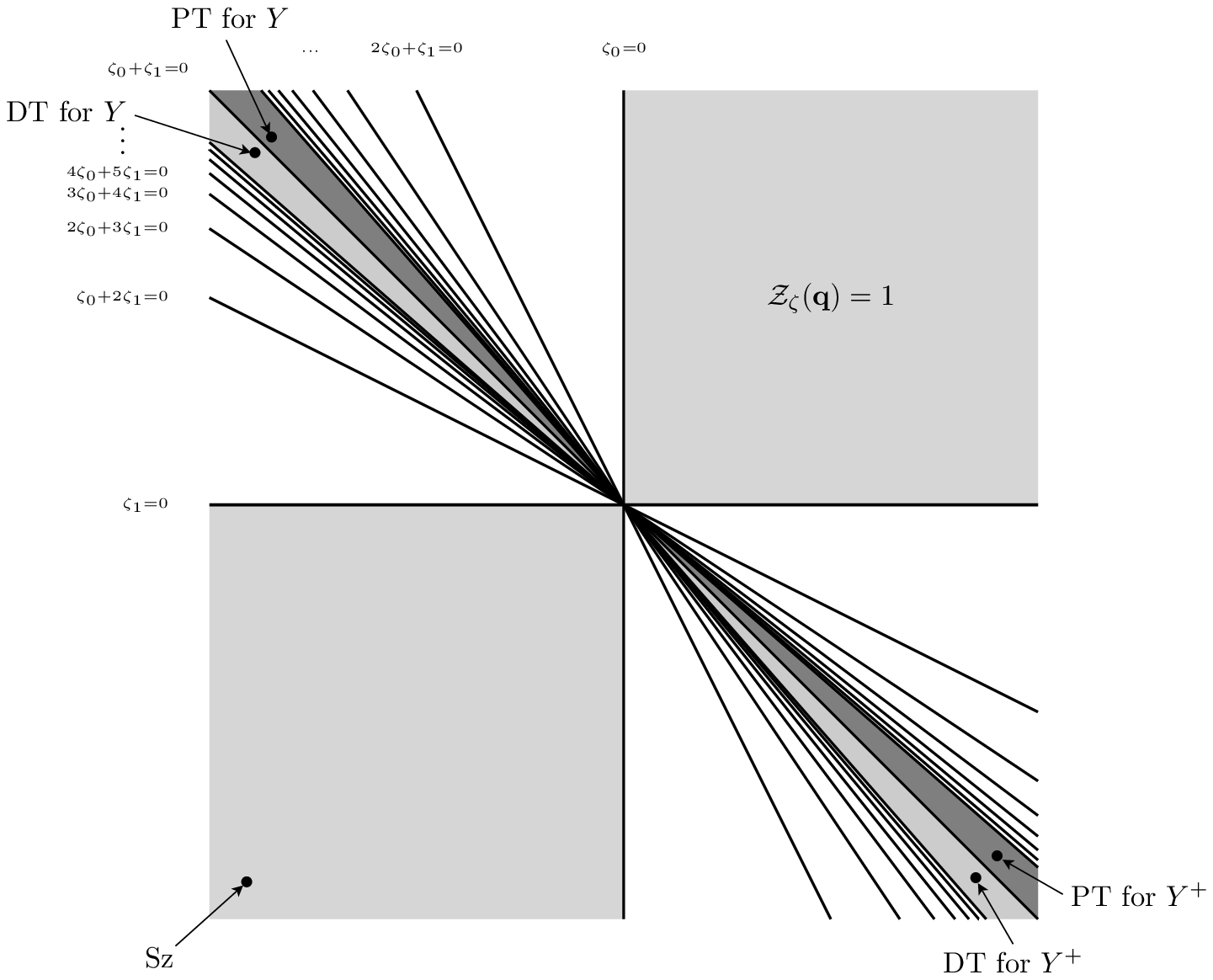}
  \caption{chamber structure}
\label{fig:zeta}
\end{figure}

Besides Szendroi's work \cite{szendroi-ncdt} our definition is
motivated also by the second named author's work with K\={o}ta
Yoshioka \cite{ny-perv1,ny-perv2}, where very similar moduli spaces
and chamber structures have been studied for the case when $Y\to X$ is
the blow-up of a nonsingular complex surface. A similarity is natural
as $Y$ is locally the total space of $\shfO_{\CP^1}(-1)$, instead of
$\shfO_{\CP^1}(-1)\oplus \shfO_{\CP^1}(-1)$. In \cite{ny-perv2} the
virtual Betti numbers of moduli spaces are computed. When the rank of
sheaves is $1$, the generating function is again an infinite product,
dropping factors from the generating function of Betti numbers of
Hilbert schemes of points on the blow-up $Y$, given by the famous
G\"ottsche formula (\cite{gottsche}).

We compute invariants by proving the wall-crossing formula which
describes how the generating function of invariants changes when the
stability parameter crosses a wall except for the wall $\{(\zeta_0,\zeta_1)\mid \zeta_0+\zeta_1=0\}$ (Theorem~\ref{wallcrossing}).
%
In the companion paper \cite{3tcy} the first named author will
generalize the main result to more general small crepant resolutions
of toric Calabi-Yau 3-folds. 
The wall-crossing formula is an example of Joyce's wall-crossing
formulas (\cite{joyce-4}), and ones in more recent work by Kontsevich
and Soibelman (\cite{ks}) (see \cite[\S ~4]{3tcy}). 
Note that there is no substantial difference between virtual counting and actual Euler numbers in our setting (Theorem \ref{thm-sign}), so our computation can be examples of both
works.
In this paper we give an alternative elementary proof independent
of \cite{joyce-4,ks}.
The wall $\{(\zeta_0,\zeta_1)\mid \zeta_0+\zeta_1=0\}$, which we do not deal with, corresponds to the DT-PT conjecture (\cite{pt1}). 
See Remark \ref{rem-final} for this wall.

The paper is organized as follows: In \S 1 we introduce perverse
coherent systems and their stability. In \S 2 we show that our
moduli spaces parameterize ideal sheaves or stable coherent systems in
suitable choices of the stability parameters. This will be used to
establish that our invariants are specialized to DT and PT invariants.
In \S 3 we prove our main results.

After we posted a previous version of this paper to the arXiv, two physics papers \cite{jafferis-moore} and \cite{chuang-jafferis} appeared on the arXiv.
In \cite{jafferis-moore}, Jafferis and Moore provide a wall-crossing formula, which looks quite similar to ours.
In \cite{chuang-jafferis}, Chuang and Jafferis 
associate a new quiver ($\tilde{A}_m^\pm$ in \S \ref{subsec-potential}) to each chamber and conjecture that the moduli spaces are isomorphic to the corresponding moduli spaces of $\zeta$-stable perverse coherent systems. 
The quiver $\tilde{A}_m^-$ is obtained from our quiver ($\tilde{A}$ in Figure \ref{newquiver}) by successive mutations in the sense of \cite{quiver-with-potentials}.
We devote \S 4 to a proof of their conjecture and to a proof of Theorem \ref{thm-sign}.

\subsection*{Acknowledgements}
The first named author (KN) is supported by JSPS Fellowships for Young Scientists (No.\ 19-2672). 
He is grateful to Y.~Kimura for teaching him some references on Calabi-Yau algebras.
He is also grateful to E.~Macri for useful comments.

The second named author (HN) is supported by the Grant-in-aid for
Scientific Research (No.\ 19340006), JSPS. This work was started while
he was visiting the Institute for Advanced Study with supports by the
Ministry of Education, Japan and the Friends of the Institute.
He is grateful to K.~Yoshioka for discussions on the result of this
paper, in particular, the explanation how to construct moduli spaces
for a projective $Y$ in \S \ref{moduli}.
He is also grateful to Y.~Soibelman for sending us a preliminary version
of \cite{ks}.

While the authors were preparing this paper, they are informed that
Bridgeland and Szendroi also reproduce the formula of Szendroi's
invariants \cite{young-conifold} via the wall-crossing formula in the
spirit of \cite{joyce-4,BTL,ks}. They thank T.~Bridgeland for discussion.

They thank W. Chuang for answering questions on the paper \cite{chuang-jafferis}. They also thank the referee for the valuable comments.

\section{Perverse coherent systems}
\subsection{Category of perverse coherent systems}\label{sec1.1}
Let $f\colon Y\to X$ be a projective morphism between quasi-projective varieties over $\C$ such that the fibers of $f$ have dimensions less than $2$ and such that $\R f_*\OO_Y = \OO_X$.
\begin{NB}
  Corrected. June 14 

  But I do not check the `birationality' is unnecessary or not. June 17
\end{NB}

A {\it perverse coherent sheaf\/}%
\begin{NB}
  I put words in italic. June 14
\end{NB}
(\cite{bridgeland-flop}, \cite{vandenbergh-3d}) is an object $E$ of $D^b(\coh)$ satisfying
\begin{itemize}
\item $H^i(E)=0$ unless $i=0,-1$,
\item $\R^1f_*(H^0(E))=0$ and $\R^0f_*(H^{-1}(E))=0$,
\item $\Hom(H^{0}(E),C)=0$ for any sheaf $C$ on $Y$ satisfying $\R f_*(C)=0$.
\end{itemize}
\begin{NB}
  I change the definition. This is ${}^{-1}\perv$ in Van den Bergh's notation. Bridgeland's t-structure is ${}^{0}\perv$ in his notation.  I use this since we have
\[
{}^{-1}\perv\simeq \mamod \simeq {}^0\fperv,
\]
and so this is suitable to describe the relation with coherent systems on $Y^+$. 6/18
\end{NB}%
Let $\perv\subset D^b(\coh)$ be the full subcategory consisting 
\begin{NB} corrected 7/5 \end{NB}%
of all perverse coherent sheaves. 
This is the core of a t-structure of $D^b(\coh)$. In particular, $\perv$ is an abelian category. 

\begin{rem}
The abelian category $\perv$ we define here is ${}^{-1}\perv$ in the notations of \cite{bridgeland-flop} and \cite{vandenbergh-3d}.
We will also use ${}^{0}\fperv$ to study the flop $Y^+\to X$ later (\S \ref{flop}). \begin{NB} added. 6/18 \end{NB}%
\end{rem}

\begin{NB}
The following definition is modified after the Morita
equivalence. July 30, HN.

Note that $\R f_*(E)$ is a coherent sheaf on $X$ for $E\in\perv$.
Let $D^b_c(\coh)$ denote the subcategory of $D^b(\coh)$ consisting of objects
$E$ such that $\R f_*(E)$ have $0$-dimensional supports
\begin{NB2}
  Changed, original was `compact supports'
\end{NB2}%
and $\mathrm{Per}_c(Y/X):=D^b_c(\coh)\cap \perv$. Here we suppose that
the $0$ object has a $0$-dimensional support as a convention.

\begin{NB2}
  When $X$ is not necessarily affine, this condition is probably better.
\end{NB2}

\end{NB}

\begin{defn}
A {\it perverse coherent system} on $Y$ is a pair $(F,W,s)$ of a \pc sheaf $F$, a vector space $W$ and a homomorphism
\begin{NB}
Corrected. June 14
\end{NB}%
$s\colon W\otimes_{\C}\OO_Y\to F$. 
We call $W$ the {\it framing} of a perverse coherent system.
\begin{NB4}
original: a finite dimensional vector space $W$
\end{NB4}

A morphism between \pc systems $(F,W,s)$ and $(F',W',s')$ is a pair of morphisms $F\to F'$ and $W\to W'$ which are compatible with $s$ and $s'$. 
\end{defn}
We sometimes denote a \pc system $(F,W,s)$ with $W=\C$ by $(F,s)$, and a \pc system $(F,W,s)$ with $W=0$ by $F$ for brevity.
\begin{NB}
  Corrected, June 17.
\end{NB}

Let $\tperv$ denote the category of \pc systems on $Y$.
It is an abelian category.

\begin{NB}
Original, kept on July 30, HN:

Let $\tperv$ denote the category of \pc systems on $Y$ and 
$\tpervc$ denote the subcategory of \pc systems $(F,W,s)$ with $F\in\pervc$.
They are abelian categories.
\end{NB}

\subsection{NCCR with framing}\label{framednccr}
\begin{NB}
  I have rewrote every definitions and statements so that they can be applied for the global situation (6/18) . 
\end{NB}%
\begin{NB}
I further rewrite the definition to begin with the case when $X$ is
quasi-projective. July 30, HN.
\end{NB}

When $X = \mr{Spec}(R)$ is affine, $\mca{P}\in \perv$ is called {\it a
  projective generator\/} if it is a projective object in $\perv$ and
$\Hom_{\perv}(\mca{P},E) = 0$ implies $E = 0$ for $E\in\perv$. For a
general $X$, $\mca{P}\in\perv$ is called {\it a local projective
  generator\/} if there exists an open covering $X = \bigcup U_i$ such
that $\mca{P}|_{U_i}$ is a projective generator of
$\mr{Per}(f^{-1}(U_i)/U_i)$ for all $i$.

A local projective generator $\mca{P}$ exists and can be taken as a vector
bundle (see \cite[Proposition 3.3.2]{vandenbergh-3d}).
When $X=\mr{Spec}(R)$ is affine, $\mca{P}$ is constructed as follows:
Let $\mca{L}$ be an ample line bundle on $Y$ generated by its global sections. 
Let $\mca{P}_0$ be a vector bundle given by an extension 
\begin{equation}\label{eq:ext}
  0\to \OO_Y^{\oplus r-1}\to\mca{P}_0\to \mca{L}\to 0 
\end{equation}
associated to a set of generators of $H^1(Y,\mca{L}^{-1})$ as an
$R$-module. Then $\mca{P} := \mca{P}_0\oplus \shfO_Y$ is a projective
generator \cite[Proposition 3.2.5]{vandenbergh-3d}. The general case can
be reduced to the affine case. We consider only a local projective generator
of a form $\mca{P} = \mca{P}_0\oplus \shfO_Y$ hereafter.

\begin{NB}
Original : kept on July 30, HN   

Assume that $X=\mr{Spec}(R)$ is affine. 
Let $\mca{L}$ be an ample line bundle on $Y$ generated by its global sections. 
Let $\mca{P}$ be a vector bundle given by an extension 
\[
0\to \OO_Y^{\oplus (r-1)}\to\mca{P}\to \mca{L}\to 0 
\] 
associated to a set of generators of $H^1(Y,\mca{L}^{-1})$ as an $R$-module. 
In the case $X$ is not affine, we can construct a vector bundle $\mca{P}$ as well (see \cite[Proposition 3.3.2]{vandenbergh-3d}).

\begin{prop}[\protect{\cite[Proposition 3.2.5 and Proposition 3.3.2]{vandenbergh-3d}}]
\begin{NB2}
    Corrected to a standard form.
\end{NB2}%
The vector bundle $\OO_Y\oplus \mca{P}$ is
\begin{NB2}
  corrected.
\end{NB2}%
a local projective generator in $\perv$. 
\end{prop}
\end{NB}

\begin{NB}
I change $\mca{P}$, $r$ as in \cite{vandenbergh-3d}. 
It seems that this must be $r-1$ in the formula in 
\S\ref{cs-pcs}. July 30, HN.
\end{NB}

\begin{thm}[\protect{\cite[Corollary 3.2.8 and Theorem
    A]{vandenbergh-3d}}]
\label{cor:Morita}
\begin{NB}
    Corrected to a standard form.
\end{NB}%
We denote the $\OO_X$-algebra $f_*\gend_Y(\mca{P})$ by $\mA$. 
Then the functors $\R f_*\R\ghom_Y(\mca{P},-)$ 
\begin{NB} corrected 7/5 \end{NB}%
and $-{\otimes}^{\LL}_\mA(\mca{P})$ give equivalences between
$D^b(\coh)$ and $D^b(\mamod)$, which are inverse to each other.
Here $\mamod$ denotes the category of coherent right $\mA$-modules.

Moreover, these equivalences restrict to equivalences between $\perv$
and
$\mamod$.

\begin{NB4}
I chage the notation.
This is the notation used in Van den Bergh's papaer.
We should explain the definition of ``coherent $\A$-module'' (I cann't find the definition in Van den Bergh's papaer!). 
\end{NB4}
\begin{NB}
Since I postpone to give a definition of $\pervc$, I delete the
corresponding statement. July 31, HN.  
\end{NB}

\begin{NB} I am not sure whether we should politely explain the affine case (6/19).

In particular, in the case $X$ is affine, the functors $\R\Hom_Y(\OO_Y\oplus \mca{P},-)$ and $-{\otimes}^{\LL}_A(\OO_Y\oplus \mca{P})$ give equivalences between $D^b(\coh)$ \textup(resp.\ $\perv$, $D^b_c(\coh)$, $\mathrm{Per}_c(Y/X)$\textup) and $D^b(\amod)$ \textup(resp.\ $\amod$, $D^b_f(\amod)$. 
$\afmod$\textup), 
Here we put $A=\End_Y(\OO_Y\oplus \mca{P})$ and 
$D^b_f(A\text{-}\mr{mod})$ is the subcategory of
$D^b(A\text{-}\mr{mod})$ consisting of objects with finite dimensional
cohomologies
and $A\text{-}\mr{fmod}:=D^b_f(A\text{-}\mr{mod})\cap A\text{-}\mr{mod}$.
\end{NB}
\end{thm}

\begin{defn}
  Let $\macmod$ \textup(resp.\ $D^b_c(\mamod)$\textup) denote the full
  subcategory of $\mamod$ \textup(resp.\ $D^b(\mamod)$\textup)
  consisting of objects $E$ which \textup(resp.\ whose cohomology
  groups\textup) have $0$-dimensional supports. Let $\pervc$ and 
  $D^b_c(\coh)$ be the corresponding categories under the equivalences
  in Corollary~\ref{cor:Morita}.
  Let $\tpervc$ denote the category of \pc systems $(F,W,s)$ such that $F\in\pervc$ and such that $W$ is finite dimensional.
\begin{NB4}
the definition of $\tpervc$ is added.
\end{NB4}
\end{defn}

\begin{NB3}
\begin{ex}\label{conifold}
Let $f\colon Y\to X$
\begin{NB}
  Changed `:' to \verb+\colon+
\end{NB}%
be the crepant resolution of the conifold, that is, $X=\{(x,y,z,w)\in \C^4\mid xy-zw=0\}$ and $Y$ is the total space of the vector bundle $\pi\colon \OO_{\CP^1}(-1)\oplus\OO_{\CP^1}(-1)\to\CP^1$.
Let us take $\mca{L}:=\pi^*\OO_{\CP^1}(1)$
\begin{NB}
  Corrected, June 16
\end{NB}%
as an ample line bundle.  Since $H^1(Y,\mca{L}^{-1})=0$, we have a
projective generator $\mca{P} := \OO_Y\oplus \mca{L}$. 
\begin{NB}
  corrected 7/5 
\end{NB}%

The endomorphism algebra $A$ is described as a quiver with
relations\textup:
\begin{NB}
  I changed the period to colon. June 16
\end{NB}%
Let $Q$ be the quiver in Figure \ref{quiver}.
Then we have
\[
A\simeq \C Q/(a_1b_ia_2=a_2b_ia_1,b_1a_ib_2=b_2a_ib_1)_{i=1,2}.
\begin{NB}
  Corrected, June 17.
\end{NB}
\]
Note that this relation is derived from the superpotential $\omega=a_1b_1a_2b_2-a_1b_2a_2b_1$ (\cite{berenstein-douglas} for example. See also \cite{quiver-with-potentials}). 
\begin{NB}
There is no explanation of (or the reference to) the `superpotential'.
\end{NB}
\end{ex}
\end{NB3}

Giving an object $V\in\mamod$ (resp.\ $\in\macmod$) is equivalent to giving the following data:
\begin{itemize}
\item a coherent (resp.\ finite length) $\OO_X$-module $V_0$,
\item an $\mA':=f_*\gend_Y(\mca{P}_0)$-module $V_1$ which is coherent (resp.\ finite length) as an $\OO_X$-bimodule, 
\begin{NB4}
coherent as an $\A$-bimodule ?  We should check.
\end{NB4}
\item a homomorphism $f_*\ghom_Y(\OO_Y,\mca{P}_0)\to \ghom_X(V_0,V_1)$
\begin{NB}
I changed $\Hom_C$ to $\Hom_\C$. June 17.
\end{NB}%
of $(\OO_X,\mA')$-bimodules, 
\item a homomorphisms $f_*\ghom_Y(\mca{P}_0,\OO_Y)\to \ghom_X(V_1,V_0)$ of $(\mA',\OO_X)$-bimodules. 
\end{itemize}


\begin{defn}
A framed $\mA$-module is a pair $(V,V_\infty,\iota)$ of an $\mA$-module $V$, a vector space $V_\infty$ and a linear map $\iota\colon V_\infty\to H^0(X,V_0)$.
\begin{NB4}
I omit ''finite dimensional'' as before.
\end{NB4}

A morphism between framed $\mA$-modules $(V,V_\infty,\iota)$ and $(V',V'_\infty,\iota')$ is a pair of an $\mA$-module homomorphism $V\to V'$ and a linear map $V_\infty\to V'_\infty$ which are compatible with $\iota$ and $\iota'$. 
\end{defn}

\begin{NB3}
\begin{ex}\label{ex-1.7}
Let $f\colon Y\to X$ be the crepant resolution of the conifold. 
We define the following new quiver with relations:
\[
\A\simeq\C \Q/(a_1b_ia_2=a_2b_ia_1,b_1a_ib_2=b_2a_ib_1)_{i=1,2},
\begin{NB}
  Corrected, June 17.
\end{NB}
\]
where $\Q$ is the quiver in Figure \ref{newquiver}.
The abelian category of framed $\mA$-modules $(V,V_\infty,\iota)$ (resp.\ framed $\mA$-modules $(V,V_\infty,\iota)$ with $\V_\infty\in \macmod$) is equivalent to the abelian category of finitely generated $\A$-modules (resp.\ finite dimensional $\A$-modules).
\end{ex}
\end{NB3}
We denote, with a slight abuse of notations, by $\hmamod$ the abelian category of framed $\mA$-modules and 
by $\hmacmod$ the subcategory of framed $\mA$-modules $(V,V_\infty,\iota)$ such that $V\in\macmod$ and such that $V_\infty$ is finite dimensional.
\begin{NB4}
``such that $V_\infty$ is finite dimensional'' added
\end{NB4}

\begin{prop}\label{prop-framed-morita}
The category $\tperv$ \textup(resp.\ $\tpervc$\textup) is equivalent
to the category $\hmamod$ \textup(resp.\ $\hmacmod$\textup). 
\end{prop}
\begin{proof}
First we put $W=V_\infty$.
Using the adjunction, we have
\begin{NB} changed ($\pi\to p$) 7/5 \end{NB}%
\begin{align*}
\Hom_Y(W\otimes_{\C}\OO_Y,F)&=\Hom_Y(p^*W,F)\\
&=\Hom_\C(V_\infty,p_*F)\\
&=\Hom_\C(V_\infty,H^0(X,V_0)),
\end{align*}
where $p$ is the projection from $Y$ to a point. The equivalences follow immediately.
\end{proof}

\subsection{Stability}

Let $\tz = (\zeta_0, \zeta_1, \zeta_\infty)$ be a triple of real
numbers. For a nonzero object $\F=(F,W,s)\in\tpervc$ we define
\begin{equation*}
   \theta_{\tz}(\F) :=
   \frac{\zeta_0\dim H^0(F) + \zeta_1\dim H^0(F\otimes\mca{P}_0^\vee)
   + \zeta_\infty\dim W}
   {\dim H^0(F) + \dim H^0(F\otimes\mca{P}_0^\vee) + \dim W}.
\end{equation*}

\begin{defn}\label{sfam}
A \pc system $\F\in\tpervc$ is
$\theta_{\tz}$-(semi)stable if we have
\begin{equation*}
   \theta_{\tz}(\F') \,(\le)\, \theta_{\tz}(\F)
\end{equation*}
for any nonzero proper subobject $0\neq \F'\subsetneq \F$ in $\tpervc$.
\end{defn}

Here we adapt the convention for the short-hand notation. The above
means two assertions: semistable if we have `$\le$', and stable if we
have `$<$'.

\begin{NB}
  This definition is slightly different from the stability of 
  framed modules. For the latter, we usually consider
  submodules $(F',W',s')$ with either $W'=0$ or $W' = W$. This
  corresponds to the stability condition for the `new quiver' in the
  sense of \cite[\S4]{ny-perv1}.
\end{NB}

\begin{rem}\label{rem:normalization}
\begin{NB}
added (7/5).
\end{NB}
\begin{NB}
    Added July 31, HN.
\end{NB}%
\begin{enumerate}
\item
As we shall see later, the space of the parameters $\tz$ has a
  chamber structure defined by integral hyperplanes so that the
  (semi)stability is unchanged if we stay in a chamber (see \S \ref{subsec-classification}).
  \begin{NB}
    Need a reference.
  \end{NB}%
%
\item
Given a real number $c$ let $\tz'$ be the triple of real numbers $(\zeta_0+c, \zeta_1+c, \zeta_\infty+c)$. 
Then we have
\[
\theta_{\tz'}(\F)=\theta_{{\tz}}(\F)+c.
\]
Hence $\theta_{\tz'}$-(semi)stability and $\theta_{\tz}$-(semi)stability are equivalent. 
In particular, given a $\theta_{\tz}$-(semi)stable \pc system $\F\in\tpervc$ we can normalize $\tz$ so that $\theta_{\tz}(\F)=0$. 
\item
This stability condition depends on the choice of $\tz$, {\it as
    well as\/} the choice of $\mca{P}_0$.
\item
Given a pair of real numbers $\zeta=(\zeta_0,\zeta_1)$, we define the $\theta_{\zeta}$-(semi)stability for a \pc sheaf by the same conditions. In other words, a \pc system $F\in\pervc$ is $\theta_{\zeta}$-(semi)stability if and only if the \pc system $(F,0,0)$ with trivial framing is $\theta_{\tz}$-(semi)stability for some (equivalently any) $\zeta_\infty$. 
\end{enumerate}
\end{rem}

\begin{thm}[\protect{\cite{rudakov}}]\begin{NB}added (6/19)\end{NB}%
\begin{NB} I changed the notations 7/5 \end{NB}%
Let a stability parameter $\tz\in\R^3$ be fixed.
\begin{enumerate}
\item A \pc system $\F\in\tpervc$ has a unique Harder-Narasimhan filtration:
\[
\F=\F^0\supset \F^1\supset \cdots \supset \F^L\supset \F^{L+1}=0
\]
such that $\F^l/\F^{l+1}$ is $\theta_{\tz}$-semistable for $l=0,1,\ldots,L$ and 
\[
\theta_{\tz}(\F^0/\F^{1})<\theta_{\tz}(\F^1/\F^2)<\cdots<\theta_{\tz}(\F^L/\F^{L+1}).
\]

\item A $\theta_{\tz}$-semistable \pc system $\F\in\tpervc$ has a Jordan-H\"older filtration:
\[
\F=\F^0\supset \F^1\supset \cdots \supset \F^L\supset \F^{L+1}=0
\]
such that $\F^l/\F^{l+1}$ is $\theta_{\tz}$-stable for $l=0,1,\ldots,L$ and 
\[
\theta_{\tz}(\F^0/\F^{1})=\theta_{\tz}(\F^1/\F^2)=\cdots=\theta_{\tz}(\F^L/\F^{L+1}).
\]
\end{enumerate}
\end{thm}

\subsection{Moduli spaces of perverse coherent systems}\label{moduli}
In this subsection, we assume that $\zeta_0$, $\zeta_1$ and $\zeta_\infty$ are rational numbers.
\begin{thm}\label{thm-construction-1}
There is a coarse moduli scheme parameterizing $S$-equivalence classes of $\theta_{\tilde{\zeta}}$-semistable objects in $\tpervc$.
\end{thm}
The theorem is deduced from a more general construction (Theorem \ref{thm-construction-2}), which was explained to the second named author by K\={o}ta Yoshioka . 
We only give a sketch of the proof, as we are mainly interested in the
case when $X$ is affine, and hence we can alternatively use the
construction in \cite{king} (see \S \ref{defofncdt}), 
and Yoshioka wrote a paper containing the proof (\cite[Proposition 1.6.1]{yoshioka_perverse}).

For a while, we assume that $X$ is projective. 
Take an ample line bundle $\OO_X(1)$ over $X$.
We define the $\mca{P}$-twisted Hilbert polynomial of $F\in\perv$ by
(cf.\ \cite[\S4]{Ytwisted2})
\begin{equation*}
   \chi(\mca{P},F(n)) = \chi(\R f_*(\mca{P}^\vee\otimes F)(n)).
\end{equation*}
From Theorem~\ref{cor:Morita} this is nothing but the usual Hilbert
polynomial for the corresponding sheaf $\R f_*(\mca{P}^\vee\otimes
F)$. We expand this as
\begin{equation*}
  \chi(\mca{P},F(n)) = \sum_i a_i^{\mca{P}}(F) \binom{n+i}{i}.
\end{equation*}
We say $F$ is {\it $d$-dimensional\/}, if $a_d^{\mca{P}}(F) > 0$ and
$a_i^{\mca{P}}(F) = 0$, $i > d$. Then an $d$-dimensional object
$F\in\perv$ is {\it $\mca{P}$-twisted (semi)stable\/} if
\begin{equation*}
  \chi(\mca{P},F'(n)) (\le) \frac{a_d^{\mca{P}}(F')}{a_d^{\mca{P}}(F)}
  \chi(\mca{P},F(n)) \qquad \text{for $n\gg 0$}
\end{equation*}
holds for any proper subobject $0\neq F'\subsetneq F$ in $\perv$.
From the inequality, $F$ cannot contain a nonzero subobject $F'$ with
$a_d^{\mca{P}}(F') = 0$. This condition is referred as `$F$ is {\it of
  pure dimension $d$}' in the usual stability for coherent
sheaves. Under that condition the above is equivalent to
\[
  \chi(\mca{P},F'(n))/a_d^{\mca{P}}(F')
  (\nolinebreak\le\nolinebreak)\linebreak[2]
  \chi(\mca{P},F(n))/ {a_d^{\mca{P}}(F)}.
\]

We can construct the moduli space of $\mca{P}$-twisted semistable
sheaves by modifying the construction of the moduli space of usual stable
sheaves by Simpson \cite{S:1} (see \cite{HL}) as follows:
By Theorem~\ref{cor:Morita} we may construct it as a moduli space of
semistable $\mA$-modules $\overline{F} = \R f_*(\mca{P}^\vee\otimes
F)$, where the stability condition is defined as usual.
Now $\mA$ is an example of a sheaf of rings of differential operators
on $X$ in the sense of \cite[\S2]{S:1}, hence the moduli space can be
constructed. (See also \cite{Ytwisted2}.)
For a later purpose, we review the argument briefly.
The moduli space is a GIT quotient of the scheme $Q$ parameterizing
all quotients $\left[V\otimes_\C\mA(-m) \twoheadrightarrow
  \overline{F}\right]$ of $\mA$-modules by $\operatorname{SL}(V)$ for
the vector space $V = \Hom_\mA(\mA,\overline{F}(m))
\begin{NB}
= \Hom_{\perv} (\mca{P}, F(m))  
\end{NB}
$ for a fixed sufficiently large $m$.
\begin{NB}
  In $\perv$, $Q$ is considered as the scheme parameterizing quotients
\(
   V\otimes_\C \mca{P}(-m)
   \twoheadrightarrow F.
\)
\end{NB}%
The scheme $Q$ is a closed subscheme of the usual quot-scheme
parameterizing quotients in $\mr{Coh}(X)$.
\begin{NB}
We have a natural homomorphism
\(
   \shfO_X \to \mA
\)
induced from
\(
   \shfO_Y \to \mca{P}^\vee\otimes_{\shfO_Y}\mca{P}.
\)
Thus we have a homomorphism
\(
   \Hom_\mA(\mA,\overline{F}(m))
\to
   \Hom_{\shfO_X}(\shfO_X,\overline{F}(m)),
\)
which is clearly injective. The condition is closed.
\end{NB}%
The polarization of $Q$ comes from the embedding into the Grassmann
variety of quotients
\(
  \left[H^0(V\otimes_\C \mA(l-m)) \twoheadrightarrow 
     H^0(\overline{F}(l))\right]
\)
for sufficiently large $l$.
In $\perv$, the Grassmann variety parameterizes quotients
\(
  \left[V\otimes_\C H^0(\mca{P}^\vee\otimes \mca{P}(l-m))
    \twoheadrightarrow 
     H^0(\mca{P}^\vee\otimes {F}(l))\right]
\)
under the equivalence $\overline{F} = \R f_*(\mca{P}^\vee\otimes F)$.

We next generalize the stability condition slightly. Suppose
$\mca{P}$, $\mca{P}'$ are local projective generators. We say
$F\in\perv$ is {\it $(\mca{P},\mca{P}')$-twisted (semi)stable\/} if
\begin{equation*}
  \chi(\mca{P},F'(n)) (\le) 
  \frac{\chi(\mca{P}',F'(m))}{\chi(\mca{P}',F(m))}
  \chi(\mca{P},F(n)) \qquad \text{for $m\gg n\gg 0$}
\end{equation*}
holds for any proper subobject $0\neq F'\subsetneq F$ in $\perv$.
If $\mca{P} = \mca{P}'$, $(\mca{P},\mca{P}')$-twisted (semi)stability
is equivalent to the above $\mca{P}$-twisted stability.
\begin{NB}
  This is natural from the construction of moduli spaces: first take
  $m$ large to consider the quot-scheme, and then take $l\gg m$ to
  define a polarization. But it also follows from the direct argument:
  Suppose $F$ is $(\mca{P},\mca{P})$-twisted semistable. Considering
  the leading coefficients with respect to $m$, we find
  \begin{equation*}
      \chi(\mca{P},F'(n)) \le
  \frac{a_d^{\mca{P}}(F')}{a_d^{\mca{P}}(F)}
  \chi(\mca{P},F(n))
  \end{equation*}
  for large $n$. Therefore $F$ is $\mca{P}$-twisted semistable.
  Moreover the equality holds here if and only if the equality holds
  in the inequality in the definition of the
  $(\mca{P},\mca{P})$-twisted semistability. Therefore the converse
  and the equivalence on twisted stabilities follow.
\end{NB}%
The moduli space of $(\mca{P},\mca{P}')$-twisted semistable sheaves
can be constructed as above. In fact, the scheme $Q$ for which we take a GIT
quotient is the same as above, but we use the different polarization
from the embedding into the Grassmann variety using $\mca{P}'$ instead
of $\mca{P}$, i.e., quotients
\(
  \left[V\otimes_\C H^0(\mca{P}^{\prime\vee}\otimes \mca{P}(l-m))
    \twoheadrightarrow 
     H^0(\mca{P}^{\prime\vee}\otimes {F}(l))\right]
\)
for sufficiently large $l$.
This modification is very similar to (in fact, simpler than) the stability
condition considered in \cite[Sect.~2]{ny-perv2}.

\begin{NB}
Let us show that the set of type $\lambda$ sheaves with respect to the
$(\mca{P},\mca{P'})$-twisted semistability with fixed 
$\mca{P}$ and $\mca{P}'$-twisted Hilbert polynomials is bounded.

We have (\cite[(4.5)]{Ytwisted2})
\begin{equation*}
   a_d^{\mca{P}}(F) = \operatorname{rank}\mca{P} \cdot
   a_d^{\mca{O}_Y}(F),
\qquad
   a_{d-1}^{\mca{P}}(F) = \operatorname{rank}\mca{P} \cdot
   (a_{d-1}^{\mca{O}_Y}(F) - a_d^{\mca{O}_Y}(F)) + a_{d-1}^{\mca{O}_Y}(F|_D),
\end{equation*}
where $D$ is an effective divisor taken as in \cite[one line above of
(4.5)]{Ytwisted2}.
Therefore
\(
   \frac{a_d^{\mca{P}'}(F)}{a_d^{\mca{P}}(F)}
\)
is the constant $\operatorname{rank}P'/\operatorname{rank}P$
independent of $F$ and
\begin{equation*}
   \frac{a_{d-1}^{\mca{P}'}(F)}{a_d^{\mca{P}}(F)}
   = \frac{\operatorname{rank}P'}{\operatorname{rank}P}
   \left(
   \frac{a_{d-1}^{\mca{O}_Y}(F)}{a_d^{\mca{O}_Y}(F)} - 1
   \right)
   + \frac{a_{d-1}^{\mca{O}_Y}(F|_D)}{\operatorname{rank}P\cdot a_d^{\mca{O}_Y}(F)}.
\end{equation*}
The remaining argument is the same as in \cite[Prop.~4.2]{Ytwisted2}.
\end{NB}

We can also construct moduli spaces of perverse coherent systems. 
Let $\alpha$ be a polynomial of rational coefficients such
that $\alpha(n) > 0$ for $n\gg 0$. A perverse coherent system
$(F,W,s)$ with a finite dimensional framing $W$ is
{\it $(\mca{P},\mca{P}',\alpha)$-(semi)stable\/} if
\begin{multline}\label{eq-twisted-stability}
  \dim W' \cdot \alpha(n) + \chi(\mca{P},F'(n)) (\le) 
  \frac{\chi(\mca{P}',F'(m))}{\chi(\mca{P}',F(m))}
  \left(\dim W  \cdot \alpha(n) + \chi(\mca{P},F(n))\right)
\\
  \text{for $m\gg n\gg 0$}
\end{multline}
holds for any proper subobject $0\neq (F',W',s')\subsetneq (F,W,s)$ in
$\tperv$. If the homomorphism $s\colon W\to \Hom_Y(\shfO_Y,F)$ has
nontrivial kernel, $(F',W',s') = (0,\ker s,0)$ violates the
inequality. Therefore $W$ can be considered as a subspace of
$\Hom_Y(\shfO_Y,F)$.  We can further consider $W$ as a subspace of
$\Hom_\mA(\mA,\overline{F}(m))\otimes \Hom_\mA(\mA,\mA(m))^\vee$ for
sufficiently large $m$ thanks to the projection $\mca{P}\to\shfO_Y$.
\begin{NB}
  Since $\mca{P} = \shfO_Y\oplus\mca{P}_0$, we have
  $\Hom_Y(\shfO_Y,F)\subset \Hom_{\perv}(\mca{P},F)
  \cong \Hom_\mA(\mA, \overline{F})$. Then we further compose
  $\Hom_\mA(\mA, \overline{F}) \subset
  \Hom_\mA(\mA, \overline{F}(m))\otimes \Hom_\mA(\mA,\mA(m))^\vee$.
\end{NB}
Now we can construct the moduli space of 
$(\mca{P},\mca{P}',\alpha)$-semistable perverse coherent systems as a
GIT quotient of a closed subscheme of the product of the quot-scheme
and the Grassmann variety as in \cite{He,LePotier}.
In summary, we have the following theorem:

\begin{thm}\label{thm-construction-2}
Under the assumption that $X$ is projective, there is a coarse moduli scheme parameterizing $S$-equivalence classes of $(\mca{P},\mca{P}',\alpha)$-semistable 
perverse coherent systems with a finite dimensional framing.
\end{thm}

Now, we will explain how Theorem \ref{thm-construction-1} is deduced from Theorem \ref{thm-construction-2}.

First, we replace $X$ by a projective scheme $\overline{X}$ containing $X$ as an open subscheme and construct a moduli space for $\overline{X}$.
Then the moduli space for $X$ is an open subscheme of the moduli space for $\overline{X}$
(here we use the assumption that the objects are in $\tpervc$, not just in $\tperv$).
Therefore we may assume $X$ is projective from the beginning.  

Note that for an object $F\in \pervc$, the Hilbert polynomials are constant. 
We may assume that $\zeta_0$, $\zeta_1$ and $\zeta_\infty$ are integers and normalized so that $\theta_{\tz}(\F) = 0$ as we mentioned in Remark~\ref{rem:normalization} (2). 
Taking sufficiently large $c\in\R$ so that $\zeta_0+c$, $\zeta_1 + c
> 0$, we put $\mca{P} = \shfO_Y^{\oplus (c+ \zeta_0)} \oplus \mca{P}_0^{\oplus (c + \zeta_1)}$, $\mca{P}' =\shfO_Y\oplus \mca{P}_0$ and $\alpha = \zeta_\infty$.
Then the inequality \eqref{eq-twisted-stability} turns out to be
\begin{multline*}
  \frac{ \zeta_0 \dim H^0(F') + \zeta_1 \dim H^0(F'\otimes
    \mca{P}_0^\vee) + \zeta_\infty \dim W'}
    {\dim H^0(F') + \dim H^0(F'\otimes \mca{P}_0^\vee)}
\\
    (\le) 
  \frac{ \zeta_0 \dim H^0(F) + \zeta_1 \dim H^0(F\otimes
    \mca{P}_0^\vee) + \zeta_\infty \dim W}
    {\dim H^0(F) + \dim H^0(F\otimes \mca{P}_0^\vee)} = 0,
\end{multline*}
where $c$ cancels out in both hand sides. This is equivalent to the
inequality in Definition~\ref{sfam}.
Therefore we can apply the above construction of the moduli
space provided $\zeta_\infty = -\zeta_0 \dim H^0(F) - \zeta_1 \dim
H^0(F\otimes \mca{P}_0^\vee) > 0$. Fortunately this condition is not
restrictive, as there is no $\theta_{\tz}$-stable object (with $W\neq
0$) except $F = 0$ if $\zeta_\infty \le 0$.
\begin{NB}
Consider the subobject $(F,0,0)\subsetneq (F,W,s)$.

Even if $(F,W,s)$ are merely $\theta_{\tz}$-semistable, we can replace
it by $(F,0,0)\oplus (0,W,0)$ in its $S$-equivalence class. Therefore
there is no problem to construct the moduli space.
\end{NB}

\section{Coherent systems as \pc systems}\label{cs-pcs}
\subsection{Chambers corresponding to DT and PT}
\begin{NB} This subsection is rewritten (6/19).\end{NB}%
Fix an ample line bundle $\mca{L}$ and a vector bundle $\mca{P}_0$ on
$Y$ as in \S \ref{framednccr}.
\begin{NB}
Please change $\mca{P}$ to $\mca{P}_0$ after this point. July 31, HN.
\end{NB}%
Take $\mca{L}$ as a polarization of $Y$. 
\begin{NB}
  Since I have used $\shfO_Y(1)$ for different thing (i.e.,
  $f^*\shfO_X(1)$), please delete this and replace any $\shfO_Y(1)$ to
  $\mca{L}$ hereafter. July 31, HN.
\end{NB}
For an element $E\in D^b_c(\coh)$, let $r(E)$ denote the degree one coefficient of the Hilbert polynomial $\chi(E\otimes \mca{L}^{\otimes k})$.
\begin{NB}
  Please change the notation $\mr{rk}(E)$ to $r(E)$ as `$\mr{rk}(E)$' is
  usually the rank of a torsion free sheaf. July 31, HN.
\end{NB}

\begin{defn}\label{def-st}
Let $(\zeta_0,\zeta_1)$ be a pair of real numbers.
A \pc system $(F,s)\in \tpervc$ with a $1$-dimensional framing 
is said to be $(\zeta_0,\zeta_1)$-(semi)stable if it is $\theta_{\tz}$-(semi)stable for
\[
\tz=(\zeta_0, \zeta_1,-\zeta_0\cdot\dim H^0(E)-\zeta_1\cdot\dim H^0(E\otimes \mca{P}_0^\vee)).
\]
\textup{(See the normalization in Remark~\ref{rem:normalization}(2).)}
\end{defn}
\begin{NB}
I exchanged the definition and the lemma (remark). (7/5).
\end{NB}
\begin{lem}\label{spcs}
A \pc system $(F,s)\in \tpervc$ is $(\zeta_0,\zeta_1)$-(semi)stable if
the following conditions are satisfied:
\begin{enumerate}
\item[\textup{(A)}] for any nonzero subobject $0\neq E\subseteq F$, we have
\[
\zeta_0\cdot\dim H^0(E)+\zeta_1\cdot\dim H^0(E\otimes \mca{P}_0^\vee)\,(\le)\,0,
\]
\item[\textup{(B)}] for any proper subobject $E\subsetneq F$ through which $s$ factors, we have
\[
\zeta_0\cdot\dim H^0(E)+\zeta_1\cdot\dim H^0(E\otimes \mca{P}_0^\vee)\,(\le)\,\zeta_0\cdot\dim H^0(F)+\zeta_1\cdot\dim H^0(F\otimes \mca{P}_0^\vee).
\]
\end{enumerate}
\end{lem}

Let $r$ be the positive integer in \eqref{eq:ext}.
\begin{NB}
Original : July 31, HN
appeared when we constructed $\mca{P}_0$ in \S \ref{framednccr}. 
\end{NB}
\begin{lem}\label{lem2.3}
For any $E\in\pervc$, we have
\[
r\cdot \dim H^0(Y,E)- \dim H^0(Y,E\otimes \mca{P}_0^\vee)=\rk(E).
\]
\end{lem}
\begin{proof}
Since the Hilbert polynomial $\chi(E\otimes\mca{L}^{\otimes k})$ is degree one, we have $r(E)=\chi(E)-\chi(E\otimes \mca{L}^{-1})$.
\begin{NB}
This equality is the `trivial consequence' of the fact that
$\chi(E(k))$ is linear in $k$.  
\end{NB}
On the other hand, since $E\in\pervc$ we have
\[
H^0(E)=\chi(E),\quad
H^0(E\otimes \mca{P}_0^\vee)=\chi(E\otimes \mca{P}_0^\vee)=(r-1)\chi(E)+\chi(E\otimes \mca{L}^{-1}).
\]
Hence the claim follows.
\end{proof}

We set $\zeta^\circ=(-r,1)$ and $\zeta^\pm=(-r\pm\varepsilon,1)$ for
sufficiently small $\varepsilon>0$ specified later.
\begin{NB} changed 7/5 \end{NB}%
\begin{NB}
  For the resolved conifold, we have $r=1$. Therefore $\zeta^\circ$ is
  on the imaginary root hyperplane. (In particular, the definition of
  $r$ must be the same as in \cite{vandenbergh-3d}.
\end{NB}%
\begin{figure}[htbp]\label{parameter}
  \centering
  \includegraphics{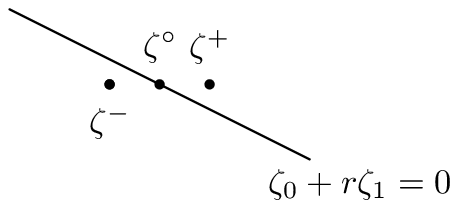}
  \caption{stability parameters}
\end{figure}

\begin{cor}\label{cor2.4}
\begin{NB} changed 7/5 \end{NB}%
A \pc system $(F,s)\in\tpervc$ is $\zeta^\circ$-(semi)stable (resp.\ $\zeta^\pm$-(semi)stable) if and only if the following conditions are satisfied:
\begin{enumerate}
\item[\textup{(A)}] for any nonzero subobject $0\neq E\subsetneq F$, we have 
\[
-\rk (E)\,(\le)\,0\quad (\text{resp.}\ -\rk (E)\pm\varepsilon\cdot\chi(E)\,(\le)\,0),
\]
\item[\textup{(B)}] for any nonzero subsheaf $0\neq E\subsetneq F$ through which $s$ factors, we have
\[
0\,(\le)-\rk (F/E)\quad (\text{resp.}\ 0\,(\le)-\rk (F/E)\pm\varepsilon\cdot\chi(F/E)).
\]
\end{enumerate}
\end{cor}

\begin{lem}\label{lem1}
Let $0\to E\to F\to G\to 0$ be an exact sequence in $\coh$. 
If $F$ is perverse coherent, then so is $G$. 
\end{lem}
\begin{proof}

Applying the functor $f_*$ to the short exact sequence, we get the exact sequence
\[
\R^1f_*F\to\R^1f_*G\to\R^2f_*E.
\]
Since $F$ is perverse coherent we have $\R^1f_*F=0$. 
Because dimensions of the fibers of $f\colon Y\to X$ are less than $2$, we have $\R^2f_*E=0$.
Hence we get $\R^1f_*G=0$.  
\begin{NB}
For ${}^0\perv$ we should also check the following:

Let $C$ be a sheaf satisfying $\R f_*(C)=0$. 
Applying the functor $\Hom(-,C)$ to the short exact sequence, we get an exact sequence
\[
0\to\Hom(G,C)\to\Hom(F,C). 
\]
Since $F$ is perverse coherent we have $\Hom(F,C)=0$. 
So we get $\Hom(G,C)=0$.
\end{NB}
\end{proof}
\begin{cor}\label{cor1}
Let $0\to E\to F\to G\to 0$ be an exact sequence in $\coh$. If $E$ and $F$ are perverse coherent, then this is an exact sequence in $\perv$ as well.
\end{cor}

\begin{lem}\label{cokerisperv}
Let $(F,s)$ be a perverse coherent system such that $F$ is a sheaf. 
Then the sequence
\[
0 \to \im_{\coh}(s)  \to  F  \to  \coker_{\coh}(s)  \to  0 
\]
is exact in $\perv$.
\end{lem}
\begin{proof}
By Lemma \ref{lem1} applied to $\mathcal{O}_Y\to\im_{\coh}(s)$, we see that $\im_{\coh}(s)$ is perverse coherent. 
Then the claim follows from Corollary \ref{cor1}.
\end{proof}

\begin{lem}\label{fispc}
Let $(F,s)$ be a coherent system such that $\dim\coker_{\coh}(s)=0$. 
\begin{NB} corrected 7/5 \end{NB}%
Then $F$ is perverse coherent.
\end{lem}
\begin{proof}
We have the exact sequences 
\[
\begin{array}{ccccccccc}
0 & \to & \ker(s) & \to & \OO_Y & \to & \im(s) & \to & 0 \\
0 & \to & \im(s) & \to & F & \to & \coker(s) & \to & 0 
\end{array}
\]
in $\coh$. 
Applying Lemma \ref{lem1} for the first exact sequence, we have $\im(s)$ is perverse coherent. 
Since $\coker(s)$ is $0$-dimensional, it is perverse coherent.
Hence $F$ is \pc by the second exact sequence.
\end{proof}

\begin{lem}\label{fissheaf}
If a \pc system $(F,s)\in\tpervc$ is $\zeta^\circ$-semistable, 
\begin{NB} changed 7/5 \end{NB}%
then $F$ is a sheaf.
\end{lem}
\begin{proof}
We have the canonical exact sequence
\[
0\to H^{-1}(F)[1]\to F\to H^0(F)\to 0
\]
in $\perv$.  
By the condition (A) in Corollary \ref{cor2.4}
\begin{NB} corrected 7/5 \end{NB}%
, we have $\rk(H^{-1}(F))=-\rk(H^{-1}(F)[1])\leq 0$.
Since $H^{-1}(F)$ is a sheaf, this inequality means that $H^{-1}(F)$ is $0$-dimensional. 
The defining condition of \pc sheaves requires $\R^0f_*(H^{-1}(F))=0$.
So we have $H^{-1}(F)=0$.
\end{proof}

\begin{NB} I omitted this lemma since it is not clear what "sufficiently small" means. 7/5
\begin{lem}\label{dtlem}
For a nonzero coherent sheaf $F\in\cohc$ we have
\[
-\rk (E)-\varepsilon\cdot\chi(E)<0.
\] 
\end{lem}
\begin{proof}
If $\rk(E)>0$, the equation holds since $\varepsilon$ is sufficiently small.
Otherwise, $E$ is $0$-dimensional and so $\chi(E)>0$. Thus the equation holds as well. 
\end{proof}
\end{NB}
Fixing the numerical class of $F\in D^b_c(\coh)$, 
we take sufficiently small $\varepsilon>0$ so that $\varepsilon\cdot\chi(F)<1$. \begin{NB} added 7/5 \end{NB}
\begin{prop}\label{dt-ncdt}
Given a $\zeta^-$-stable \pc system $(F,s)\in\tpervc$, then $F$ is a sheaf and $s$ is surjective in $\coh$.
On the other hand, given a coherent sheaf $F\in\cohc$ and a surjection $s\colon \OO_Y\to F$ in $\coh$, then $F$ is \pc and the \pc system $(F,s)$ is $\zeta^-$-stable.
\end{prop}
\begin{proof}
\begin{NB}
I modified, so as not to use above lemma. 7/5
\end{NB}
Assume that $(F,s)\in\tpervc$ is $\zeta^-$-stable. 
By Lemma \ref{fissheaf}, $F$ is a sheaf.   
By Lemma \ref{lem1}, $\coker_{\coh}(s)$ is perverse coherent. 
Assume $s$ is not surjective in $\coh$. 
Since $\coker_{\coh}(s)$ is a sheaf, we have $r(\coker_{\coh}(s))\geq 0$.
Since $\coker_{\coh}(s)$ is perverse coherent, we have $\chi(\coker_{\coh}(s))\geq 0$. 
Hence we have 
\[
0> -r(\coker_{\coh}(s))-\varepsilon\cdot\chi(\coker_{\coh}(s)).
\]
By Lemma \ref{cokerisperv}, this contradicts the condition (B) of $\zeta^-$-stability of $(F,s)$. 
So $s$ is surjective in $\coh$.

On the other hand, assume that $s$ is surjective in $\coh$.
Let 
\[
0\to E\to F\to G\to 0
\]
be an exact sequence in $\perv$. 
Since $H^{-2}(G)=H^{-1}(F)=0$, so we have $H^{-1}(E)=0$ by the long exact sequence.
Thus $E$ is a sheaf, and so we have $\rk (E)\geq 0$. 
Since $E$ is \pc we have $\chi(E)\geq 0$. 
So the condition (A) holds.
Moreover, assume that $s$ factors through $E$. Then $E\to F$ is surjective in $\coh$ and so $G$ is a sheaf shifted by $[1]$.  
Suppose $r(G)\leq -1$. Since $F$ is perverse coherent, we have
\[
\varepsilon\cdot\chi(G)=\varepsilon(\chi(F)-\chi(E))\leq\varepsilon\cdot\chi(F)<1.
\]
So the condition (B) holds.
Suppose $r(G)=0$. 
Then $\chi(G)<0$   
and the condition (B) holds.
\end{proof}

\begin{prop}\label{pt-ncdt}
Given a $\zeta^+$-stable \pc system $(F,s)\in\tpervc$, then $F$ is a sheaf and $(F,s)$ is a stable pair in the sense of \cite{pt1}, that is, the following conditions are satisfied:
\begin{enumerate}
\item $F$ is pure of dimension $1$, and 
\item $\coker_{\coh}(s)$ is $0$-dimensional. 
\end{enumerate}
On the other hand, given a stable pair $(F,s)$, then $F$ is \pc and the  \pc system $(F,s)$ is $\zeta^+$-stable.
\end{prop}
\begin{proof}
\begin{NB}
I modified, so as not to use above lemma. 7/5
\end{NB}
Assume that $(F,s)\in\tpervc$ is $\zeta^+$-stable. 
Let 
\[
0\to E\to F\to G\to 0
\]
be an exact sequence in $\coh$. 
Suppose that $E$ is $0$-dimensional. 
Then $E$ is \pc and by Corollary \ref{cor1} $E$ is a subobject of $F$ in $\perv$ as well. 
We have $\rk (E)=0$ and $\chi(E)>0$ because $E$ is $0$-dimensional.  
This contradicts with the inequality
\[
-\rk (E)+\varepsilon\cdot\chi(E)<0
\] 
in the condition (A) of $\zeta^+$-stability.
So $F$ is pure of dimension $1$.
By Lemma \ref{cokerisperv} and the condition (B) of $\zeta^+$-stability, we have  
\[
0<-r(\coker_{\coh}(s))+\varepsilon\cdot\chi(\coker_{\coh}(s)),
\]
unless $\coker_{\coh}(s)=0$. 
Since $\im_{\coh}(s)$ is perverse coherent, we have
\[
\varepsilon\chi(\coker_{\coh}(s))=\varepsilon(\chi(F)-\chi(\im_{\coh}(s))\leq\varepsilon\cdot\chi(F)<1.
\]
So $\rk (\coker_{\coh}(s))$ can not be positive, that is, $\coker_{\coh}(s)$ is $0$-dimensional.

Assume that $(F,s)$ is a stable pair.
Let 
\[
0\to E\to F\to G\to 0
\]
be an exact sequence in $\perv$. 
As in the proof of Proposition \ref{dt-ncdt}, $E$ is a sheaf. 
We have the exact sequences
\[
\begin{array}{ccccccccc}
0 & \to & H^{-1}(G) & \to & E & \to & \im_{\coh}(E\to F) & \to & 0 \\
0 & \to & \im_{\coh}(E\to F) & \to & F & \to & H^0(G) & \to & 0 
\end{array}
\]
in $\coh$. 
Since $F$ is pure of dimension $1$, $\im_{\coh}(E\to F)$ is $1$-dimensional unless it is zero. 
As in the proof of Lemma \ref{fissheaf}, $G$ is $1$-dimensional unless $G=0$.
Hence we have $r(E)\geq 1$ unless $E=0$. 
Because $G$ is perverse coherent, we have
$\chi(E)=\chi(F)-\chi(G)\leq\chi(F)$.
So the condition (A) is satisfied.

Moreover, assume that $s$ factors through $E$. 
Then $\im_{\coh}(E\to F)\supset \im_{\coh}(s)$ in $\coh$ and so 
\[
\coker_{\coh}(s)\twoheadrightarrow \coker_{\coh}(E\to F)=H^0(G),
\]
in $\coh$. 
In particular, $H^0(G)$ is $0$-dimensional. 
By an argument as in the proof of Proposition \ref{fissheaf}, $H^{-1}(G)$ is $1$-dimensional unless $H^{-1}(G)=0$. 
Because $E$ is perverse coherent, we have
$\chi(G)=\chi(F)-\chi(E)\leq\chi(F)$.
Hence we have
\[
-r(G)+\varepsilon\cdot\chi(G)=-r(H^0(G))+r(H^{-1}(G))+\varepsilon\cdot\chi(G)>0.
\]
\end{proof}

\subsection{Coherent systems on the flop}\label{flop}
\begin{NB} This subsection is added (6/19).\end{NB}%
Assume further $f$ is isomorphic in codimension $1$.
Let $f^+\colon Y^+\to X$ be the flop. 
Let ${}^{0}\fperv$ be the full subcategory of $D^b(\fcoh)$ consisting of objects $E$ satisfying the following conditions: 
\begin{itemize}
\item $H^i(E)=0$ unless $i=0,-1$,
\item $\R^1f_*(H^0(E))=0$ and $\R^0f_*(H^{-1}(E))=0$,
\item $\Hom(H^0(E),C)=0$ for any sheaf $C$ on $Y$ satisfying $\R f_*(C)=0$.
\end{itemize}
We can associate $\mca{L}$ and $\mca{P}_0$ with an ample line bundle $\mca{L}^+$ and a vector bundle $\mca{Q}^+_0$ on $Y^+$ such that
\begin{itemize}
\item they are involved in an exact sequence 
\[
0\to (\mca{L}^+)^{-1}\to \mca{Q}^+_0 \to \OO_{Y^+}^{\oplus r-1}\to 0,
\]
where $r$ coincides with what appeared in the defining sequence of $\mca{P}_0$,  
\item $\mca{Q}^+_0$ is a local projective generator of ${}^0\fperv$, 
\item $f^+_*\gend(\OO_{Y^+}\oplus\mca{Q}^+_0)=\mA$,  
\end{itemize}
and hence we have the following equivalences:
\[
\begin{array}{ccccc}
D^b(\coh) & \simeq & D^ b(\mamod) & \simeq & D^b(\fcoh)\\
\cup & & \cup & & \cup\\
{}^{-1}\perv & \simeq & \mamod & \simeq & {}^{0}\fperv,
\end{array}
\]
(see \cite[Theorem 4.4.2]{vandenbergh-3d}). 
Here we denote by ${}^{-1}\perv$ what we have denoted simply by $\perv$ to emphasize the difference between ${}^{-1}\perv$ and ${}^0\fperv$.
By the same argument as the previous subsection, we can verify the following propositions:

\begin{prop}\label{fdt-ncdt}
Given a $(-\zeta^+)$-stable \pc system $(F^+,s^+)\in\ftpervc$, then $F^+$ is a sheaf and $s^+$ is surjective in $\fcoh$.
On the other hand, given a coherent sheaf $F^+\in\fcohc$ and a surjection $s^+\colon \OO_{Y^+}\to F^+$ in $\fcoh$, then $F^+$ is \pc and the \pc system $(F^+,s^+)$ is $(-\zeta^+)$-stable.
\end{prop}

\begin{prop}\label{fpt-ncdt}
Given a $(-\zeta^-)$-stable \pc system $(F^+,s^+)\in\ftpervc$, then $F^+$ is a sheaf and $(F^+,s^+)$ is a stable pair.
On the other hand, given a stable pair $(F^+,s^+)$ on $Y^+$, then $F^+$ is \pc and the  \pc system $(F^+,s^+)$ is $(-\zeta^-)$-stable.\begin{NB} corrected 7/7 \end{NB}
\end{prop}

\section{Counting invariants on the resolved conifold}
In this section, we study the counting invariants on the resolved conifold.

Let $f\colon Y\to X$ be the crepant resolution of the conifold, that is, $X=\{(x,y,z,w)\in \C^4\mid xy-zw=0\}$ and $Y$ is the total space of the vector bundle $\pi\colon \OO_{\CP^1}(-1)\oplus\OO_{\CP^1}(-1)\to\CP^1$. 
This satisfies the assumptions at the beginning of \S \ref{sec1.1}.

\subsection{Quivers for the resolved conifold}\label{subsec-quiver-for-conifold}
Let $Q$ be the quiver in Figure \ref{quiver} and $A$ be the algebra defined by the following quiver with the relations:
\[
A:= \C Q/(a_1b_ia_2=a_2b_ia_1,b_1a_ib_2=b_2a_ib_1)_{i=1,2}.
\]
\begin{figure}[htbp]
  \centering
  \includegraphics{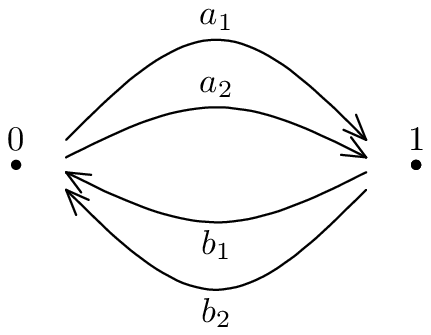}
  \caption{quiver $Q$}\label{quiver}
\end{figure}
\begin{rem}
Note that this relation is derived from the superpotential $\omega=a_1b_1a_2b_2-a_1b_2a_2b_1$ \textup{(\cite{berenstein-douglas} for example. See also \cite{quiver-with-potentials})}. 
\end{rem}
Let $\amod$ (resp. $\afmod$) denote the category of right $A$-modules (resp. finite dimensional right $A$-modules). 
For a finite dimensional $A$-module $V$, let $V_0$ and $V_1$ denote the vector spaces corresponding to the vertices $0$ and $1$ and $\dimv V:=(\dim V_0,\dim V_1)\in(\Z_{\geq 0})^2$.
\begin{NB4}
the description on $\mathrm{Mod}(A)$ is added
\end{NB4}

Let us take $\mca{L}:=\pi^*\OO_{\CP^1}(1)$ as an ample line bundle.  
Since $H^1(Y,\mca{L}^{-1})=0$, we have a projective generator $\mca{P} := \OO_Y\oplus \mca{L}$. 
The endomorphism algebra $\End_Y(\mca{P})$ is isomorphic to $A$ and we have the following equivalence:
\begin{prop}\textup{(see Theorem \ref{cor:Morita})}
\[
\perv\simeq \amod,\quad
\pervc\simeq \afmod.
\]
\end{prop}
Under these equivalences, we have
\[
V_0=H^0(Y,F),\quad V_1=H^0(Y,F\otimes \mathcal{L}^{-1})
\]
for a perverse coherent sheaf $F$.

Moreover, let $\Q$ be the quiver in Figure \ref{newquiver} and $\A$ be the following quiver with the relations:
\[
\A:= \C \Q/(a_1b_ia_2=a_2b_ia_1,b_1a_ib_2=b_2a_ib_1)_{i=1,2}.
\]
\begin{figure}[htbp]
  \centering
  \includegraphics{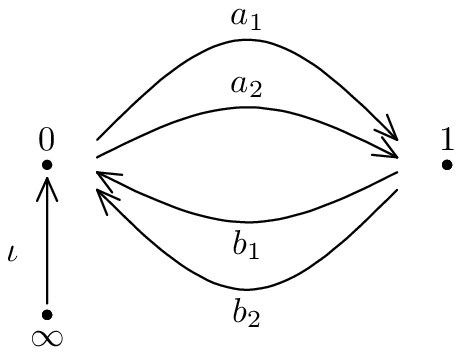}
  \caption{quiver $\Q$}\label{newquiver}
\end{figure}
Let $\hamod$ (resp. $\hafmod$) denote the category of $\A$-modules (resp. finite dimensional $\A$-modules), which is equivalent to the category $\hmamod$ (resp. $\hmacmod$) defined in \S \ref{framednccr}. 
\begin{NB4}
the description on $\mathrm{Mod}(\A)$ is added
\end{NB4}

For an $\A$-module $\tilde{V}$, let $V_0$, $V_1$ and $V_\infty$ denote the vector spaces corresponding to the vertices $0$, $1$ and $\infty$ and $\dimv \V:=(\dim V_0,\dim V_1,\dim V_\infty)\in(\Z_{\geq 0})^3$. 
\begin{prop}\textup{(see Proposition \ref{prop-framed-morita})}
\[
\tperv\simeq \hamod,\quad 
\tpervc\simeq \hafmod.
\]
\end{prop}
Under these equivalences, we have $V_\infty=W$.


\subsection{Definition of the counting invariants}\label{defofncdt}
\begin{NB} The definition of the moduli spaces should be written somewhere. In \S 1? (6/19) \end{NB}
For $\zeta\in\R^2$ and $\vv=(\mathrm{v}_0,\mathrm{v}_1)\in(\Z_{\geq 0})^2$, let $\M{\zeta}{\vv}$ (resp.\ $\mf{M}_{\zeta}^{\,\mr{s}}(\vv)$) denote the moduli space of $\zeta$-semistable (resp.\ $\zeta$-stable) $\A$-modules $\V$ with $\dimv \V=(\mathrm{v}_0,\mathrm{v}_1,1)$.
They can be constructed by applying the result of \cite{king}.
We define the generating function
\[
\mca{Z}'_{\zeta}(\q):=
\sum_{n\in\Z}\chi\left(\M{\zeta}{\vv}\right)\cdot\q^\vv
\]
where $\q^\vv=q_0^{\mathrm{v}_0}q_1^{\mathrm{v}_1}$ and $q_0$, $q_1$ are formal variables.

A $4$-dimensional torus $(\C^*)^4$ acts on the moduli space $\M{\zeta}{\vv}$ 
by rescaling the maps associated to the four arrows of the quiver $Q$. 
Since the subtorus
\[
\C^*\simeq 
\bigl\{[(\alpha,\alpha,\alpha^{-1},\alpha^{-1})]\in T\bigr\}
\]
acts trivially, we have the action of the $3$-dimensional torus $T:=(\C^*)^4/\C^*$.
We will show that
\begin{itemize}
\item
the set of $T$-fixed closed points $\M{\zeta}{\vv}^T$ is isolated (Proposition \ref{prop-parameterization}). 
\end{itemize}
Hence we have
\[
\mca{Z}'_{\zeta}(\q)=
\sum_{n\in\Z}\left|\,\M{\zeta}{\vv}^T\right|\cdot\q^\vv.
\]

We also define more sophisticated invariants. 
Let $\nu\colon\M{\zeta}{\vv}\to \Z$ be the constructible function defined in \cite{behrend-dt} (Behrend function).
We define the counting invariants 
\[
D_{\zeta}(\vv):=\sum_{n\in\Z}n\cdot\chi(\nu^{-1}(n))
\]
and encode them into the generating function
\[
\mca{Z}_{\zeta}(\q):=\sum_{\vv\in(\Z_{\geq 0})^2}D_{\zeta}(\vv)\cdot\q^\vv.
\]

The Behrend function is defined for any scheme over $\mathbb{C}$. 
In \cite{behrend-dt}, Behrend showed that if an proper scheme has a symmetric obstruction theory then the virtual counting, which is defined by integrating the constant function $1$ over the virtual fundamental cycle, coincides with the weighted Euler characteristic weighted by the Behrend function as above.
Based on this result, he proposed to define the virtual counting for a non-proper variety with a symmetric obstruction theory as the weighted Euler characteristic. 

A stability parameter $\zeta\in\R^2$ is said to be generic if $\zeta$-semistability and $\zeta$-stability are equivalent.
Since the defining relation of $A$ is derived from the derivations of the superpotential, the moduli space $\M{\zeta}{\vv}$ for a generic $\zeta$ has a symmetric obstruction theory (\cite[Theorem 1.3.1]{szendroi-ncdt}). 
\begin{NB}
We define the counting invariants 
\[
D_{\zeta}(\vv):=\sum_{n\in\Z}n\cdot\chi(\nu^{-1}(n))
\]
and encode them into the generating function
\[
\mca{Z}_{\zeta}(\q):=\sum_{\vv\in(\Z_{\geq 0})^2}D_{\zeta}(\vv)\cdot\q^\vv.
\]
\end{NB}
We define the $2$-dimensional subtorus
\[
T':=\bigl\{[(\alpha_1,\alpha_2,\beta_1,\beta_2)]\in T\mid \alpha_1\alpha_2\beta_1\beta_2=1\bigr\}
\]
of $T$. The symmetric obstruction theory above lifts to a $T'$-equivariant symmetric obstruction theory.
We will show the following propositions in \S \ref{appendix} (see \cite[Proposition 2.5.1 and Corollary 2.5.3]{szendroi-ncdt}):
\begin{itemize}
\item
$\M{\zeta}{\vv}^{T'}=\M{\zeta}{\vv}^{T}$ (Proposition \ref{prop-parameterization}).
\item
For each $T'$-fixed closed point $P\in\M{\zeta}{\vv}^{T'}$, the Zariski tangent space to $\M{\zeta}{\vv}$ at $P$ has no $T'$-invariant subspace (Proposition \ref{prop-isolated}).
\item
For each $T'$-fixed point $P\in\M{\zeta}{\vv}^{T'}$, the parity of the dimension of the Zariski tangent space to $\M{\zeta}{\vv}$ at $P$ is same as the parity of $\mathrm{v}_1$ (Corollary \ref{cor-zariski}).
\end{itemize}
According to these propositions and Behrend-Fantechi's result \cite[Theorem3.4]{behrend-fantechi}, we have the following formula (see \cite[Theorem 2.7.1]{szendroi-ncdt}):
\begin{equation}\label{eq-virtual-nonvirtual}
\mca{Z}_{\zeta}(\q)=\sum_{n\in\Z}(-1)^{\mathrm{v}_1}\left|\,\M{\zeta}{\vv}^T\right|\cdot\q^\vv \quad \text{(Theorem \ref{thm-sign})}.
\end{equation}
In particular, we have 
\[
\mca{Z}'_{\zeta}(\q)=\mca{Z}_{\zeta}(q_0,-q_1).
\]

\begin{NB3}
\begin{prop}[\protect{see \cite[Proposition 2.5.1]{szendroi-ncdt}}]
\begin{enumerate}
\item[(1)] (Proposition \ref{prop-parameterization})
The set of $T$-fixed closed points $\M{\zeta}{\vv}^T$ is isolated. 
\item[(2)] (Proposition \ref{prop-isolated})
For each $T$-fixed closed point $P\in\M{\zeta}{\vv}$, the Zariski tangent space to $\M{\zeta}{\vv}$ at $P$ has no $T$-invariant subspace.
\end{enumerate}
\end{prop}
\begin{prop}[\protect{see \cite[Corollary 2.5.3]{szendroi-ncdt}}]
For each $T$-fixed point $P\in\M{\zeta}{\vv}$, the parity of the Zariski tangent space to $\M{\zeta}{\vv}$ at $P$ is the same as the parity of $\mathrm{v}_1$.
\end{prop}
According to these propositions and Behrend-Fantechi's result \cite[Theorem3.4]{behrend-fantechi}, we have the following formula:
\begin{prop}[\protect{see \cite[Theorem 2.7.1]{szendroi-ncdt}}]
\begin{equation}\label{eq-virtual-nonvirtual}
\mca{Z}_{\zeta}(\q)=\sum_{n\in\Z}(-1)^{\mathrm{v}_1}\left|\,\M{\zeta}{\vv}^T\right|\cdot\q^\vv.
\end{equation}
\end{prop}
We also define the generating function
\[
\mca{Z}'_{\zeta}(\q):=\sum_{n\in\Z}\left|\,\M{\zeta}{\vv}^T\right|\cdot\q^\vv=\sum_{n\in\Z}\chi\left(\M{\zeta}{\vv}\right)\cdot\q^\vv.
\]
By the equation \eqref{eq-virtual-nonvirtual}, we have
\[
\mca{Z}'_{\zeta}(\q)=\mca{Z}_{\zeta}(q_0.-q_1).
\] 
\end{NB3}

\begin{NB3}
old version:

The $3$-dimensional torus $T$ acts on $Y$, and so on $\M{\zeta}{\vv}$ too.  
As we will see in \S \ref{appendix}, the set of $T$-fixed points $\M{\zeta}{\vv}^T$ is isolated. 
The argument of the proof of Theorem 2.7.1 in \cite{szendroi-ncdt} also works in our case and we have
\[
\mca{Z}_{\zeta}(\q)=\sum_{n\in\Z}(-1)^{\mathrm{v}_1}\left|\,\M{\zeta}{\vv}^T\right|\cdot\q^\vv.
\]
For simplicity, in this paper we also study
\[
\mca{Z}'_{\zeta}(\q)=\mca{Z}_{\zeta}(q_0,-q_1)=\sum_{n\in\Z}\left|\,\M{\zeta}{\vv}^T\right|\cdot\q^\vv=\sum_{n\in\Z}\chi\left(\M{\zeta}{\vv}\right)\cdot\q^\vv.
\]
\end{NB3}

\subsection{Classification of walls}\label{subsec-classification}
In this subsection, we will classify non-generic
parameters. The argument is a straightforward modification of one
in \cite[\S2]{ny-perv1}.

\begin{lem}
Let $W$ be a non-zero $\theta_\zeta$-stable $A$-module for some $\zeta\in\R^2$. Then at least one of the following
\begin{NB}
  Corrected, June 17
\end{NB}%
holds:

\[
(1)\ \dim W_0=\dim W_1 = 1
\begin{NB}
  Added, June 17
\end{NB}%
,\quad
(2)\ a_1=a_2=0,\quad
(3)\ b_1=b_2=0.
\]
\end{lem}

\begin{proof}\begin{NB}
  I changed \verb+\proof+ to \verb+\begin{proof}+ and \verb+\end{proof}+.
\end{NB}
(See \cite[Lemma 2.9]{ny-perv1}.)
Without loss of generality, we can assume $\zeta_0 \dim W_0 + \zeta_1
\dim W_1 = 0$ by Remark~\ref{rem:normalization}.
\begin{NB}
  Added, June 17.
\end{NB}%
If $W_0$ or $W_1 = 0$, we trivially have (2) or (3). Therefore we may
assume $W_0$, $W_1\neq 0$.

We set $S_0=\ker(a_1b_1)$, $T_0=\im(a_1b_1)$, $S_1=\ker(b_1a_1)$ and $T_1=\im(b_1a_1)$. By the defining relation of $Q$ we can check $(S_0,S_1)$ and $(T_0,T_1)$ are $A$-submodules of $W$. The $\theta_\zeta$-stability of $W$ implies 
\[
\zeta_0\dim S_0+\zeta_1\dim S_1\leq 0,\ \zeta_0\dim T_0+\zeta_1\dim T_1\leq 0.
\]
Since 
\[
\zeta_0\dim W_0+\zeta_1\dim W_1=0,\ \dim S_i+\dim T_i=\dim W_i,
\]
the above inequalities should be equalities. 
Again, the stability of $W$ implies $S_0=S_1=0$ or
$(S_0,S_1)=(W_0,W_1)$. 
\begin{NB}
  Corrected, June 17
\end{NB}
In the previous case, $a_1$ and $b_1$ are isomorphisms and $\dim W_0=\dim W_1$.

Taking arbitrary pairs $a_i$, $b_j$ ($i,j=1,2$), we may assume either
(a) $a_1$, $b_1$ are isomorphisms and $\dim W_0 = \dim W_1$, or
(b) $a_ib_j=0$, $b_ja_i=0$ for any $i$ and $j$. 
First we consider the case (b).
\begin{NB}
  Changed, June 17
\end{NB}%
Without loss of generality we also assume $\zeta_0\ge 0$. 
Apply the stability conditions for an $A$-submodule $(\ker a_1\cap\ker a_2,0)$, we get $\ker a_1\cap\ker a_2=0$.
The equations $a_ib_j=0$ mean $\im b_1, \im b_2\subset \ker a_1\cap\ker a_2=0$, that is $b_1=b_2=0$. 
Similarly, for the case $\zeta_0\le 0$ we have $a_1=a_2=0$.

Next consider the case (a), and hence $\zeta_0 + \zeta_1 = 0$. We
first assume $\zeta_0 < 0$.
From the defining equation of $Q$, four linear maps $b_1 a_1$, $b_2
a_1$, $b_1 a_2$, $b_2 a_2$ are pairwise commuting. We take a
simultaneous eigenvector $0\neq w_0 \in W_0$ and set
\begin{equation*}
   S_0' := \C w_0, \quad S_1' := \C a_1 w_0 + \C a_2 w_0.
\end{equation*}
Then $(S_0',S_1')$ is an $A$-submodule of $W$, and hence
\begin{equation*}
  \zeta_0\dim S_0'+\zeta_1\dim S_1'\leq 0
\end{equation*}
by the $\theta_\zeta$-semistability of $W$. Therefore we have
$\dim S_1' \le \dim S_0' = 1$. 
On the other hand, $a_1 w_0\neq 0$ as $a_1$ is an isomorphism by the
assumption. Therefore we have $\dim S'_1 = 1$, so the equality
holds in the above inequality, so we have $(S_0',S_1') = (W_0,W_1)$ by
the $\theta_\zeta$-stability of $W$. In particular, $\dim W_0 = \dim
W_1 = 1$. Exchanging $0$ and $1$, we have the same assertion when
$\zeta_1 < 0$.

Finally suppose $\zeta_0 = \zeta_1 = 0$. We define $S_0'$, $S_1'$ as
above. Since the equality
$\zeta_0\dim S_0'+\zeta_1\dim S_1' = 0$ holds, the
$\theta_\zeta$-stability of $W$ implies
$(S_0',S_1') = (W_0, W_1)$. In particular, $\dim W_0 = 1$. Exchanging
$0$ and $1$, we also get $\dim W_1 = 1$.
\end{proof}

\begin{NB}
The statement and proof of the converse are missing.
\end{NB}



In the case (1) with $\zeta_0 < 0$ (resp.\ $\zeta_0 > 0$), the
$\zeta$-stable $A$-modules are parameterized by $Y$ (resp.\
$Y^+$). \begin{NB} changed 6/19\end{NB}%
This is well-known, and can be checked easily
(cf.\ \cite[\S2.3]{ny-perv1}).

In the cases (2) or (3), the representation can be considered as a
representation of the Kronecker quiver. Then by the argument in
\cite[Lemma~2.12]{ny-perv1}, we have $(W_0, W_1) = (\C^m,\C^{m+1})$ or
$(W_0, W_1) = (\C^{m},\C^{m-1})$ \begin{NB} changed 6/19\end{NB}%
for some $m\ge 0$ or $m\ge 1$ in the latter case. Moreover, the
$\zeta$-stable $A$-module is unique up to isomorphism.
In the case (2), we denote the $\zeta$-stable $A$-module by $C^-_\pm(m)$ if $(W_0, W_1) = (\C^m,\C^{m\mp 1})$. 
Similarly, we denote the module by $C^+_\pm(m)$ in the case (3).
We can visualize these modules as in Figure~\ref{fig:C}. Each dot corresponds to a basis
vector of $W_0$ and $W_1$, and right-up (resp.\ right-down, left-up,
left-down) arrows are $a_1$ (resp.\ $a_2$, $b_1$, $b_2$).
\begin{NB}
  The figure may be moved to somewhere during the editing process. It
  is better to put it in \verb+\begin{figure}+, \verb+\end{figure}+.
\end{NB}
\begin{figure}[htbp]
  \centering
  \includegraphics{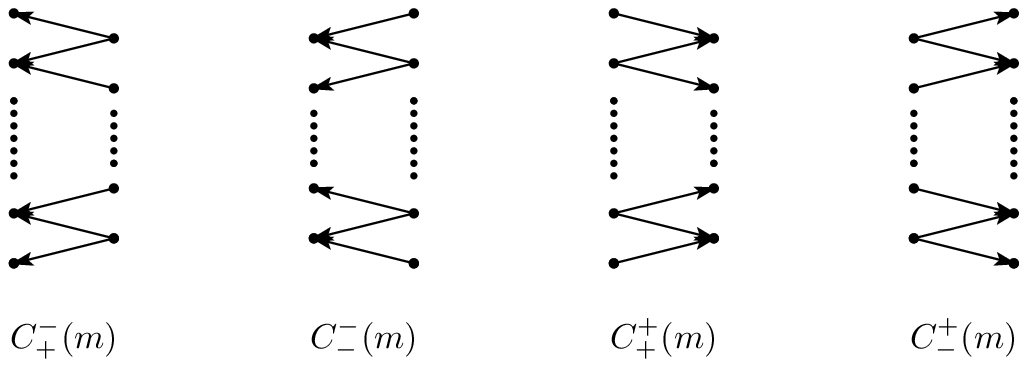}
  \caption{stable $A$-modules}
  \label{fig:C}
\end{figure}

Now, we can check the following classification:
\begin{thm}\label{thm-classification-1}
\begin{itemize}
Let $\zeta$ be a stability parameter.
\item[\textup{(1)}]
If $\zeta_0<\zeta_1$, 
then $\theta_\zeta$-stable $A$-modules $W$ are classified as follows:
\begin{itemize}
\item $C^-_+(m)$ ($m\geq 1$),
\item the $\zeta$-stable $A$-modules $W$ with $\dimv W=(1,1)$ parameterized by $Y$,
\item $C^-_-(m)$ ($m\geq 0$).
\end{itemize}
\item[\textup{(2)}]
If $\zeta_0>\zeta_1$, 
then $\theta_\zeta$-stable $A$-modules $W$ are classified as follows:
\begin{itemize}
\item $C^+_+(m)$ ($m\geq 1$),
\item the $\zeta$-stable $A$-modules $W$ with $\dimv W=(1,1)$ parameterized by $Y^+$,
\item $C^+_-(m)$ ($m\geq 0$).
\end{itemize}
\end{itemize}
\end{thm}
\begin{NB4}
If $\zeta_0=\zeta_1$, 
then $\theta_\zeta$-stable $A$-modules $W$ are classified as follows:
\begin{itemize}
\item simple modules $S_0$ and $S_1$,
\item $A$-modules $W$ with $\dimv W=(1,1)$ parameterized by $X^{\mathrm{reg}}$,
\end{itemize}
I think we don't have to write this, because this is complicated and useless.
\end{NB4}

Let $\V$ be an $\A$-module which is $\zeta$-semistable but not $\zeta$-stable. 
Let 
\[
\V=\V^0\supset \V^1\supset \cdots\supset \V^{L+1}=0
\]
be \JH of the $\theta_{\tz}$-semistable $\A$-module $\V$.
Here $L\geq 1$ because $\V$ is not $\theta_{\tz}$-stable.  
Since $\dim \V_\infty=1$, at most one of $(\V^l/\V^{l+1})_\infty$ is non-zero. In particular, there exists a non-zero $\theta_{\tz}$-stable $\A$-module $\tilde{W}$ such that $\tilde{W}_\infty=0$ and such that $\tz\cdot\dimv\,\tilde{W}=0$. In other words, there exists a non-zero $\theta_{\zeta}$-stable $A$-module $W$ such that $\zeta\cdot\dimv\,W=0$. 
We define the following walls (half lines) on the set of stability parameters:
\begin{align*}
L^-_+(m)&:=\{(\zeta_0,\zeta_1)\mid \zeta_0<\zeta_1,\ m\,\zeta_0+(m-1)\zeta_1=0\}\quad (m\geq 1),\\
L^-(\infty)&:= \{(\zeta_0,\zeta_1)\mid \zeta_0<\zeta_1,\ \zeta_0+\zeta_1=0\},\\
L^-_-(m)&:= \{(\zeta_0,\zeta_1)\mid \zeta_0<\zeta_1,\ m\,\zeta_0+(m+1)\zeta_1=0\}\quad (m\geq 0),\\
L^+_+(m)&:=\{(\zeta_0,\zeta_1)\mid \zeta_0>\zeta_1,\ m\,\zeta_0+(m-1)\zeta_1=0\}\quad (m\geq 1),\\
L^+(\infty)&:= \{(\zeta_0,\zeta_1)\mid \zeta_0>\zeta_1,\ \zeta_0+\zeta_1=0\},\\
L^+_-(m)&:=\{(\zeta_0,\zeta_1)\mid \zeta_0>\zeta_1,\ m\,\zeta_0+(m+1)\zeta_1=0\}\quad (m\geq 0).
\end{align*}
By Theorem \ref{thm-classification-1}, the set of non-generic stability parameters is the union of the origin $(0,0)$ and the walls above. 
\begin{NB3}
\begin{thm}\label{classification}
For each parameter $\zeta$, $\theta_\zeta$-stable $A$-modules $W$ such that $\zeta\cdot\dimv\,W=0$ are classified as follows: \begin{NB} added 7/7 \end{NB}
\begin{itemize}
\item For a parameter $\zeta$ on the wall $L^-_+(m)$
  \textup(resp.\ $L^-_-(m)$, $L^+_+(m)$, $L^+_-(m)$\textup) we have
  exactly one $\theta_\zeta$-stable $A$-module $C^-_+(m)$ \textup(resp.\
  $C^-_-(m)$, $C^+_+(m)$, $C^+_-(m)$\textup) up to isomorphisms.

\item For a parameter $\zeta$ on $L^-(\infty)$ \textup(resp.\
  $L^+(\infty)$\textup) the $\theta_\zeta$-stable $A$-modules are
  parameterized by $Y$ \textup(resp.\ $Y^+$\textup).

\item For $\zeta = 0$, the $\theta_\zeta$-stable $A$-modules are
  parameterized by $X\setminus \{0\}$ and simple modules $s_i$ ($i=0,1$).

\item Otherwise, there is no $\theta_\zeta$-stable $A$-modules.
\end{itemize}
\begin{NB}
  \textup{(1)} For a parameter $\zeta$ on the wall $L^-_+(m)$
  \textup(resp.\ $L^-_-(m)$, $L^+_+(m)$, $L^+_-(m)$\textup) we have
  exactly one $\theta_\zeta$-stable $A$-module $C^-_+(m)$ \textup(resp.\
  $C^-_-(m)$, $C^+_+(m)$, $C^+_-(m)$\textup) up to isomorphisms.

  \textup{(2)} For a parameter $\zeta$ on $L^-(\infty)$ \textup(resp.\
  $L^+(\infty)$\textup) the $\theta_\zeta$-stable $A$-modules are
  parameterized by $Y$ \textup(resp.\ $Y^+$\textup).

  \textup{(3)} For $\zeta = 0$, the $\theta_\zeta$-stable $A$-modules are
  parameterized by $X\setminus \{0\}$ and simple modules $s_i$ ($i=0,1$).

  \textup{(4)} Otherwise, there is no $\theta_\zeta$-stable $A$-modules.
\end{NB}
\begin{NB} we should not use "enumerate" or "itemize" ? 7/7\end{NB}%
\end{thm}
\end{NB3}


\begin{NB}
  Comment out 
  \verb+ \smallskip+

  We should not put any \verb+skip+ command, as it should be
  determined by the publisher.
\end{NB}

\begin{rem}\label{rem-classification}\begin{NB} corrected 7/7 \end{NB}%
Recall that the derived category of finite dimensional representations of Kronecker quiver is equivalent to the derived category of coherent sheaves on $\CP^1$. 
\begin{NB} the following is not correct. it depends on the choice of tilting bundle on $\CP^1$
Under this equivalence, $C^-_+(m)$ and $C^+_+(m)$ (resp.\ $C^-_-(m)$ and $C^+_-(m)$) correspond to $\OO_{\CP^1}(-m-1)[1]$ (resp.\ $\OO_{\CP^1}(m)$). 
\end{NB}

Under the equivalence $D^b(\amod)\simeq D^b(\coh)$, $C^-_+(m)$ and $C^-_-(m)$ correspond $z_*\OO_{\CP^1}(m-1)$ and $z_*\OO_{\CP^1}(-m-1)[1]$, where $z\colon \CP^1\to Y$ is the zero section.

Under the equivalence $D^b(\amod)\simeq D^b(\fcoh)$, $C^+_+(m)$ and $C^+_-(m)$ correspond $z^+_*\OO_{\CP^1}(-m-1)[1]$ and $z^+_*\OO_{\CP^1}(m-1)$, where $z^+\colon \CP^1\to Y^+$ is the zero section.

Moreover the stable
\begin{NB}
  Changed from simple to stable, June 17.
\end{NB}%
objects on the wall $L^-(\infty)$ correspond to
skyscraper sheaves on $Y$, ones on the wall $L^+(\infty)$ correspond
to skyscraper sheaves on $Y^+$.
\begin{NB}
I think `simple objects on the wall' does not make sense. I think you
mean `simple
objects in the category of semistable objects for the stability
parameter on the wall', but this is nothing but `stable objects'.
June 17
\end{NB}
\end{rem}

\subsection{Wall-crossing formula}
\begin{NB} This subsection is updated except for Proposition \ref{prop3.7} (6/19).\end{NB}%
Let $L$ be one of the walls $L^-_+(m)$, $L^-_-(m)$, $L^+_+(m)$ or $L^+_-(m)$. 
Take a parameter $\zeta^\circ=(\zeta_0,\zeta_1)$ on $L$ and set $\zeta^\pm=(\zeta_0\pm \varepsilon,\zeta_1\pm \varepsilon)$ for sufficiently small $0<\varepsilon\ll 1$ such that they are in chambers adjacent to the wall $L$. 
Note that, by the classification in \S \ref{subsec-classification}, we have the unique 
\begin{NB4}$\zeta\to$\end{NB4}%
 $\zeta^\circ$-stable $A$-module $C$ such that $\zeta\cdot \dimv C=0$. 
We fix these notations throughout this subsection. 

\begin{lem}\label{noext}
$\Ext^1_{A}(C,C)=0$.
\end{lem} 
\begin{NB}
Let $X$ be one of the $A$-modules $C^-_+(m)$, $C^-_-(m)$, $C^+_+(m)$ and $C^+_-(m)$ appeared in Proposition \ref{classification}.
Then $\Ext^1_{\afmod}(X,X)=0$.
\end{NB}
\begin{proof}
By Remark \ref{rem-classification}, it is enough to check $\Ext^1_Y(z_*L,z_*L)=0$ for any line bundle $L$ on $\CP^1$, 
where $z\colon \CP^1\to \OO(-1)\oplus\OO(-1)=Y$ is the zero section.
By the adjunction we have
\[
\Ext^\bullet_Y(z_*L,z_*L)=\Ext^\bullet_{\CP^1}(\LL z^*z_*L,L).
\]
Since we have the Koszul resolution 
\[
\begin{array}{ccccccc}
0 & \to & \wedge^2\bigl((\pi^*\OO(-1)\oplus\pi^*\OO(-1))^*\bigr)\otimes\pi^*L&&&&\\
& \to & \wedge^1\bigl((\pi^*\OO(-1)\oplus\pi^*\OO(-1))^*\bigr)\otimes\pi^*L&&&&\\
& \to & \wedge^0\bigl((\pi^*\OO(-1)\oplus\pi^*\OO(-1))^*\bigr)\otimes\pi^*L & \to & z_*L &\to & 0 
\end{array}
\]
of $z_*L$, the object $\LL z^*z_*L$ is quasi-isomorphic to the complex
\[
0\to L(2)\to L(1)\oplus L(1)\to L\to 0.
\]
We can compute $\Ext^*_{\CP^1}(\LL z^*z_*L,L)$ by the spectral sequence of the double complex. 
The only non-zero in the $E_1$-terms are  $\Hom_{\CP^1}(L,L)\simeq \C$ and $\Ext^1_{\CP^1}(L(2),L)\simeq \C$.
Thus the spectral sequence degenerates and we have $\Hom_Y(z_*L,z_*L)=\Ext^3_Y(z_*L,z_*L)=\C$ and $\Ext^1_Y(z_*L,z_*L)=\Ext^2_Y(z_*L,z_*L)=0$.
\end{proof}

\begin{prop}\label{prop3.5}
\begin{enumerate}
\item
Let $\V'$ be a $\zeta^+$-stable $\A$-module.
Then we have an exact sequence 
\[
0\to \V\to \V'\to C^{\oplus k}\to 0,
\]
where $\V$ is a $\zeta^\circ$-stable $\A$-module.
The integer $k$ and the isomorphism class of $\V$ are determined uniquely. 
Moreover, the composition of the maps
\[
\C^k\overset{\sim}{\longrightarrow} \Hom_{\A}(C,C^{\oplus k})\overset{\V'\circ}{\longrightarrow}\Ext^1_{\A}(C,\V)
\]
is injective, where we regard $\V'$ as an element in $\Ext^1_{\A}(C^{\oplus k},\V)$.
\item
Let $\V''$ be a $\zeta^-$-stable $\A$-module.
Then we have an exact sequence
\[
0\to C^{\oplus k}\to \V''\to\V \to 0,
\]
where $\V$ is a $\zeta^\circ$-stable $\A$-module.
The integer $l$ and the isomorphism class of $\V$ are determined uniquely.
Moreover, the composition of the maps
\[
\C^l\overset{\sim}{\longrightarrow} \Hom_{\A}(C^{\oplus k},C)\overset{\circ\V''}{\longrightarrow}\Ext^1_{\A}(\V,C)
\]
is injective, where we regard $\V''$ as an element in $\Ext^1_{\A}(\V,C^{\oplus k})$.
\end{enumerate}
\end{prop}
\begin{proof}
We set
\[
\zeta_\infty=-\zeta_0\cdot\dim V_0'-\zeta_1\cdot\dim V_1'. 
\]
Note that $\V'$ is $\theta_{\tz}$-semistable and $\theta_{\tz}(\V')=0$.
Let
\[
\V'=\V^0\supset\cdots\supset \V^{L}\supset \V^{L+1}=0
\]
be \JH of $\V'$ with respect to the $\theta_{\tz}$-stability.
As we have mentioned before, there is an integer $0\leq l\leq L$ such that $\dim(\V^l/\V^{l+1})_\infty=1$ and $\dim(\V^{l'}/\V^{l'+1})_\infty=0$ for any ${l'}\neq {l}$. 
Then for ${l'}\neq l$ we have 
\begin{align*}
&\zeta^+_0\cdot \dim(\V^{l'}/\V^{{l'}+1})_0+\zeta^+_1\cdot \dim(\V^{l'}/\V^{{l'}+1})_1\\
&=\varepsilon\cdot(\dim(\V^{l'}/\V^{{l'}+1})_0+\dim(\V^{l'}/\V^{{l'}+1})_1)>0.
\end{align*}
From the $\zeta^+$-stability of $\V'$, we have $l=L$. 

Due to the classification in \S \ref{subsec-classification} and Lemma \ref{noext}, $\V'/\V_{L}$ is isomorphic to the direct sum $C^{\oplus k}$ for some $k$. The uniqueness follows from the uniqueness of factors of a Jordan-H\"older filtration. 

The composition of the maps is injective, since otherwise $\V'$ has $C$ as a direct summand and can not be $\zeta^+$-stable.

We can verify the claim of (2) similarly.
\end{proof}

Let $\mathrm{Gr}(k,\mathcal{V})$ be the Grassmannian variety of $k$-dimensional vector subspaces of a vector space $\mathcal{V}$. 
\begin{prop}\label{prop3.6}
\begin{enumerate}
\item
Let $\V$ be a $\zeta^\circ$-stable $\A$-module. 
For an element
\[
x\in \mathrm{Gr}(k,\dim\Ext^1_{\A}(C,\V)),
\]
let $\V'$ denote the framed $A$-module given by the universal extension 
\[
0\to \V\to \V'\to C^{\oplus k}\to 0
\]
corresponding to $x$. 
Then $\V'$ is $\zeta^+$-stable.
\item
Let $\V$ be a $\zeta^\circ$-stable $\A$-module. 
For an element
\[
y\in \mathrm{Gr}(l,\dim\Ext^1_{\A}(\V,C)),
\]
let $\V''$ denote the $\A$-module given by the universal extension 
\[
0\to C^{\oplus k}\to \V'\to \V\to 0
\]
corresponding to $y$. 
Then $\V''$ is $\zeta^-$-stable.
\end{enumerate}
\end{prop}
\begin{proof}
We set $\zeta_\infty$ and $\zeta_\infty^+$ so that 
\[
\tz\cdot \dimv (\V)=\tz\cdot \dimv (\V')=\tz^+\cdot \dimv (\V')=0.
\]
Let $\tilde{S}$ be a nonzero proper subobject of $\V'$ in $\hamod$. 
We should check $\tz^+\cdot \dimv(\tilde{S})<0$.

Suppose $\tilde{S}\cap \V=\emptyset$, then $\tilde{S}$ is mapped into $C^{\oplus k}$ injectively. 
Since $\V'$ does not have $C$ as its direct summand, $\tilde{S}$ is not isomorphic to a direct sum of $C$. 
So we have $\zeta^\circ\cdot \dimv(\tilde{S})<0$ because of the $\zeta^\circ$-stability of $C$. 
Since $\varepsilon$ is sufficiently small we have $\tz^+\cdot \dimv(\tilde{S})<0$ as well. 

Suppose $\emptyset\neq\tilde{S}\cap \V\subsetneq\V$. 
Since $\V$ is $\zeta^\circ$-stable and $C^{\oplus k}$ is $\zeta^\circ$-semistable we have
\[
\tz\cdot\dimv(\tilde{S}\cap \V)<0,\quad \tz\cdot\dimv(\im(\tilde{S}\to C^{\oplus k}))\leq 0.
\]
So we have
\[
\tz\cdot\dimv(\tilde{S})=\tz\cdot\dimv(\tilde{S}\cap \V)+\tz\cdot\dimv(\im(\tilde{S}\to C^{\oplus k}))<0.
\]
Because $\varepsilon$ is sufficiently small we have $\tz^+\cdot\dimv(\tilde{S})<0$ as well. 

Suppose $\tilde{S}\cap \V=\V$. 
Since $C^{\oplus k}$ is $\zeta^\circ$-semistable we have
$\tz\cdot\dimv(\coker(\tilde{S}\to C^{\oplus k}))\geq 0$. 
Because $\coker(\tilde{S}\to C^{\oplus k})\neq\emptyset$ and $\coker(\tilde{S}\to C^{\oplus k})_\infty=0$ we have
$\tz^+\cdot\dimv(\coker(\tilde{S}\to C^{\oplus k}))> 0$.
Hence we get
\[
\tz^+\cdot\dimv(\tilde{S})=\tz^+\cdot\dimv(\V')-\tz^+\cdot\dimv(\coker(\tilde{S}\to C^{\oplus k}))<0.
\]
We can verify the claim of (2) similarly.
\end{proof}

\begin{prop}\label{prop3.7} 
For a $\zeta^\circ$-stable $\A$-module $\V$ we have
\[
\ext^1_{\A}(C,\V)-\ext^1_{\A}(\V,C)=\dim C_0.
\]
\end{prop}
\begin{NB}
This lemma follows from the argument using the Koszul resolution of $A$ and $\A$. I quote the argument from the paper on 3-dimensional toric toric Calabi-Yau.
\begin{quote}
Let $A=(Q,\omega)$ be a quiver with superpotential. 
Assume that the Koszul type complex of projective $A$-bimodules
\[
\begin{array}{ccccccccc}
0 & \to & \oplus_{i\in Q_0}A\,\otimes\,\C\cdot e_i\,\otimes\,A & \overset{d_3}{\to} & \oplus_{a\in Q_1}A\,\otimes\,\C\cdot a\,\otimes\, A & & &&\\
 & \overset{d_2}{\to} & \oplus_{b\in Q_1}A\,\otimes\,\C\cdot b\,\otimes\,A & \overset{d_1}{\to} & \oplus_{i\in Q_0}A\,\otimes\,\C\cdot e_i\,\otimes\,A & \overset{m}{\to} & A &\to & 0,
\end{array}
\]
where all tensor products are taken over $A_0:=\C[Q_0]$, gives a resolution of $A$ (see \cite{graded3CY}, for example). In particular, $A$ is $3$-dimensional Calabi-Yau.

Take $(n_i)\in\Z_{\geq 0}^{Q_0}$ such that at least one of $n_i$'s is not zero. 
We make a new quiver $\Q$ by adding one vertex $\infty$ and $n_i$-arrows from the vertex $\infty$ to the vertex $i$ to the original quiver $Q$. 
The original superpotential $\omega$ gives the superpotential on the new quiver $\Q$ too. 
Let denote $\A:=(\Q,\omega)$ and consider the Koszul type complex of $\A$-bimodules
\[
\begin{array}{ccccccccc}
0 & \to & \oplus_{i\in Q_0}\A\,\otimes\,\C\cdot e_i\,\otimes\,\A & \to & \oplus_{a\in Q_1}\A\,\otimes\,\C\cdot a\,\otimes\, \A & & &&\\
 & \to & \oplus_{b\in \Q_1}\A\,\otimes\,\C\cdot b\,\otimes\,\A & \to & \oplus_{i\in \Q_0}\A\,\otimes\,\C\cdot e_i\,\otimes\,\A & \to & \A & \to & 0.
\end{array}
\]
This is also exact. 
The exactness at the last two terms is equivalent to the definition of generators and relations of the algebra $\A$.
The exactness at the first three terms is derived from that of the previous complex. 


Note that for $E$, $F\in \hafmod$ (resp.\ $\in\afmod$) we can compute $\Ext^*_{\A}(E,F)$ (resp.\ $\Ext^*_{A}(E,F)$) using the Koszul resolution. 
Let $d_i^{\A}(E,F)$ (resp.\ $d_i^{A}(E,F)$) denote the derivation in the complex derived from the Koszul resolution.
Then we have
\begin{align*}
&\homd_{\A}(E,F)-\ext^1_{\A}(E,F)\\
=\,&\sum_{i\in \Q_0}(\dim E_i\cdot \dim F_i)-\sum_{b\in\Q_1}(\dim E_{\mr{out}(b)}\cdot \dim F_{\mr{in}(b)})+\mr{rank}\left(d_2^{\A}(E,F)\right).
\end{align*}
Note that 
\[
\mr{rank}\left(d_2^{\A}(E,F)\right)=\mr{rank}\left(d_2^{A}(E,F)\right)
=\mr{rank}\left(d_2^{A}(F,E)\right)
=\mr{rank}\left(d_2^{\A}(F,E)\right),
\]
where the second equation comes from the self-duality of the Koszul complex of $A$.
Hence we have
\begin{align*}
&\homd_{\A}(E,F)-\ext^1_{\A}(E,F)+\ext^1_{\A}(F,E)-\homd_{\A}(F,E)\\
=\,& \sum_{b\in\Q_1}(\dim E_{\mr{out}(b)}\cdot \dim F_{\mr{in}(b)}-\dim E_{\mr{in}(b)}\cdot \dim F_{\mr{out}(b)}).
\end{align*}
\end{quote}
\begin{NB2}
I will update the following proof if we decide to use them (6/19).
\end{NB2}
\end{NB}

\begin{proof}
Let $S_\infty$ be the simple $\A$-module corresponding to the extended vertex $\infty$ and $V$ be the kernel of the natural map $\V\to S_\infty$
\begin{NB3}
We have the exact sequence
\[
0\to V\to \tilde{V}\to \C_\infty\to 0.
\]
of $\A$-modules,
where we regard $V$ as a framed $A$-module with trivial framing and $\C_\infty:=(0,\C,0)$ is the framed $A$-module with $1$-dimensional framing and trivial $A$-module.
\end{NB3}

First, applying the functor $\Hom_{\A}(C,-)$ to the
short exact sequence we have the following exact sequence:
\begin{equation*}
\xymatrix@R=.8pc@C=.9pc{
  & & \Hom_{\A}(C,S_\infty) \ar[lld]
\\
             \Ext^1_{\A}(C,V) \ar[r]
           & \Ext^1_{\A}(C,\V) \ar[r] 
           & \Ext^1_{\A}(C,S_\infty).
}
\end{equation*}
\begin{NB}
This is an example of the commutative diagram via \verb+xy-pic+
\begin{verbatim}
\begin{equation*}
\xymatrix@R=.8pc@C=.9pc{
  & & \Hom_{\A}(C,\C_\infty) \ar[lld]
\\
             \Ext^1_{\A}(C,V) \ar[r]
           & \Ext^1_{\A}(C,\V) \ar[r] 
           & \Ext^1_{\A}(C,\C_\infty).
}
\end{equation*}
\end{verbatim}
\[
\Hom_{\A}(C,\C_\infty)\to\Ext^1_{\A}(C,V)\to\Ext^1_{\A}(C,\V)\to\Ext^1_{\A}(C,\C_\infty).
\]
\end{NB}%
Clearly $\Hom_{\A}(C,S_\infty)=0$. 
We can also see that any extension 
\[
0\to S_\infty\to *\to C\to 0
\]
of $\A$-modules is always trivial, that is , $\Ext^1_{\A}(C,S_\infty)=0$. 
Hence we have
\[
\Ext^1_{\A}(C,\V)=\Ext^1_{\A}(C,V)=\Ext^1_{A}(C,V).
\]

On the other hand, applying the functor $\Hom_{\A}(-,C)$ to the short exact sequence we have the following exact sequence:
\begin{equation*}
\xymatrix@R=.8pc@C=.9pc{
  & \Hom_{\A}(\V,C) \ar[r] & \Hom_{\A}(V,C) \ar[lld]
\\
             \Ext^1_{\A}(\C_\infty,C) \ar[r]
           & \Ext^1_{\A}(\V,C) \ar[r] 
           & \Ext^1_{\A}(V,C).
}
\end{equation*}

Since both $\V$ and $C$ are $\zeta^\circ$-stable \begin{NB} corrected 7/7\end{NB}%
and they are not isomorphic, we have $\Hom_{\A}(\V,C)=0$.

Note that giving an extension 
\[
0\to C \to *\to S_\infty\to 0
\]
is equivalent to giving a map $\C\to C_0$
Hence we have $\ext^1_{\A}(S_\infty,C)=\dim C_0$.

Given an extension
\[
0\to C\to V'\to V\to 0
\]
of $A$-modules, by taking any lift of $\iota\colon V_\infty\simeq \C\to V_0$ with respect to the surjection $V'_0\to V_0$, we get an extension 
\[
0\to C\to \V'\to \V\to 0
\]
of framed $A$-modules. This means the map
\[
\Ext^1_{\A}(\V,C)\to\Ext^1_{\A}(V,C)\simeq\Ext^1_{A}(V,C)
\]
is surjective.

Now we have
\begin{align*}
&\ext^1_{\A}(C,\V)-\ext^1_{\A}(\V,C)\\
=&\ \ext^1_{A}(C,V)-\left(-\hom_{A}(V,C)+\dim C_0+\ext^1_{A}(V,C)\right)\\
=&\,\dim C_0+\left(\ext^1_{A}(C,V)-\ext^2_{A}(C,V)+\ext^3_{A}(C,V)\right)\\
=&\,\dim C_0+\chi_A(C,V)-\hom_{A}(C,V).
\end{align*}
Since $f$ is relative dimension $1$, the Euler form on $\cohc$ vanishes by Hirzebruch-Riemann-Roch theorem, and so does the Euler form on $\afmod$.

Both $\V$ and $C$ are $\zeta^\circ$-stable \begin{NB} corrected 7/7\end{NB}%
and they are not isomorphic, so we have $\Hom_{\A}(C,\V)=0$. Since the induced map $\Hom_{A}(C,V)\to \Hom_{\A}(C,\V)$ \begin{NB} corrected 7/7\end{NB}%
is injective, we have $\Hom_{A}(C,V)=\Hom_{\A}(C,V)=0$.

Finally the claim follows.
\end{proof}

We define stratifications on $\mf{M}^{\,\mr{s}}_{\zeta^\circ}(\vv)$ and $\M{\zeta^\pm}{\vv}$ as follows:
\begin{itemize}
\item let $\mf{M}^{\,\mr{s}}_{\zeta^\circ}(\vv)_N$ denote the subset of $\mf{M}^{\,\mr{s}}_{\zeta^\circ}(\vv)$ consisting of $\A$-modules $\V$ such that 
$\dim\Ext^1_{\Q}(C,\V)=N$,
\item let $\M{\zeta^+}{\vv'}_{N,k}$ denote the subset of $\M{\zeta^+}{\vv'}$ consisting of $\A$-modules $\V'$ such that there exists a $\zeta^\circ$-stable $\A$-module $\V\in \M{\zeta}{\vv'-k\cdot \dimv(C)}_N$ and an exact sequence
\[
0\to \V\to \V'\to C^{\oplus k}\to 0.
\]
\item let $\mf{M}^{\,\mr{s}}_{\zeta^\circ}(\vv)^N$ denote the subset of $\mf{M}^{\,\mr{s}}_{\zeta^\circ}(\vv)$ consisting of $\A$-modules $\V$ such that 
$\dim\Ext^1_{\Q}(\V,C)=N$, and 
\item let $\M{\zeta^-}{\vv'}^{N,k}$ denote the subset of $\M{\zeta^-}{\vv'}$ consisting of $\A$-modules $\V'$ such that there exists a $\zeta^\circ$-stable $\A$-module $\V\in \M{\zeta}{\vv'-k\cdot \dimv(C)}^N$ and an exact sequence
\[
0\to C^{\oplus k}\to \V'\to \V\to 0.
\]
\end{itemize}
\begin{NB}
I added explanations for the reason $\M{\zeta^\circ}{\vv}_N$ and $\M{\zeta^+}{\vv'}_{N,k}$ are subvarieties.
\end{NB}
\begin{lem}\label{lem-Koszul}
The subsets $\mf{M}^{\,\mr{s}}_{\zeta^\circ}(\vv)_N\subset\mf{M}^{\,\mr{s}}_{\zeta^\circ}(\vv)$ and $\M{\zeta^\pm}{\vv}_{N,k}\subset\M{\zeta^\pm}{\vv}$ have natural subscheme structures.
\end{lem}
\begin{proof}
Note that, for a morphism $f\colon E\to F$ of vector bundles on a scheme $X$ and an integer $n$, the subset $\{x\in X\mid \mr{rank}(f)=n\}$ has a natural subscheme structure given by the minor determinants.

We denote the subalgebra $\C Q_0=\oplus_{i\in Q_0}\C e_i$ of $A$ by $S$. 
For an $S$-module $T$ we define an $A$-bimodule $F_T$ by
\[
F_T=A\otimes_S T\otimes_S A.
\]
For $i, i'\in Q_0$ let $T_{i,i'}$ denote the $1$-dimensional $S$-module given by
\[
e_j\cdot 1=\delta_{j,i},\quad 1\cdot e_j=\delta_{j,i'},
\]
and we set 
\[
F_i:=F_{T_{i,i}},\quad F_{i,i'}:=F_{T_{i,i'}}.
\]
Note that $F_{i,i'}$ has the following natural basis:
\[
\{p\otimes 1\otimes q\mid p\in Ae_i,\ q\in e_{i'}A\}.
\]

For a quiver with superpotential $A$, the {\it Koszul complex} of $A$ is the following complex of $A$-bimodules: 
\[
0 \to \bigoplus_{i\in Q_0}F_i \overset{d_3}{\longrightarrow} \bigoplus_{a\in Q_1}F_{\mr{out}(a),\mr{in}(a)} \overset{d_2}{\longrightarrow} \bigoplus_{b\in Q_1}F_{\mr{in}(b),\mr{out}(b)} \overset{d_1}{\longrightarrow} \bigoplus_{i\in Q_0}F_i \overset{m}{\longrightarrow} A \to 0.
\]
Here the maps $m$, $d_1$, $d_3$ are given by 
\begin{align*}
m(p\otimes 1\otimes q)&=pq\quad (p\in Ae_i,\ q\in e_iA),\\
d_1(p\otimes 1\otimes q)&=(pb\otimes 1\otimes q)-(p\otimes 1\otimes bq)
\quad (p\in Ae_{\mr{in}(b)},\ q\in e_{\mr{out}(b)}A),\\
d_3(p\otimes 1\otimes q)&=\left(\sum_{a\colon \mr{in}(a)=i}pa\otimes 1\otimes q\right)-\left(\sum_{a\colon \mr{out}(a)=i}p\otimes 1\otimes aq\right)\quad(p\in Ae_i,\ q\in e_iA).
\end{align*}
The map $d_2$ is defined as follows; 
Let $c$ be a cycle in the quiver $Q$. 
We define the map $\partial_{c;a,b}\colon F_{\mr{out}(a),\mr{in}(a)}\to F_{\mr{in}(b),\mr{out}(b)}$ by 
\[
\partial_{c;a,b}(p\otimes 1\otimes q)=\sum_{
\begin{subarray}{c}
r\in e_{\mr{in}(a)}Ae_{\mr{out}(b)},\\
s\in e_{\mr{in}(b)}Ae_{\mr{out}(a)},\\
arbs=c
\end{subarray}
}ps\otimes 1\otimes rq.
\]
Then $d_2=\partial_{\omega}$ is defined as the linear combination of $\partial_c$'s. 

Since $A$ is graded $3$-dimensional Calabi-Yau algebra, the Koszul complex is exact (\cite[Theorem 4.3]{graded3CY}). 

We also consider the following Koszul type complex $\A$-bimodules:
\[
0 \to \bigoplus_{i\in Q_0}\tilde{F}_i \overset{\tilde{d}_3}{\longrightarrow} \bigoplus_{a\in Q_1}\tilde{F}_{\mr{out}(a),\mr{in}(a)} \overset{\tilde{d}_2}{\longrightarrow} \bigoplus_{b\in \Q_1}\tilde{F}_{\mr{in}(b),\mr{out}(b)} \overset{\tilde{d}_1}{\longrightarrow} \bigoplus_{i\in \Q_0}\tilde{F}_i \overset{\tilde{m}}{\longrightarrow} \A \to 0,
\]
where $\tilde{F}_i$, $\tilde{F}_{i,i'}$, $\tilde{d}_*$ and $\tilde{m}$ are defined in the same way.
This is also exact. 
The exactness at the last three terms is equivalent to the definition of generators and relations of the algebra $\A$.
The exactness at the first two terms is derived from that of the exactness of the Koszul complex of $A$.

Let $\tilde{\mca{V}}=\oplus_{i\in\Q_0}\tilde{\mca{V}}_i$ be the universal bundle on $\M{\zeta^\circ}{\vv}$. 
The Koszul complex of $\A$-bimodules induces the following complexes of the vector bundles on $\M{\zeta}{\vv}$:
\[
\bigoplus_{a\in Q_1}\Hom(C_{\mr{out}(a)},\tilde{\mca{V}}_{\mr{in}(a)})\overset{d_2}{\longrightarrow}\bigoplus_{b\in \Q_1}\Hom(C_{\mr{in}(b)},\tilde{\mca{V}}_{\mr{out}(b)})\overset{d_1}{\longrightarrow}\bigoplus_{i\in \Q_0}\Hom(C_i,\tilde{\mca{V}}_i)
{\to} 0.
\]
If we restrict this complex to some closed point $\V$ of $\M{\zeta}{\vv}$, then the right and left cohomologies give $\Hom(C,\V)$ and $\Ext^1(C,\V)$ respectively.

Note that for $\V\in\mf{M}_{\zeta^\circ}^{\,\mr{s}}(\vv)$ we have $\hom(C,\V)=0$. 
This means that the morphism $d_1$ between vector bundles are surjective, and hence $\ker(d_1)$ is a vector bundle.  
The subset $\mf{M}_{\zeta^\circ}^{\,\mr{s}}(\vv)_N$ consists of 
closed points $\V$ such that $\mr{rank}(d_2(\V))=\mr{rank}(\ker(d_1))-N$ and so has a natural subscheme structure.

Let $\M{\zeta^+}{\vv'}_{k}$ denote the subscheme of $\M{\zeta^+}{\vv'}$ consisting of closed points $\V'$ such that 
$\mr{rank}(d_1(\V'))=\dimv\,C\cdot \vv'-k$. 
We have the canonical morphism $\M{\zeta^+}{\vv'}_{k}\to \mf{M}_{\zeta^\circ}^{\,\mr{s}}(\vv)$ ($\vv=\vv'-k\cdot\dimv\,C$) such that a closed point $\V'\in\M{\zeta^+}{\vv'}_{k}$ is mapped to the closed point $\V\in\mf{M}_{\zeta^\circ}^{\,\mr{s}}(\vv)$ appeared in the exact sequence
\[
0\to \V\to \V'\to C^{\oplus k}\to 0.
\]
The subset $\M{\zeta^+}{\vv'}_{N,k}$ coincides the inverse image of $\mf{M}_{\zeta^\circ}^{\,\mr{s}}(\vv)_N$ with respect to the above morphism and so has a natural subscheme structure. 

Similarly, we can define subschemes $\mf{M}_{\zeta^\circ}^{\,\mr{s}}(\vv)^N$, $\M{\zeta^+}{\vv'}^{k}$ and $\M{\zeta^+}{\vv'}^{N,k}$ using the exact sequence
\[
\bigoplus_{a\in Q_1}\Hom(\tilde{\mca{V}}_{\mr{in}(a)},C_{\mr{out}(a)})\overset{d_2}{\longrightarrow}\bigoplus_{b\in \Q_1}\Hom(\tilde{\mca{V}}_{\mr{out}(b)},C_{\mr{in}(b)})\overset{d_1}{\longrightarrow}\bigoplus_{i\in \Q_0}\Hom(\tilde{\mca{V}}_i,C_i)
{\to} 0,
\]
whose cohomologies give $\Hom(\V,C)$ and $\Ext^1(\V,C)$.
\end{proof}

By Proposition \ref{prop3.5} and Proposition \ref{prop3.6}, 
the natural map 
\[
\M{\zeta^+}{\vv'}_{N,k}\to \Ms{\zeta^\circ}{\vv}_N
\]
is a $\mathrm{Gr}(k,N)$-fibration. 
So, we have 
\[
\chi(\M{\zeta^+}{\vv'}_{N,k})=\chi(\mathrm{Gr}(k,N))\cdot\chi(\M{\zeta}{\vv}_N)
\]
and 
\begin{align*}
\sum_{\vv'}\chi(\M{\zeta^+}{\vv'})\cdot\q^{\vv'}&=
\sum_{\vv',N,k}\chi(\M{\zeta^+}{\vv'}_{N,k})\cdot\q^{\vv'}\\
&=\sum_{\vv,N,k}\chi(\mathrm{Gr}(k,N))\cdot\chi(\Ms{\zeta^\circ}{\vv}_{N})\cdot\q^{\vv+k\cdot\dimv(C)}\\
&=\sum_{\vv,N}\left(\sum_k\chi(\mathrm{Gr}(k,N))\cdot\q^{k\cdot\dimv(C)}\right)\chi(\Ms{\zeta^\circ}{\vv}_{N})\cdot\q^\vv\\
&=\sum_{\vv,N}\left(1+\q^{\dimv(C)}\right)^N\chi(\Ms{\zeta^\circ}{\vv}_{N})\cdot\q^\vv.
\end{align*}
Similarly we have
\[
\sum_{\vv'}\chi(\M{\zeta^-}{\vv'})\cdot\q^{\vv'}=
\sum_{\vv,N}\left(1+\q^{\dimv(C)}\right)^N\chi(\Ms{\zeta^\circ}{\vv}^{N})\cdot\q^\vv.
\]
By Proposition \ref{prop3.7} we have
$\Ms{\zeta^\circ}{\vv}_N=\Ms{\zeta^\circ}{\vv}^{N+\dim C_0}$. \begin{NB} corrected 7/7\end{NB}%
Hence we have 
\begin{align*}
\sum_{\vv'}\chi(\M{\zeta^-}{\vv'})\cdot\q^{\vv'}
&=\sum_{\vv,N}\left(1+\q^{\dimv(C)}\right)^N\chi(\Ms{\zeta^\circ}{\vv}^{N})\cdot\q^\vv\\
&=\sum_{\vv,N}\left(1+\q^{\dimv(C)}\right)^N\chi(\Ms{\zeta^\circ}{\vv}_{N-\dim C_0})\cdot\q^\vv\\
&=\sum_{\vv,N}\left(1+\q^{\dimv(C)}\right)^{N+\dim C_0}\chi(\Ms{\zeta^\circ}{\vv}_{N})\cdot\q^\vv\\
&=\left(1+\q^{\dimv(C)}\right)^{\dim C_0}\cdot\sum_{\vv'}\chi(\M{\zeta^+}{\vv'})\cdot\q^{\vv'}.
\end{align*}
In summary, we have the following {\it wall-crossing formula}: 
\begin{thm}\label{wallcrossing}
\[
\mca{Z}'_{\zeta^-}(\q)=\left(1+\q^{\dimv(C)}\right)^{\dim C_0}\cdot \mca{Z}'_{\zeta^+}(\q).
\]
\end{thm}

\subsection{DT, PT and NCDT}\label{dt-pt-ncdt}
\begin{NB} This subsection is added (6/19).\end{NB}%
Let $I_n(Y,d)$ denote the moduli space of ideal sheaves $\mca{I}_{Z}$ of one dimensional subschemes $Z\subset Y$. 
whose Hilbert polynomials are given by
\[
\chi(\OO_Z\otimes \mca{L}^{\otimes K})=n+K\cdot \int_{[C]} d\cdot c_1(\mca{L})=n+dK.
\]
We define the {\it Donaldson-Thomas invariants} $I_{n,d}$ from $I_n(Y,d)$ using Behrend's function as is \S \ref{defofncdt} (\cite{thomas-dt}, \cite{behrend-dt}), 
and their generating function by
\[
\mca{Z}_{\mr{DT}}(Y;q,t):=\sum_{n,d}I_{n,d}\cdot q^nt^d.
\] 

Let $P_n(Y,d)$ denote the moduli space of stable pairs $(F,s)$ such that the Hilbert polynomials of $F$'s are given by the same equation as above. 
We define the {\it Pandharipande-Thomas invariants} $P_{n,d}$ from $P_n(Y,d)$ using Behrend's function (\cite{pt1}), 
and their generating function by
\[
\mca{Z}_{\mr{PT}}(Y;q,t):=\sum_{n,d}P_{n,d}\cdot q^nt^d.
\] 

We define the Donaldson-Thomas invariants and the Pandharipande-Thomas invariants of $Y^+$ using $[C^+]\in H_2(Y^+)$ instead of $[C]\in H_2(Y)$. 
Note that the natural isomorphism $H_2(Y)\to H_2(Y^+)$ maps $[C]$ to $-[C^+]$.

For $F\in D^b_c(Y)$ we have
\[
\chi(F\otimes \mca{L}^{\otimes K})=\chi(F)+K(\chi(F)-\chi(F\otimes \mca{L}^{-1})).
\]
and so $n=\mathrm{v}_0$, $d=\mathrm{v}_0-\mathrm{v}_1$. 
Put $q=q_0q_1$ and $t=q_1^{-1}$, then we have $q^nt^d=q_0^{\mathrm{v}_0}q_1^{\mathrm{v}_1}$.
We set $\zeta^\pm=(-1\pm\varepsilon,1)$ for sufficiently small $\varepsilon>0$. The results in \S \ref{cs-pcs} are summarized as follows:   
\begin{prop}
\begin{NB}
$(q_0q_1)^nq_1^d=q_0^{\mathrm{v}_0}q_1^{\mathrm{v}_1}$ for $Y$%
\end{NB}%
\begin{align*}
\mca{Z}_{\mr{DT}}(Y;q_0q_1,q_1^{-1})&=\mca{Z}_{\zeta^-}(\q), &
\mca{Z}_{\mr{DT}}(Y^+;q_0q_1,q_1^{-1})&=\mca{Z}_{-\zeta^+}(\q),\\
\mca{Z}_{\mr{PT}}(Y;q_0q_1,q_1^{-1})&=\mca{Z}_{\zeta^+}(\q),&
\mca{Z}_{\mr{PT}}(Y^+;q_0q_1,q_1^{-1})&=\mca{Z}_{-\zeta^-}(\q). 
\end{align*}
\end{prop}
\begin{rem}
Here we denote, with a slight abuse of the notations, by $\mca{Z}_{\pm\zeta^\pm}(\q)$ the generating functions of the virtual counting of $\M{\pm\zeta^\pm}{\vv}$ for sufficiently small $\varepsilon>0$ for each $\vv$. We can not take $\varepsilon>0$ uniformly.  
\end{rem}

We set $\zeta^{(\pm)}=(\pm 1,\pm 1)$. 
Note that $\M{\zeta^{(+)}}{\vv}$ is empty unless $\vv=0$ and so $\mca{Z}_{\zeta^{(+)}}(\q)=1$.
The invariants $D_{\zeta^{(-)}}(\vv)$ are the {\it non-commutative Donaldson-Thomas invariants} defined in \cite{szendroi-ncdt}.
We denote their generating function $\mca{Z}_{\zeta^{(-)}}(\q)$ by $\mca{Z}_{\mr{NCDT}}(\q)$.

Applying the wall-crossing formula in Theorem \ref{wallcrossing}, we obtain the following relations between generating functions:
\begin{thm}\label{thm-main}
\begin{align}
&\mca{Z}_{\mr{NCDT}}(\q)=\left(\prod_{m\geq 1}^\infty (1+q_0^{m}(-q_1)^{m+1})^m\right)\cdot\mca{Z}_{\mr{DT}}(Y;q_0q_1,q_1^{-1}),\label{eq-young}\\
&\mca{Z}_{\mr{NCDT}}(\q)=\left(\prod_{m\geq 1}^\infty (1+q_0^{m}(-q_1)^{m-1})^m\right)\cdot\mca{Z}_{\mr{DT}}(Y^+;q_0q_1,q_1^{-1}),\\
&\mca{Z}_{\mr{PT}}(Y;q_0q_1,q_1^{-1})=\prod_{m\geq 1}^\infty (1+q_0^{m}(-q_1)^{m-1})^m,\label{eq-pt}\\
&\mca{Z}_{\mr{PT}}(Y^+;q_0q_1,q_1^{-1})=\prod_{m\geq 1}^\infty (1+q_0^{m}(-q_1)^{m+1})^m.
\end{align}
\end{thm}
\begin{rem}\label{rem-final}
\begin{enumerate}
\item The formula \eqref{eq-young} was shown by a combinatorial method in \cite{young-conifold}.
\item The generating function of Donaldson-Thomas invariants is described in terms of the topological vertex (\cite{mnop}). 
The topological vertex for the conifold is computed in \cite{superrigid-dt}:
\begin{equation}\label{eq-bb}
\mca{Z}_{\mr{DT}}(Y;q,t)=\left(\prod_{m\geq 1}^\infty (1-(-q)^{m})^{-m}\right)^2\left(\prod_{m\geq 1}^\infty (1-(-q)^{m}t)^m\right).
\end{equation}
\item The DT-PT correspondence conjecture (\cite{pt1}) asserts that 
\[
\mca{Z}_{\mr{DT}}(Y;q,t)=\mca{Z}_{\mr{PT}}(Y;q,t)\cdot \left(\prod_{m\geq 1}^\infty (1-(-q)^{m})^{-m}\right)^{e(Y)}.
\]
The conjecture for the conifold follows from formula \eqref{eq-pt} and \eqref{eq-bb}, although Theorem \ref{wallcrossing} does not cover the wall $L^-(\infty)$.
The wall-crossing for the wall $L^-(\infty)$ requires Joyce's general theory (see \cite{toda-dtpt} and \cite{thomas_stoppa}).
\item 
Since the contribution of the wall $L^-_+(m)$ (resp.\ $L^-_-(m)$) coincides with that of $L^+_+(m)$ (resp.\ $L^+_-(m)$), we have the following formula, which provides a conceptual interpretation of the result in \cite{szendroi-ncdt} and \cite{young-conifold}:
\begin{align*}
\mca{Z}_{\mr{NCDT}}(\q)&=\mca{Z}_{\mr{DT}}(Y;q_0q_1,q_1^{-1})\cdot\mca{Z}_{\mr{PT}}(Y^+;q_0q_1,q_1^{-1})\\
&=\mca{Z}_{\mr{PT}}(Y;q_0q_1,q_1^{-1})\cdot\mca{Z}_{\mr{DT}}(Y^+;q_0q_1,q_1^{-1}). 
\end{align*}
\item 
Since the contribution of the wall $L^-_+(m)$ coincides with that of $L^-_-(m)$ after the change $(q_0q_1,q_1^{-1})\mapsto (q_0q_1,q_1)$ of variables, we get the flop invariance of DT and PT invariants:
\begin{align*}
\mca{Z}_{\mr{PT}}(Y;q_0q_1,q_1^{-1})=\mca{Z}_{\mr{PT}}(Y^+;q_0q_1,q_1),\\
\mca{Z}_{\mr{DT}}(Y;q_0q_1,q_1^{-1})=\mca{Z}_{\mr{DT}}(Y^+;q_0q_1,q_1).\\
\end{align*}
\end{enumerate}
\end{rem}

\section{Replacement of tilting bundles and stabilities}\label{appendix}
In the final section, we provide an alternative description of the moduli spaces $\M{\zeta}{\vv}$ for generic stability parameter $\zeta$ in the case of the conifold.
As by-products, we can see that the torus fixed point set $\M{\zeta}{\vv}^T$ is isolated and parameterized by the "pyramid partitions" which appeared in \cite{szendroi-ncdt}, \cite{young-conifold} and \cite{chuang-jafferis}. 

\subsection{Characterization of stable objects}
Let $\zeta_{\mathrm{triv}}$ and $\zeta_{\mathrm{cyclic}}$ be stability parameters such that
\[
\zeta_{\mathrm{triv},0}, \zeta_{\mathrm{triv},1}>0,\quad 
\zeta_{\mathrm{cyclic},0}, \zeta_{\mathrm{cyclic},1}<0.
\]
For $m\geq 1$, let $\zeta^{m,\pm}=(\zeta^{m,\pm}_{\,0},\zeta^{m,\pm}_{\,1})$ be stability parameters such that 
\begin{align*}
\zeta^{m,+}_{\,0}<\zeta^{m,+}_{\,1},\quad m\zeta^{m,+}_{\,0}+(m-1)\zeta^{m,+}_{\,1}<0,\quad (m+1)\zeta^{m,+}_{\,0}+m\zeta^{m,+}_{\,1}>0,\\ 
\zeta^{m,-}_{\,0}<\zeta^{m,-}_{\,1},\quad (m-1)\zeta^{m,-}_{\,0}+m\zeta^{m,-}_{\,1}>0,\quad m\zeta^{m,-}_{\,0}+(m+1)\zeta^{m,-}_{\,1}<0
\end{align*}
(see Figure \ref{fig:zeta2}).

\begin{figure}[htbp]
\def\JPicScale{.8}
  \centering
   \input{pic12.tpc}
  \caption{$\zeta^{m,\pm}$, $\zeta_{\mathrm{triv}}$ and $\zeta_{\mathrm{cyclic}}$}
\label{fig:zeta2}
\end{figure}

\begin{lem}\label{lem-4.1}
\begin{enumerate}
\item
A \pc system $(F,s)\in \tpervc$ is $\zeta^{m,+}$-stable if and only if the following three conditions are satisfied: 
\begin{align}
&\Hom_{\tpervc}((F,s),(z_*\OO_{\CP^1}(m-1),0,0))=0,\label{eq-4.1.1}\\
&\Hom_{\tpervc}((z_*\OO_{\CP^1}(m),0,0),(F,s))=0,\label{eq-4.1.2}\\
&\Hom_{\tpervc}((\OO_x,0,0),(F,s))=0\quad (\forall x\in Y).\label{eq-4.1.3}
\end{align}
\item
A \pc system $(F,s)\in \tpervc$ is $\zeta^{m,-}$-stable if and only if the following three conditions are satisfied:
\begin{align}
&\Hom_{\tpervc}((z_*\OO_{\CP^1}(-m)[1],0,0),(F,s))=0,\\
&\Hom_{\tpervc}((F,s),(z_*\OO_{\CP^1}(-m-1)[1],0,0))=0,\\
&\Hom_{\tpervc}((F,s),(\OO_x,0,0))=0\quad (\forall x\in Y).
\end{align}
\end{enumerate}
\end{lem}
\begin{proof}
Let $(F,s)\in \tpervc$ be a \pc system satisfying the conditions \eqref{eq-4.1.1}-\eqref{eq-4.1.3}.
Assume that $(F,s)$ is not $\zeta^{m,+}$-stable. 
Let 
\[
(F,s)=\tilde{F}^0\supset\tilde{F}^1\supset\cdots\supset\tilde{F}^L\supset 0
\]
be a filtration of $(F,s)$ by $\theta_{\tilde{\zeta}^{m,+}}$-stable subquotients, which is given by combining \HN of $(F,s)$ and Jordan-H\"older filtrations of its factors. 
Here $\theta_{\tilde{\zeta}^{m,+}}$ is chosen so that $\theta_{\tilde{\zeta}^{m,+}}(F,s)=0$.
Hence we have $\theta_{\tilde{\zeta}^{m,+}}(\tilde{F}_0/\tilde{F}_1)<0$,  
$\theta_{\tilde{\zeta}^{m,+}}(\tilde{F}_k)>0$ and 
either $(\tilde{F}_0/\tilde{F}_1)_\infty=0$ or $(\tilde{F}_k)_\infty=0$.
First suppose $(\tilde{F}_0/\tilde{F}_1)_\infty=0$.
By the classification in \S \ref{subsec-classification}, $\tilde{F}_0/\tilde{F}_1$ is isomorphic to $(z_*\OO_{\CP^1}(m'-1),0,0)$ for some $m'\leq m$. 
Next suppose $(\tilde{F}_k)_\infty=0$. Then $\tilde{F}_k$ is isomorphic to one of the following objects:
\begin{align*}
&(z_*\OO_{\CP^1}(m'-1),0,0) &(m'> m),\\ 
&(\OO_x,0,0) &(x\in Y),\\
&(z_*\OO_{\CP^1}(-m'-1)[1],0,0)&(m'\geq 0). 
\end{align*}
In both cases, the existence of nonzero homomorphisms
\begin{align*}
&\Hom_{\pervc}(z_*\OO_{\CP^1}(m'-1),z_*\OO_{\CP^1}(m-1))\neq0\quad &(m\geq m'),\\
&\Hom_{\pervc}(z_*\OO_{\CP^1}(m),z_*\OO_{\CP^1}(m'-1))\neq0\quad &(m< m'),\\
&\Hom_{\pervc}(z_*\OO_{\CP^1}(m),z_*\OO_{\CP^1}(-m'-1)[1])\neq0\quad &(1\leq m,\,0\leq m').
\end{align*}
contradicts the conditions \eqref{eq-4.1.1}-\eqref{eq-4.1.3}.
Hence $(F,s)$ is $\zeta^{m,+}$-stable.
The opposite direction is trivial.
We can show the claim (2) in the same way.
\end{proof}

\subsection{New framed quivers}\label{subsec-new-framed-quivers}
\begin{NB3}rewrittend\end{NB3}%
Recall that we put $\mca{L}=\pi^*\OO_{\CP^1}(1)$. 
For an integer $m$, we set $\mathcal{L}_m:=\mathcal{L}^{\otimes m}$.
Let $\mathcal{P}_m^+$ (resp. $\mathcal{P}_m^-$) be the full subcategory of $D^b(\coh)$ consisting of objects $F$ such that $F\otimes \mathcal{L}_{-m}[1]\in \perv$ (resp. $F\otimes \mathcal{L}_{m+1}\in \perv$). 
Note that $\mathcal{L}_m[-1]\oplus \mathcal{L}_{m+1}[-1]$ (resp. $\mathcal{L}_{-m-1}\oplus \mathcal{L}_{-m}$) is a projective generator in $\mathcal{P}_m^+$ (resp. $\mathcal{P}_m^-$) and gives the equivalence 
\[
\Phi_m^\pm\colon \mathcal{P}_m^\pm\to \amod.
\]
Note that $A$ is the algebra defined in \S \ref{subsec-quiver-for-conifold}, which is independent of $m$.
We set ${}^c\mathcal{P}_m^\pm:=\mathcal{P}_m^\pm\cap D^b_c(\coh)$, which is equivalent to $\afmod$.

Let $A_m^+$ be the algebra defined by the following quiver with relations:
\begin{figure}[htbp]
  \centering
  \includegraphics{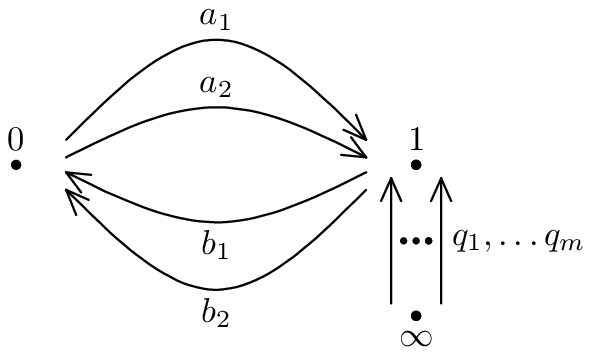}
  \caption{quiver $Q_m^+$}\label{quiver(Qm+)}
\end{figure}
the quiver $Q_m^+$ is given as in Figure \ref{quiver(Qm+)} and the following relations are added to the usual ones:
\begin{equation}\label{eq-rel+}
b_1q_1=0,\quad b_1q_{i+1}=b_2q_{i}\ (i=1,\ldots m-1),\quad b_2q_m=0.
\end{equation}
Similarly, we define the algebra $A_m^-$ as follows:
\begin{figure}[htbp]
  \centering
  \includegraphics{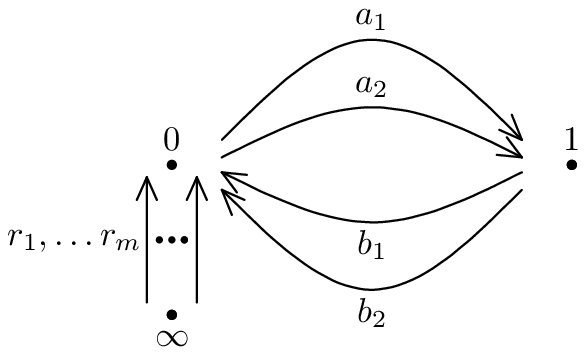}
  \caption{quiver $Q_m^-$}\label{quiver-m-}
\end{figure}
the quiver $Q_m^-$ is given as in Figure \ref{quiver-m-} and the following relations are added to the usual ones:
\begin{equation}\label{eq-rel-}
a_1r_{i+1}=a_2r_{i}\ (i=1,\ldots m-1).
\end{equation}

Let $S_\infty$ and $P_\infty$  be the simple and indecomposable projective $A_m^\pm$-modules corresponding to the extended vertex $\infty$. 
Let $P$ denote the kernel of the canonical map $P_\infty\to S_\infty$. 
\begin{prop}\label{prop-O_Y}
\[
\Phi_m^\pm(\OO_Y)=P.
\]
\end{prop}
\begin{proof}
We will prove the claim for $\Phi_m^+$.
Let $B_1, B_2\in H^0(Y,\mathcal{L})$ be basis elements in 
\[
\C^2\simeq H^0(\CP^1,\mathcal{O}_{\CP^1}(1)) \hookrightarrow H^0(Y,\mathcal{L})
\] 
We consider the map 
\[
\sum (B^{(i,i)}_1+B^{(i.i+1)}_2)\colon (\mathcal{L}_{-m-1})^{\oplus m}\to(\mathcal{L}_{-m})^{\oplus m+1},
\]
where 
$B^{(i,j)}_\varepsilon=\pi^{j}\circ B_\varepsilon\circ \eta^i$ and $\pi^{i}$ and $\eta^i$ are the canonical projection to inclusion of the $i$-th factor of the direct sum.
This map is injective and the cokernel is isomorphic to the structure sheaf $\OO_Y$.
Applying $\Phi_m^+$ we get the following map:
\[
\sum (b^{(i,i)}_1+b^{(i.i+1)}_2)\colon P_1^{\oplus m}\to P_0^{\oplus m+1},
\]
where $P_0$ and $P_1$ are the indecomposable projective $A$-modules and $b^{(i,j)}$ is defined as above. 
We can verify that this map is injective and the cokernel is isomorphic to $P$.
Hence the claim follows.
\end{proof}
Let $\tilde{\mathcal{P}}_m^\pm$ denote the category of pairs $(F,W,s)$, where $F\in\mathcal{P}_m^\pm$ and $s\colon W\otimes \OO_Y\to F$.
Let ${}^c\tilde{\mathcal{P}}_m^\pm$ denote the full subcategory of pairs $(F,W,s)$ such that $F\in{}^c\mathcal{P}_m^\pm$ and such that $W$ is finite dimensional.
\begin{prop}\label{prop-equiv}
\[
\tilde{\mathcal{P}}_m^\pm\simeq A_m^\pm\text{-}\mathrm{Mod},\quad
{}^c\tilde{\mathcal{P}}_m^\pm\simeq A_m^\pm\text{-}\mathrm{mod}.
\]
\end{prop}
\begin{proof}
First, giving a pair $(F,W,s)\in \tilde{\mathcal{P}}_m^\pm$ is equivalent to giving a linear map $W\to \Hom(\OO_Y,F)$. 
Note that
\[
\Hom(\OO_Y,F)\simeq \Hom_A(P,V)\simeq \Hom_{A_m^\pm}(P,V)\simeq \Ext^1_{A_m^\pm}(S_\infty,V)
\]
where $V:=\Phi^\pm_m(P)$.
The claim follows, since giving a linear map $W\to \Ext^1_{A_m^\pm}(S_\infty,V)$ is equivalent to giving an $A_m^\pm$-module, .
\end{proof}

\begin{NB3}
\begin{prop}
The category $\tilde{\mathcal{P}}_m^+$ (resp. $\tilde{\mathcal{P}}_m^-$) is equivalent to the category of finite dimensional $A_m^+$-modules (resp. $A_m^-$-modules).
\end{prop}
\begin{proof}
Let $V=V_0\oplus V_1$ be a finite dimensional $A$-module and 
\[
b_1,b_2\colon V_0\to V_1
\]
be the linear maps associated to the $A$-module structure. 
We consider the following map:
\begin{equation}\label{eq-new-quiver}
\sum (b^{(i,i)}_1+b^{(i.i+1)}_2)\colon V_0^{\oplus m}\to V_1^{\oplus m+1},
\end{equation}
where 
$b^{(i,j)}_\varepsilon=\pi^{j}\circ b_\varepsilon\circ \eta^i$ and $\pi^{i}$ and $\eta^i$ are the canonical projection to inclusion of the $i$-th factor of the direct sum.
Note that giving a finite dimensional $A_m^+$-module $\tilde{V}=V_0\oplus V_1\oplus V_\infty$ is equivalent to giving a finite dimensional $A$-module $V=V_0\oplus V_1$ and a linear map from $V_\infty$ to the kernel of the map \eqref{eq-new-quiver}.

Let $B_1, B_2\in H^0(Y,\mathcal{L})$ be basis elements in 
\[
\C^2\simeq H^0(\CP^1,\mathcal{O}_{\CP^1}(1)) \hookrightarrow H^0(Y,\mathcal{L})
\] 
For an element $\F\in \mathcal{P}_m^+$, the corresponding $A$-module is given 
\[
H^1(Y,F\otimes \mathcal{L}_{-m-1})\oplus H^1(Y,F\otimes \mathcal{L}_{-m})
\]
as a vector space and the linear maps associated to the arrows $b_1$ and $b_2$ is induced by $B_1$ and $B_2$.
For $m\geq 1$, we consider the map 
\[
\sum (B^{(i,i)}_1+B^{(i.i+1)}_2)\colon (\mathcal{L}_{-m-1})^{\oplus m}\to(\mathcal{L}_{-m})^{\oplus m+1},
\]
where we use similar notations as those in \eqref{eq-new-quiver}.
This map is injective and the cokernel is isomorphic to the structure sheaf $\OO_Y$.
Moreover, this map is the universal extension corresponding to the subspace
\[
\C^m\simeq \Ext^1(\CP^1,\mathcal{O}_{\CP^1}(-m-1))\hookrightarrow \Ext^1(\OO_Y,\mathcal{L}_{-m-1}).
\]
For $F\in\mathcal{P}_m^+$ we have an exact sequence
\begin{equation}\label{eq-exact}
0\to H^0(Y,F)\to H^1(Y,F\otimes \mathcal{L}_{-m-1})^{\oplus m}\to H^1(Y,F\otimes \mathcal{L}_{-m})^{\oplus m+1}.
\end{equation}
Giving a morphism $s\colon W\otimes \OO_Y\to F$ ia equivalent to giving a linear map from $W$ to $H^0(Y,F)$, which is isomorphic to the kernel of the last map in \eqref{eq-exact}.
As we mentioned in the previous paragraph, this is equivalent to giving an $A_m^+$-module structure on 
\[
H^1(Y,F\otimes \mathcal{L}_{-m-1})\oplus H^1(Y,F\otimes \mathcal{L}_{-m})\oplus  W.
\]

We can show the statement for $F\in\mathcal{P}_m^-$ in the same way.
\end{proof}
\end{NB3}

\begin{NB3}
Similarly, we define the algebra $A_m^-$ as follows:
\begin{figure}[htbp]
  \centering
  \includegraphics{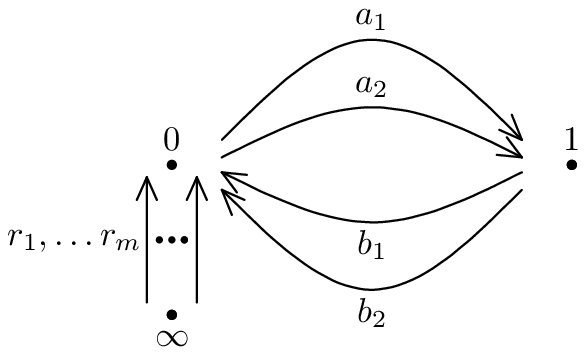}
  \caption{quiver $Q_m^-$}\label{quiver-m-}
\end{figure}
the quiver $Q_m^-$ is given as in Figure \ref{quiver-m-} and the following relations are added to the usual ones:
\begin{equation}\label{eq-rel-}
a_1r_{i+1}=a_2r_{i}\ (i=1,\ldots m-1).
\end{equation}
\end{NB3}

We can define $\zeta$-(semi)stability for finite dimensional $A_m^\pm$-modules as in Definition \ref{def-st}. 
In order to make it clear in what category we work, we use the notation``$(\zeta,{}^c\tilde{\mathcal{P}}_m^\pm)$-(semi)stability''.
From now on, the $\zeta$-(semi)stability for modules of the original quiver $\tilde{Q}$ is written ``$(\zeta,\tpervc)$-(semi)stability''. 
We can construct the moduli spaces $\mathfrak{M}^{A_m^\pm}_{\zeta}(\mathrm{v}_0,\mathrm{v}_1)$
of $(\zeta,{}^c\tilde{\mathcal{P}}_m^\pm)$-semistable $A_m^\pm$-modules $V$ with $\dimv\,V=(\mathrm{v}_0,\mathrm{v}_1,1)$. 

\begin{NB3}old version:
For an integer $m$, we set $\mathcal{L}_m:=\mathcal{L}^{\otimes m}$.
Let $\mathcal{P}_m^+$ (resp. $\mathcal{P}_m^-$) be the full subcategory of $D^b_c(\coh)$ consisting of objects $F$ such that $F\otimes \mathcal{L}_{-m}[1]\in \pervc$ (resp. $F\otimes \mathcal{L}_{m+1}\in \pervc$). 
Note that $\mathcal{L}_m[-1]\oplus \mathcal{L}_{m+1}[-1]$ (resp. $\mathcal{L}_{-m-1}\oplus \mathcal{L}_{-m}$) is a projective generator in $\mathcal{P}_m^+$ (resp. $\mathcal{P}_m^-$).
Let $\tilde{\mathcal{P}}_m^\pm$ denote the category of pairs $(F,W,s)$, where $F\in\mathcal{P}_m^\pm$ and $s\colon W\otimes \OO_Y\to F$.

Let $b_1, b_2\in H^0(Y,\mathcal{L})$ be basis elements in 
\[
\C^2\simeq H^0(\CP^1,\mathcal{O}_{\CP^1}(1)) \hookrightarrow H^0(Y,\mathcal{L}).
\] 
For $m\geq 1$, we consider the map 
\[
\sum (b^{(i,i)}_1+b^{(i.i+1)}_2)\colon (\mathcal{L}_{-m-1})^{\oplus m}\to(\mathcal{L}_{-m})^{\oplus m+1},
\]
where 
$b^{(i,j)}_\varepsilon=\pi^{j}\circ b_\varepsilon\circ \eta^i$ and $\pi^{i}$ and $\eta^i$ are the canonical projection to inclusion of the $i$-th factor of the direct sum.
This map is injective and the cokernel is isomorphic to the structure sheaf $\OO_Y$.
Moreover, this map is the universal extension corresponding to the subspace
\[
\C^m\simeq \Ext^1(\CP^1,\mathcal{O}_{\CP^1}(-m-1))\hookrightarrow \Ext^1(\OO_Y,\mathcal{L}_{-m-1}).
\]
For $F\in\mathcal{P}_m^+$ we have an exact sequence
\[
0\to H^0(Y,F)\to H^1(Y,F\otimes \mathcal{L}_{-m-1})^{\oplus m}\to H^1(Y,F\otimes \mathcal{L}_{-m})^{\oplus m+1}.
\]
Hence $\tilde{\mathcal{P}}_m^+$ is equivalent to the module category of the algebra $A_m^+$ defined by the following quiver with relations:
\begin{figure}[htbp]
  \centering
  \includegraphics{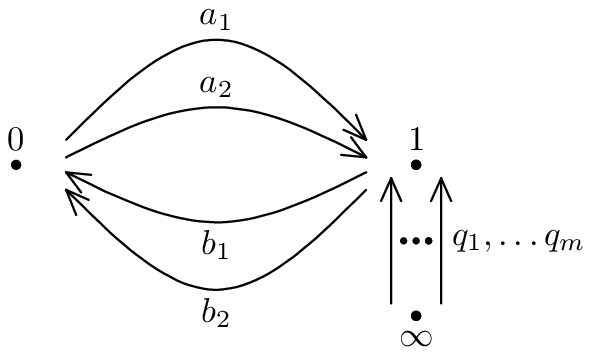}
  \caption{quiver $Q_m^+$}\label{quiver(Qm+)}
\end{figure}
the quiver $Q_m^+$ is given as in Figure \ref{quiver(Qm+)} and the following relations are added to the usual ones:
\begin{equation}\label{eq-rel+}
b_1q_1=0,\quad b_1q_{i+1}=b_2q_{i}\ (i=1,\ldots m-1),\quad b_2q_m=0.
\end{equation}

Similarly, $\tilde{\mathcal{P}}_m^-$ is equivalent to the module category of the algebra $A_m^-$ defined by the following quiver with relations:
\begin{figure}[htbp]
  \centering
  \includegraphics{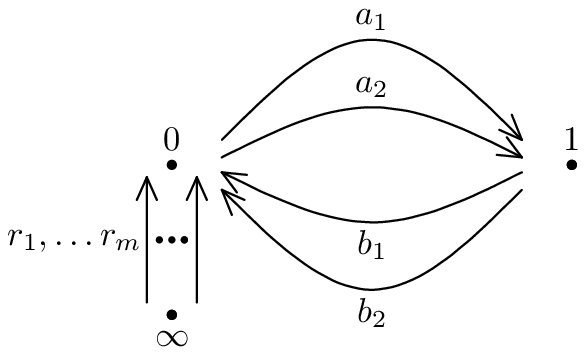}
  \caption{quiver $Q_m^-$}\label{quiver-m-}
\end{figure}
the quiver $Q_m^-$ is given as in Figure \ref{quiver-m-} and the following relations are added to the usual ones:
\begin{equation}\label{eq-rel-}
a_1r_{i+1}=a_2r_{i}\ (i=1,\ldots m-1).
\end{equation}

We can define $\zeta$-(semi)stability for $A_m^\pm$-modules as in Definition \ref{def-st}. 
In order to make it clear in what category we work, we use the notation ``$(\zeta,\tilde{\mathcal{P}}_m^\pm)$-(semi)stability''.
From now on, the $\zeta$-(semi)stability for modules of the original quiver $\tilde{Q}$ is written ``$(\zeta,\tpervc)$-(semi)stability''. 
We can construct the moduli spaces $\mathfrak{M}^{A_m^\pm}_{\zeta}(\mathrm{v}_0,\mathrm{v}_1)$
of $(\zeta,\tilde{\mathcal{P}}_m^\pm)$-semistable $A_m^\pm$-modules $V$ with $\dimv\,V=(\mathrm{v}_0,\mathrm{v}_1,1)$. 
As we will see in \S \ref{subsec-potential}, the relations derived from a potential, and hence we can also define a symmetric obstruction theory and virtual counting.
\end{NB3}

\subsection{Potentials}\label{subsec-potential}
\begin{NB3}we change the order of the subsections and rewrite\end{NB3}%
Let $\tilde{Q}_m^+$ be the quiver in Figure \ref{quiver(Qm+hat)} and $\omega_m^+$ be the following potential: 
\[
a_1b_1a_2b_2-a_1b_2a_2b_1
+p_1b_1q_1
+p_2(b_1q_2-b_2q_1)
+\cdots
+p_m(b_1q_m-b_2q_{m-1})
-p_{m+1}b_2q_m.
\]
Let $\tilde{A}_m^+$ be the algebra defined by the quiver with the potential $(\tilde{Q}_m^+,\omega_m^+)$ (see \cite{quiver-with-potentials}).
\begin{figure}[htbp]
  \centering
  \includegraphics{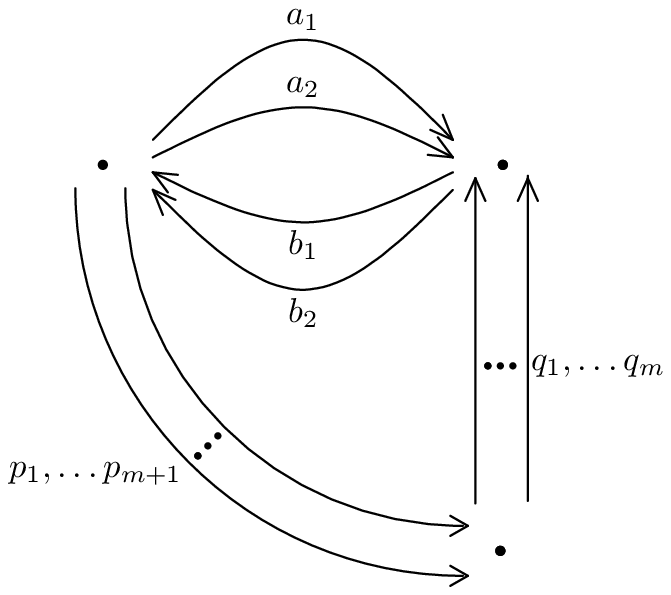}
  \caption{quiver $\tilde{Q}_m^+$}\label{quiver(Qm+hat)}
\end{figure}
\begin{lem}\label{lem-A-hat1}
\[
p_jb_{\varepsilon_1}a_{\varepsilon_2}b_{\varepsilon_3}\dots b_{\varepsilon_L}q_{j'}=0\quad (\varepsilon_l=1,2).
\]
\end{lem}
\begin{proof}
Assume $\varepsilon_L=1$ and ${j}\geq {j'}$. Then we have
\begin{align*}
p_{j}b_{\varepsilon_1}a_{\varepsilon_2}b_{\varepsilon_3}\dots b_{1}q_{j'}
&=
p_{j}b_{\varepsilon_1}a_{\varepsilon_2}b_{\varepsilon_3}\dots b_{2}q_{j'-1}\\
&=
p_{j}b_{2}a_{\varepsilon_2}b_{\varepsilon_1}\dots b_{\varepsilon_{L-2}}q_{j'-1}\\
&=
p_{j-1}b_{1}a_{\varepsilon_2}b_{\varepsilon_1}\dots b_{\varepsilon_{L-2}}q_{j'-1}\\
&=
p_{j-1}b_{\varepsilon_1}a_{\varepsilon_2}b_{\varepsilon_3}\dots b_{1}q_{j'-1}\\
&=\cdots\\
&=
p_{j-j'+1}b_{\varepsilon_1}a_{\varepsilon_2}b_{\varepsilon_3}\dots b_{1}q_{1}\\
&=0.
\end{align*}
We can show the claims for other cases in the same way.
\end{proof}
Fix an element $(\mathrm{v}_0,\mathrm{v}_1)\in(\mathbb{Z}_{\geq 0})^2$. 
For $\zeta=(\zeta_0,\zeta_1)$, we put $\theta_\zeta=(\zeta_0,\zeta_1,-\zeta_0\mathrm{v}_0-\zeta_1\mathrm{v}_1)$. 
Let $\mathfrak{M}^{\tilde{A}_m^+}_{\zeta}(\mathrm{v}_0,\mathrm{v}_1)$ denote the moduli space of 
$\theta_\zeta$-stable $\tilde{A}_m^+$-modules with dimension vectors $(\mathrm{v}_0,\mathrm{v}_1,1)$.
\begin{lem}\label{prop-A-hat2}
Let $V$ be an $\tilde{A}_m^+$-module with dimension vector $(\mathrm{v}_0,\mathrm{v}_1,1)$.
If $V$ is $\theta_{{\zeta}_{\mathrm{cyclic}}}$-stable, then $p_j=0$ for $j=1,\ldots m+1$.
\end{lem}
\begin{proof}
By the $\theta_{{\zeta}_{\mathrm{cyclic}}}$-stability, $V_0$ coincides with the union of the images of $b_{\varepsilon_1}a_{\varepsilon_2}b_{\varepsilon_3}\dots b_{\varepsilon_L}q_j$'s.
Thus the claim follows from Lemma \ref{lem-A-hat1}.
\end{proof}

\begin{prop}\label{cor-A-hat-3}
\[
\mathfrak{M}^{A_m^+}_{\zeta_{\mathrm{cyclic}}}(\mathrm{v}_0,\mathrm{v}_1)\simeq
\mathfrak{M}^{\tilde{A}_m^+}_{\zeta_{\mathrm{cyclic}}}(\mathrm{v}_0,\mathrm{v}_1).
\]
In particular, $\mathfrak{M}^{A_m^+}_{\zeta_{\mathrm{cyclic}}}(\mathrm{v}_0,\mathrm{v}_1)$ has a symmetric obstruction theory.
\end{prop}
\begin{proof}
The claim follows directly from Lemma \ref{prop-A-hat2} and the definition of the potential.
\end{proof}

Similarly, we define $\tilde{A}_m^-=(\tilde{Q}_m^-,\omega_m^-)$ by the quiver $\tilde{Q}_m^-$ in Figure \ref{quiver(Qm-hat)} and the following potential:
\[
\omega_m^-=a_1b_1a_2b_2-a_1b_2a_2b_1
+s_1(a_1r_2-a_2r_1)
+\cdots
+s_{m-1}(b_1r_m-b_2r_{m-1}).
\]

\begin{figure}[htbp]
  \centering
  \includegraphics{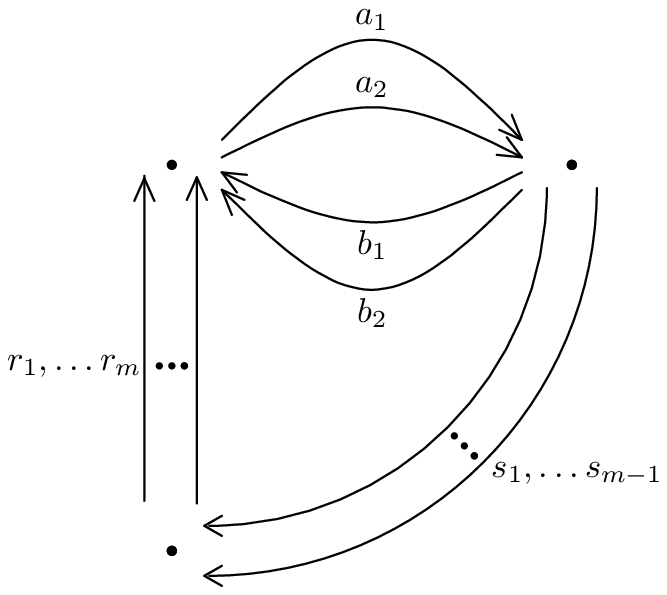}
  \caption{quiver $\tilde{Q}_m^-$}\label{quiver(Qm-hat)}
\end{figure}

We can prove the following claim in the same way:
\begin{prop}\label{cor-A-hat-4}
\[
\mathfrak{M}^{A_m^-}_{\zeta_{\mathrm{cyclic}}}(\mathrm{v}_0,\mathrm{v}_1)\simeq
\mathfrak{M}^{\tilde{A}_m^-}_{\zeta_{\mathrm{cyclic}}}(\mathrm{v}_0,\mathrm{v}_1).
\]
In particular, $\mathfrak{M}^{A_m^-}_{\zeta_{\mathrm{cyclic}}}(\mathrm{v}_0,\mathrm{v}_1)$ has a symmetric obstruction theory.
\end{prop}

\subsection{Moduli spaces}

\begin{lem}\label{lem-4.2}
Let $(F,s)\in\tpervc$ be a $(\zeta^{m,\pm},\tpervc)$-stable object, then $F\in {}^c{\mathcal{P}}_m^\pm$.
\end{lem}
\begin{proof}
Take a sufficiently large $c\in\R$ such that $\zeta_0+c,\zeta_1+c>0$ and set
$\zeta_{\mathrm{triv}}=(\zeta_0+c,\zeta_1+c)$ (see Figure \ref{fig:zeta2}).
Let
\[
(F,s)=\tilde{F}^0\supset\tilde{F}^1\supset\cdots\supset\tilde{F}^L\supset 0
\]
be a filtration of a $(\zeta^{m,+},\tpervc)$-stable object $(F,s)\in\tpervc$ by $(\theta_{\tilde{\zeta}_{\mathrm{triv}}},\tpervc)$-stable subquotients, which is given by combining \HN of $(F,s)$ and Jordan-H\"older filtrations of its factors. 
Since $\theta_{\tilde{\zeta}_{\mathrm{triv}}}(\tilde{F}^0/\tilde{F}^1)<0$, we have $(\tilde{F}^0/\tilde{F}^1)_\infty=\C$.
Moreover, we have $(\tilde{F}^0/\tilde{F}^1)_0=(\tilde{F}^0/\tilde{F}^1)_1=0$ from the $(\theta_{\tilde{\zeta}_{\mathrm{triv}}},\tpervc)$-stability of $\tilde{F}^0/\tilde{F}^1$. 
Note that we have
\[
\theta_{\tilde{\zeta}_{\mathrm{triv}}}(\mathrm{v}_0,\mathrm{v}_1,0)\leq 
\theta_{\tilde{\zeta}_{\mathrm{triv}}}(\mathrm{v}'_0,\mathrm{v}'_1,0)
\iff
\theta_{\tilde{\zeta}^{m,+}}(\mathrm{v}_0,\mathrm{v}_1,0)\leq 
\theta_{\tilde{\zeta}^{m,+}}(\mathrm{v}'_0,\mathrm{v}'_1,0)
\]
for any $\mathrm{v}_0,\mathrm{v}_1,\mathrm{v}'_0$ and $\mathrm{v}'_1$.
Since $\theta_{\tilde{\zeta}_{\mathrm{triv}}}(\tilde{F}^l/\tilde{F}^{l+1})\leq \theta_{\tilde{\zeta}_{\mathrm{triv}}}(\tilde{F}^L)$, we also have
\[
\theta_{\tilde{\zeta}^{m,+}}(\tilde{F}^l/\tilde{F}^{l+1})\leq \theta_{\tilde{\zeta}^{m,+}}(\tilde{F}^L)<0,
\]
where the last inequality is the consequence of $\zeta^{m,+}$-stability of $(F,s)$.
By the classification in \S \ref{subsec-classification}, a $(\theta_{\tilde{\zeta}_{\mathrm{triv}}},\tpervc)$-stable object $\tilde{F}$ with a $0$-dimensional framing such that $\theta_{\tilde{\zeta}^{m,+}}(\tilde{F})<0$ is isomorphic to $z_*\OO_{\CP^1}(m'-1)$ for some $1\leq m'\leq m$.
Thus we get a description of $F\in\pervc$ as successive extensions of $z_*\OO_{\CP^1}(m'-1)$'s ($1\leq m'\leq m$). 
Since $z_*\OO_{\CP^1}(m'-1)\in {\mathcal{P}}_m^+$ for $m'\leq m$, we have $F\in {\mathcal{P}}_m^+$.
We can show the claim for a $\zeta^{m,-}$-stable object in the same way.
\end{proof}

Let $\zeta_{\mathrm{cyclic}}$ be a stability parameter such that $(\zeta_{\mathrm{cyclic}})_0,(\zeta_{\mathrm{cyclic}})_1<0$.
\begin{lem}\label{lem-4.3}
Let $(F,s)\in\tpervc$ be a $(\zeta^{m,\pm},\tpervc)$-stable object, then $(F,s)\in {}^c\tilde{\mathcal{P}}_m^\pm$ is $(\zeta_{\mathrm{cyclic}},{}^c\tilde{\mathcal{P}}_m^\pm)$-stable.
\end{lem}
\begin{proof}
Note that the simple $A_m^+$-modules $S_0$ and $S_1$ correspond to $\OO_{\CP^1}(m)[-1]$ and $\OO_{\CP^1}(m-1)$ in $\mathcal{P}_m^+$ respectively.
For a $(\zeta^{m,+},\tpervc)$-stable object $(F,s)\in\tpervc$, as in Lemma \ref{lem-4.1}, it is enough to show that 
\begin{align*}
\Hom_{\tilde{\mathcal{P}}_m^+}((F,s),(\OO_x[-1],0,0))&=0,\\
\Hom_{\tilde{\mathcal{P}}_m^+}((F,s),(z_*\OO_{\CP^1}(m)[-1],0,0))&=0
\end{align*}
for any $x\in Y$ and
\[
\Hom_{\tilde{\mathcal{P}}_m^+}((F,s),(z_*\OO_{\CP^1}(m-1),0,0))=0.
\]
The first two equalities hold since $F,\,\OO_x, \OO_{\CP^1}(m)\in \pervc$.
For the third equation, we have
\begin{align*}
&\Hom_{\tilde{\mathcal{P}}_m^+}((F,s),(z_*\OO_{\CP^1}(m-1),0,0))\\
&=\{f\in \Hom_{\mathcal{P}_m^-}(F,z_*\OO_{\CP^1}(m-1))\mid f\circ s=0\}\\
&=\{f\in \Hom_{\pervc}(F,z_*\OO_{\CP^1}(m-1))\mid f\circ s=0\}\\
&=\Hom_{\tpervc}((F,s),(z_*\OO_{\CP^1}(m-1),0,0))\\
&=0,
\end{align*}
where the last equality follows from the $(\zeta^{m,+},\tpervc)$-stability of $(F,s)\in\tpervc$.
We can show the claim for a $(\zeta^{m,-},\tpervc)$-stable object in the same way.
\end{proof}

The chamber structure in the space of stability parameters on ${}^c\tilde{\mathcal{P}}_m^\pm$ is the same as that on $\tpervc$. 
The stable objects $\tilde{F}\in {}^c\tilde{\mathcal{P}}_m^\pm$ on a wall with $\tilde{F}_\infty=0$ is obtained from those for $\tpervc$ by applying $-\otimes\mathcal{L}_m[-1]$ (resp. $-\otimes\mathcal{L}_{-m-1}$).
Note that the parameter $\zeta^{m,\mp}$ for ${}^c\tilde{\mathcal{P}}_m^\pm$ is in the chamber between the walls $L^-_\mp(m)$ and $L^-_\mp(m+1)$ with stable objects $z_*\OO_{\CP^1}$ and $z_*\OO_{\CP^1}(-1)$ on them.

\begin{lem}\label{lem-4.4}
Let $(F,s)\in{}^c\tilde{\mathcal{P}}_m^\pm$ be a $(\zeta^{m,\mp},{}^c\tilde{\mathcal{P}}_m^\pm)$-stable object, then  $F=0$.
\end{lem}
\begin{proof}
Take \HN of a $(\zeta^{m,-},{}^c\tilde{\mathcal{P}}_m^+)$-stable object $(F,s)\in{}^c\tilde{\mathcal{P}}_m^+$ and Jordan-H\"older filtrations of its factors with respect to the $\theta_{\tilde{\zeta}_{\mathrm{triv}}}$-stability. 
Then, as in the proof of Lemma \ref{lem-4.2}, we get a description of $F\in{\mathcal{P}}_m^+$ as successive extensions of $z_*\OO_{\CP^1}(-m')$'s ($m'\geq 1$), $\OO_x[-1]$'s ($x\in Y$) and $z_*\OO_{\CP^1}(m')[-1]$'s ($m'\geq m$).
So we have $H^0(Y,F)=0$ and hence $s=0$. 
Then the stability requires that $F=0$. 
We can show the claim for a $(\zeta^{m,+},{}^c\tilde{\mathcal{P}}_m^-)$-stable object in ${}^c\tilde{\mathcal{P}}_m^-$ similarly.
\end{proof}

\begin{lem}\label{lem-4.5}
Let $(F,s)\in{}^c\tilde{\mathcal{P}}_m^\pm$ be a $(\zeta_{\mathrm{cyclic}},{}^c\tilde{\mathcal{P}}_m^\pm)$-stable object, then $F\in \pervc$.
\end{lem}
\begin{proof}
Take \HN of $(F,s)\in{}^c\tilde{\mathcal{P}}_m^+$ and Jordan-H\"older filtrations of its factors with respect to the $\theta_{\tilde{\zeta}^{m,-}}$-stability. 
Then, by Lemma \ref{lem-4.4}, we get a description of $F\in{}^c{\mathcal{P}}_m^+$ as successive extensions of $z_*\OO_{\CP^1}(m'-1)$'s ($1\leq m'\leq m$) and hence we have $F\in \pervc$.
We can show the claim for an object in ${}^c\tilde{\mathcal{P}}_m^-$ similarly.
\end{proof}

\begin{lem}\label{lem-4.6}
Let $(F,s)\in{}^c\tilde{\mathcal{P}}_m^\pm$ be a $(\zeta_{\mathrm{cyclic}},{}^c\tilde{\mathcal{P}}_m^\pm)$-stable object, then $(F,s)\in \tpervc$ is $\zeta^{m,\pm}$-stable.
\end{lem}
\begin{proof}
For a $(\zeta_{\mathrm{cyclic}},{}^c\tilde{\mathcal{P}}_m^+)$-stable object $(F,s)\in{}^c\tilde{\mathcal{P}}_m^+$, 
as in the proof of Lemma \ref{lem-4.3}, we have
\begin{align*}
&\Hom_{\tpervc}((F,s),(z_*\OO_{\CP^1}(m-1),0,0))\\
&=\Hom_{\tilde{\mathcal{P}}_m^+}((F,s),(z_*\OO_{\CP^1}(m-1),0,0))=0.
\end{align*}
As we claimed in the proof of Lemma \ref{lem-4.5}, $F$ is described as a successive extensions of $z_*\OO_{\CP^1}(m'-1)$'s ($1\leq m'\leq m$). 
So we have 
\begin{align*}
\Hom_{\tpervc}((E,0,0),(F,s))
=\Hom_{\pervc}(E,F)
=0
\end{align*}
for $E=z_*\OO_{\CP^1}(m)$ or $E=\OO_x$ ($x\in Y$).
By Lemma \ref{lem-4.1}, $(F,s)\in \tpervc$ is $(\zeta^{m,+},\tpervc)$-stable.
We can show the claim for an object in ${}^c\tilde{\mathcal{P}}_m^-$ similarly.
\end{proof}

\begin{thm}\label{thm-moduli}
\begin{align*}
\mathfrak{M}_{\zeta^{m,+}}(\mathrm{v}_0,\mathrm{v}_1)&\simeq \mathfrak{M}^{A_m^+}_{\zeta_{\mathrm{cyclic}}}((m-1)\mathrm{v}_0+(-m)\mathrm{v}_1,m\mathrm{v}_0+(-m-1)\mathrm{v}_1),\\
\mathfrak{M}_{\zeta^{m,-}}(\mathrm{v}_0,\mathrm{v}_1)&\simeq \mathfrak{M}^{A_m^-}_{\zeta_{\mathrm{cyclic}}}(m\mathrm{v}_0+(-m+1)\mathrm{v}_1,(m+1)\mathrm{v}_0+(-m)\mathrm{v}_1).
\end{align*}
\end{thm}

\subsection{Fixed points}
\begin{NB3}rewritten\end{NB3}%
The readers may refer \cite{szendroi-ncdt} and \cite{young-conifold} for the definition of ``pyramid partitions with length $m$'' (see Figure \ref{erc+}), 
and \cite{chuang-jafferis} for the definition of ``finite type pyramid partitions with length $m$'' (see Figure \ref{erc}).
\begin{figure}[htbp]
  \centering
  \includegraphics{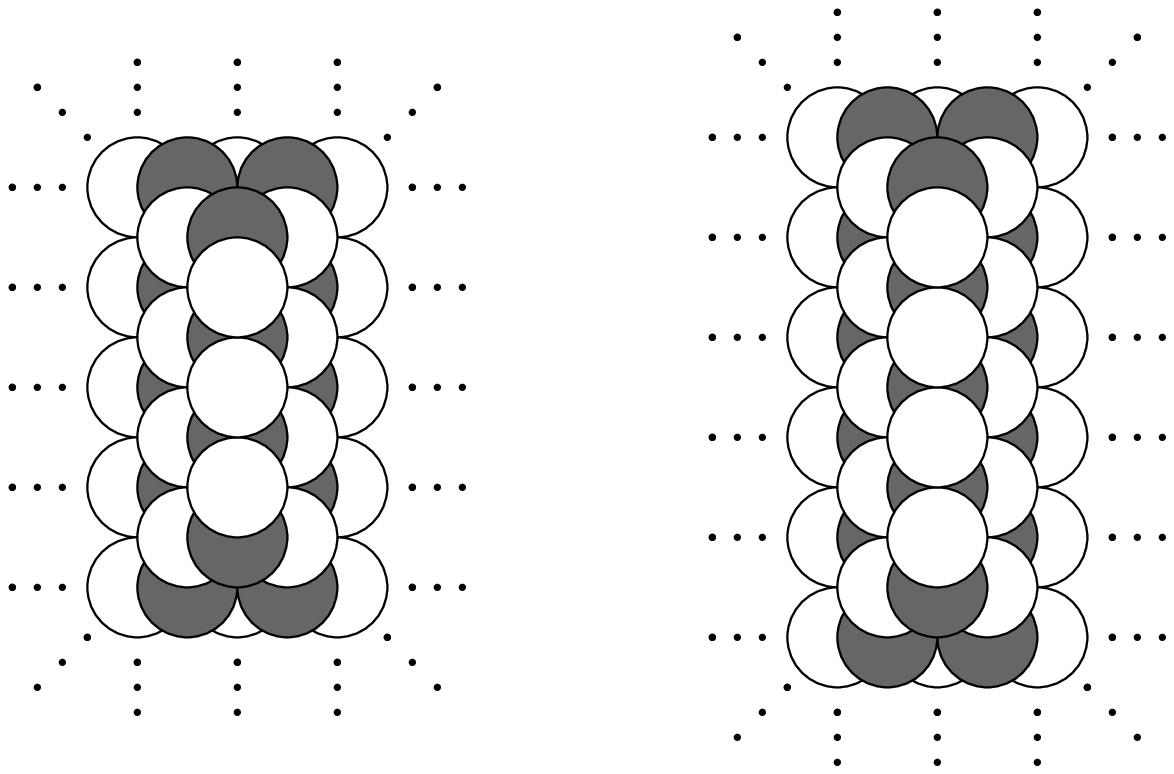}
  \caption{the empty room configurations for pyramid partitions with length $3$ and with length $4$}\label{erc+}
\end{figure}
\begin{figure}[htbp]
  \centering
  \includegraphics{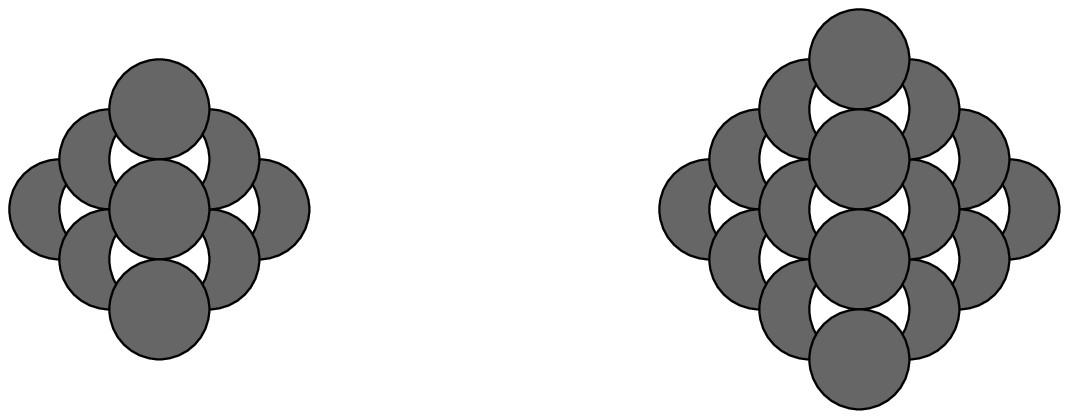}
  \caption{the empty room configurations for finite type pyramid partitions with length $3$ and with length $4$}\label{erc}
\end{figure}
\begin{prop}\label{prop-parameterization}
\begin{enumerate}
\item[(1)] The set of $T'$-fixed closed points 
\[
\mathfrak{M}_{\zeta^{m,+}}(\mathrm{v}_0,\mathrm{v}_1)^{T'}=\mathfrak{M}^{A_m^+}_{\zeta_{\mathrm{cyclic}}}((m-1)\mathrm{v}_0+(-m)\mathrm{v}_1,m\mathrm{v}_0+(-m-1)\mathrm{v}_1)^{T'}
\]
is isolated and parameterized by finite type pyramid partitions with length $m$ and with $(m-1)\mathrm{v}_0+(-m)\mathrm{v}_1$ white stones and $m\mathrm{v}_0+(-m-1)\mathrm{v}_1$ black stones.
\item[(2)] The set of ${T'}$-fixed closed points 
\[\mathfrak{M}_{\zeta^{m,-}}(\mathrm{v}_0,\mathrm{v}_1)^{T'}=\mathfrak{M}^{A_m^-}_{\zeta_{\mathrm{cyclic}}}(m\mathrm{v}_0+(-m+1)\mathrm{v}_1,(m+1)\mathrm{v}_0+(-m)\mathrm{v}_1)^{T'}
\]
is isolated and parameterized by pyramid partitions with length $m$ and with $m\mathrm{v}_0+(-m+1)\mathrm{v}_1$ white stones and $(m+1)\mathrm{v}_0+(-m)\mathrm{v}_1$ black stones.
\end{enumerate}
\end{prop}
\begin{proof}
Recall that $P$ denotes the kernel of the canonical map $P_\infty\to S_\infty$.
The $A$-module $P$ has the canonical $T$-weight decomposition such that each weight space is $1$-dimensional and parameterized by the empty room configuration for finite type pyramid partitions with length $m$ (resp. for pyramids partition with length $m$).

We put $c:=b_2a_2b_1a_1+a_2b_2a_1b_1(=xy=zw)\in Z(A)\subset A$. 
Here $Z(A)$ is the center of $A$, which is isomorphic to the coordinate ring $\C[x,y,z,w]/(xy-zw)$ of the conifold. 
Let $v_B\in P$ be a $T$-weight vector corresponding to a stone $B$ in the empty room configuration. 
Then $c\cdot v_B$ (unless $=0$ in the case (2)) is the $T$-weight vector corresponding to the stone just behind $B$. 
Any $T'$-weight space of $P$ is described as $\C[c]\cdot v_B$ for some stone $B$ in the empty room configuration.

Let $(F,s)\in{}^c\tilde{\mca{P}}_m^+$ (resp. $\in{}^c\tilde{\mca{P}}_m^-$) be a $\zeta_{\mathrm{cyclic}}$-stable object, then $F$ is a quotient of $P$ as an $A$-module. 
Any $T'$-weight space of $F$ is described as $I\cdot v_B$ for some stone $b$ and for some ideal $I\in \C[c]$. 
Assume that $(F,s)$ is $T'$-invariant. 
Then $P/F$ must be supported at the singularity $0\in \mathrm{Spec}Z(A)$ and so $I$ must be a monomial ideal.
Thus the claims follow.
\end{proof}

Let $N^m_{\mathrm{pyramid}}(n_0,n_1)$ (resp. $N^m_{\mathrm{fin}\text{-}\mathrm{pyramid}}(n_0,n_1)$) denote the number of pyramid partitions (resp. finite type pyramid partitions) with length $m$ and with $n_0$ white stones and $n_1$ black stones.
We encode them into the generating functions
\begin{align*}
\mathcal{Z}^m_{\mathrm{pyramid}}(\mathbf{p})&:=\sum_{(n_0,n_1)\in(\Z_{\geq 0})^2}N^m_{\mathrm{pyramid}}(n_0,n_1)\cdot p_0^{n_0}p_1^{n_1},\\
\mathcal{Z}^m_{\mathrm{fin}\text{-}\mathrm{pyramid}}(\mathbf{p})&:=\sum_{(n_0,n_1)\in(\Z_{\geq 0})^2}N^m_{\mathrm{fin}\text{-}\mathrm{pyramid}}(n_0,n_1)\cdot p_0^{n_0}p_1^{n_1},
\end{align*}
where $p_0$ and $p_1$ are formal variables.
\begin{thm}\label{thm-phyramid}
\begin{align*}
&\mathcal{Z}^m_{\mathrm{fin}\text{-}\mathrm{pyramid}}(\mathbf{p})=
\prod_{m'=1}^m \left(1+p_0^{m-m'}p_1^{m-m'+1}\right)^{m'},\\
&\mathcal{Z}^m_{\mathrm{pyramid}}(\mathbf{p})=\\
&\left(\prod_{m'\geq 1}^\infty (1-p_0^{m'}p_1^{m'})^{-m'}\right)^2
\left(\prod_{m'\geq 1}^\infty \left(1+p_0^{m+m'-1}q_1^{m+m'}\right)^{m'}\right)
\left(\prod_{m'\geq m}^\infty (1+q_0^{-m+m'+1}q_1^{-m+m'})^m\right).
\end{align*}
\end{thm}
\begin{proof}
The claim follows from Theorem \ref{wallcrossing}, Proposition \ref{prop-parameterization} and the Behrend-Bryan's formula \eqref{eq-bb}.
\end{proof}

\begin{NB3}
Let $S_\infty$, $P_\infty$ be the simple and indecomposable projective $A_m^+$-modules corresponding to the extended vertex and set $P:={P}_\infty/S_\infty$.
The $A$-module $P$ has the canonical $T$-weight decomposition such that each weight space is $1$-dimensional and parameterized by``the empty room configuration for the finite type pyramid partition with length $m$'' under the notation in \cite{chuang-jafferis} (see Figure \ref{erc}).

\begin{figure}[htbp]
  \centering
  \includegraphics{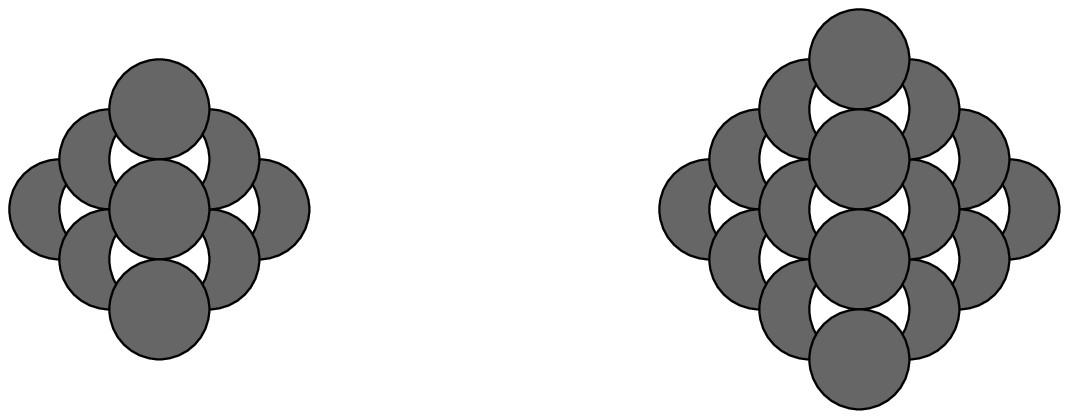}
  \caption{the empty room configurations for finite type pyramid partitions with length $3$ and with length $4$}\label{erc}
\end{figure}

\begin{figure}[htbp]
  \centering
  \includegraphics{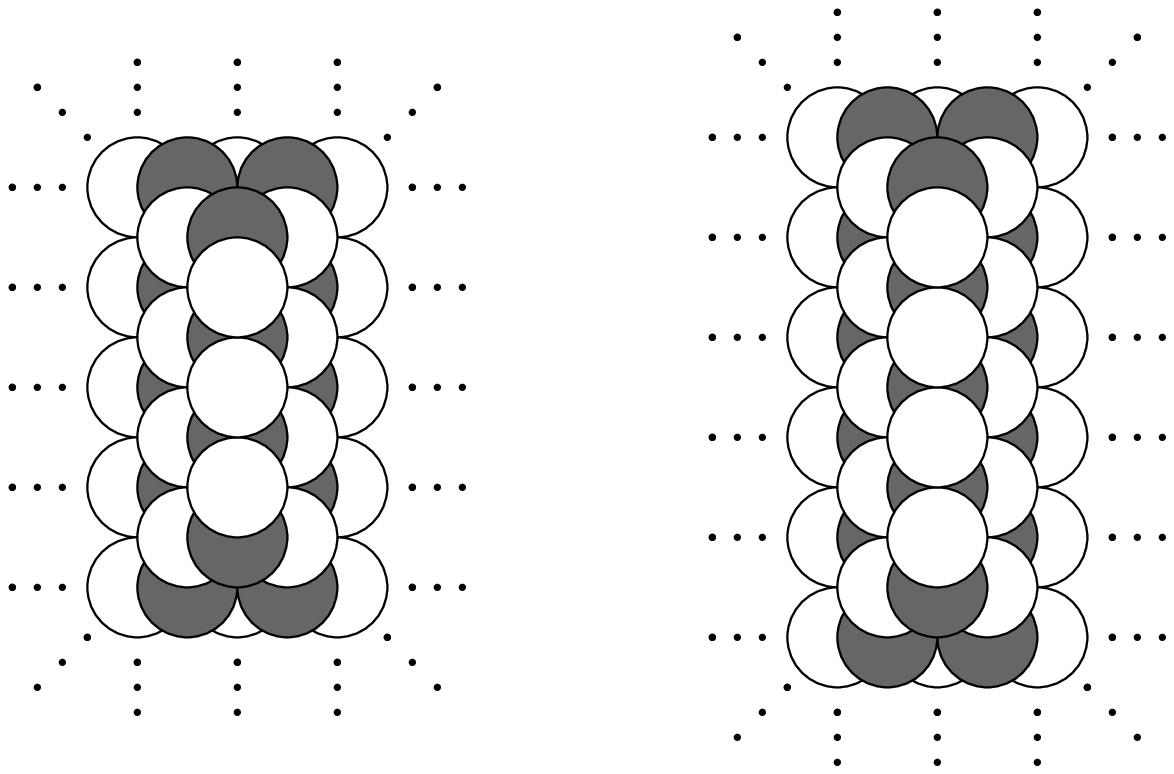}
  \caption{the empty room configurations for pyramid partitions with length $3$ and with length $4$}\label{erc+}
\end{figure}

Let $(F,s)\in\tilde{P}_m^+$ be a $\zeta_{\mathrm{cyclic}}$-stable object.
Then $F$ is a quotient of $P$ as an $A_m^+$-module. 
Thus the $T$-fixed point set 
\[
\mathfrak{M}_{\zeta^{m,+}}(\mathrm{v}_0,\mathrm{v}_1)^T\simeq \mathfrak{M}^{A_m^+}_{\zeta_{\mathrm{cyclic}}}((m-1)\mathrm{v}_0+(-m)\mathrm{v}_1,m\mathrm{v}_0+(-m-1)\mathrm{v}_1)^T
\]
is isolated and parameterized by ``finite type pyramid partitions with length $m$''. 
In particular, we get the formula for the generating function of the number of finite type pyramid partitions with length $m$, which is conjectured in \cite{chuang-jafferis}, as an application of the wall-crossing formula. 

Similarly, 
\[
\mathfrak{M}_{\zeta^{m,-}}(\mathrm{v}_0,\mathrm{v}_1)^T\simeq \mathfrak{M}^{A_m^-}_{\zeta_{\mathrm{cyclic}}}(m\mathrm{v}_0+(-m+1)\mathrm{v}_1,(m+1)\mathrm{v}_0+(-m)\mathrm{v}_1)^T
\]
is isolated and parameterized by ``pyramid partitions with length $m$'' (see Figure \ref{erc+}). 
In particular, we can recover Young's formula (\cite{young-conifold}) for the generating function of the number of pyramid partitions with length $m$.
\end{NB3}

\begin{NB3}
\subsection{Potentials}\label{subsec-potential}
Let $\hat{Q}_m^+$ be the quiver in Figure \ref{quiver(Qm+hat)} and $\omega_m^+$ be the following potential: 
\[
a_1b_1a_2b_2-a_1b_2a_2b_1
+p_1b_1q_1
+p_2(b_1q_2-b_2q_1)
+\cdots
+p_m(b_1q_m-b_2q_{m-1})
-p_{m+1}b_2q_m.
\]
Note that the quiver with the potential $\hat{A}_m^+=(\hat{Q}_m^+,\omega_m^+)$ has relations \eqref{eq-rel+} and $b_1a_ib_2=b_2a_ib_1$ ($i=0,1$). 
\begin{figure}[htbp]
  \centering
  \includegraphics{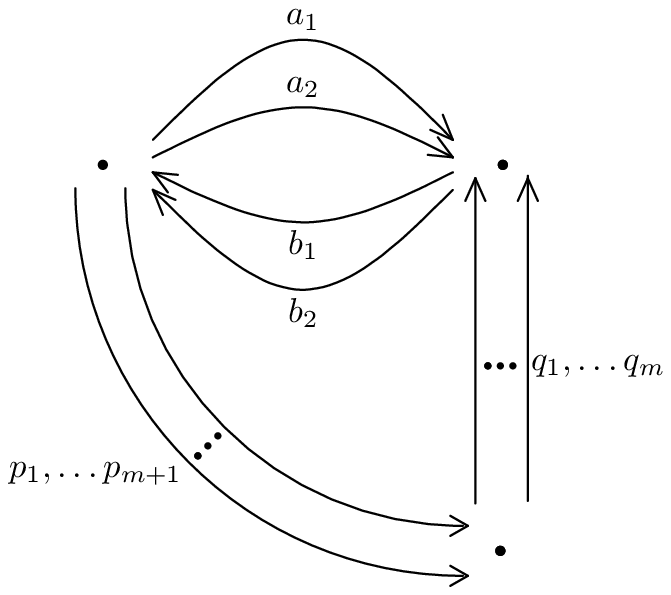}
  \caption{quiver $\hat{Q}_m^+$}\label{quiver(Qm+hat)}
\end{figure}

\begin{lem}\label{lem-A-hat1}
\[
p_ib_{\varepsilon_1}a_{\varepsilon_2}b_{\varepsilon_3}\dots b_{\varepsilon_k}q_j=0\quad (\varepsilon_l=1,2).
\]
\end{lem}
\begin{proof}
Assume $\varepsilon_k=1$ and $i\geq j$. Then we have
\begin{align*}
p_ib_{\varepsilon_1}a_{\varepsilon_2}b_{\varepsilon_3}\dots b_{1}q_j
&=
p_ib_{\varepsilon_1}a_{\varepsilon_2}b_{\varepsilon_3}\dots b_{2}q_{j-1}\\
&=
p_ib_{2}a_{\varepsilon_2}b_{\varepsilon_1}\dots b_{\varepsilon_{k-2}}q_{j-1}\\
&=
p_{i-1}b_{1}a_{\varepsilon_2}b_{\varepsilon_1}\dots b_{\varepsilon_{k-2}}q_{j-1}\\
&=
p_{i-1}b_{\varepsilon_1}a_{\varepsilon_2}b_{\varepsilon_3}\dots b_{1}q_{j-1}\\
&=\cdots\\
&=
p_{i-1}b_{\varepsilon_1}a_{\varepsilon_2}b_{\varepsilon_3}\dots b_{1}q_{1}\\
&=0.
\end{align*}
We can show the claims for other cases in the same way.
\end{proof}

\begin{prop}\label{prop-A-hat2}
Let $V$ be an $\hat{A}_m^+$-module which is $\theta_{\tilde{\zeta}_{\mathrm{cyclic}}}$-stable.
Then $p_i=0$ for $i=1,\ldots m+1$.  
\end{prop}
\begin{proof}
By the $\theta_{\tilde{\zeta}_{\mathrm{cyclic}}}$-stability, $V_0$ coincides with the union of the images of $b_{\varepsilon_1}a_{\varepsilon_2}b_{\varepsilon_3}\dots b_{\varepsilon_k}q_j$'s.
Thus the claim follows from Lemma \ref{lem-A-hat1}.
\end{proof}

\begin{cor}\label{cor-A-hat-3}
\[
\mathfrak{M}^{A_m^+}_{\zeta_{\mathrm{cyclic}}}(\mathrm{v}_0,\mathrm{v}_1)\simeq
\mathfrak{M}^{\hat{A}_m^+}_{\zeta_{\mathrm{cyclic}}}(\mathrm{v}_0,\mathrm{v}_1).
\]
In particular, $\mathfrak{M}^{A_m^+}_{\zeta_{\mathrm{cyclic}}}(\mathrm{v}_0,\mathrm{v}_1)$ has a symmetric obstruction theory.
\end{cor}

Similarly, we define $A_m^+=(\hat{Q}_m^-,\omega_m^-)$ by the quiver $\hat{Q}_m^-$ in Figure \ref{quiver(Qm-hat)} and the following potential:
\[
\omega_m^+=a_1b_1a_2b_2-a_1b_2a_2b_1
+s_1(a_1r_2-a_2r_1)
+\cdots
+s_{m-1}(b_1r_m-b_2r_{m-1}).
\]
Then we have
\[
\mathfrak{M}^{A_m^-}_{\zeta_{\mathrm{cyclic}}}(\mathrm{v}_0,\mathrm{v}_1)\simeq
\mathfrak{M}^{\hat{A}_m^-}_{\zeta_{\mathrm{cyclic}}}(\mathrm{v}_0,\mathrm{v}_1).
\]

\begin{figure}[htbp]
  \centering
  \includegraphics{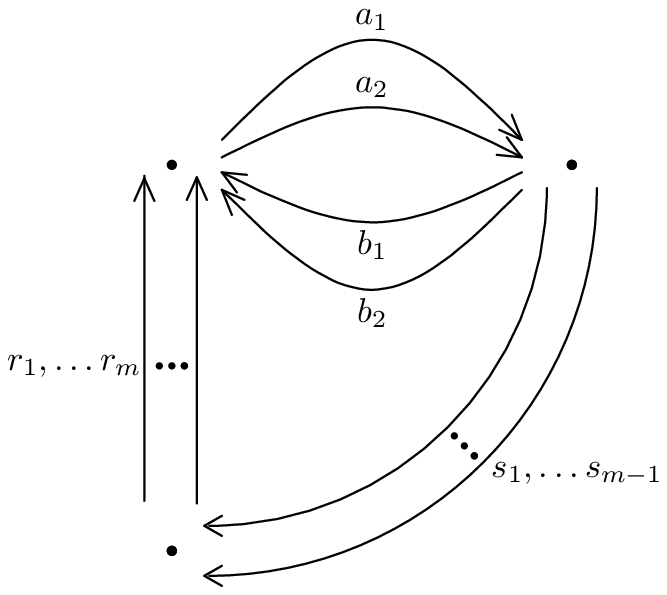}
  \caption{quiver $\hat{Q}_m^-$}\label{quiver(Qm-hat)}
\end{figure}
\end{NB3}

\subsection{Zariski tangent spaces at the fixed points}
\begin{NB3}added\end{NB3}%

\begin{lem}\label{lem-ext^2}
Let $V$ be an $A$-module such that 
\[
\left(\Phi^\pm_m\right)^{-1}(V)\in \mathcal{P}_m^\pm\cap \pervc.
\]
Then we have 
\begin{equation}\label{eq-ext^2}
\Ext^2_{A_m^\pm}(S_\infty,V)=0.
\end{equation}
\end{lem}
\begin{proof}
Note that we have
\[
\Ext^2_{A_m^\pm}(S_\infty,V)\simeq \Ext^1_{A_m^\pm}(P,V)\simeq \Ext^1_{A}(P,V)\simeq \Ext^1_{Y}\left(\OO_Y,\left(\Phi^\pm_m\right)^{-1}(V)\right).
\]
The last one vanishes by the assumption $\left(\Phi^\pm_m\right)^{-1}(V)\in \pervc$ .
\end{proof}

\begin{prop}\label{prop-ext^1}
Let $\tilde{V}$ be a $\zeta_{\mathrm{cyclic}}$-stable $A_m^\pm$-module with $V_\infty\simeq \C$ and $V$ be the kernel of the natural map $\V\to S_\infty$.
Then we have
\[
\dim \Ext^1_{A_m^\pm}(\tilde{V},\tilde{V})=
\dim \Ext^1_{A}({V},{V})-\dim \Hom_{A}({V},{V})+
\dim \Hom_Y (\OO_Y,F).
\]
\end{prop}
\begin{proof}
Applying the functor $\Hom_{A_m^\pm}(\tilde{V},-)$ for the
short exact sequence
\begin{equation}\label{eq-short-exact}
0\to V\to \tilde{V}\to S_\infty\to 0
\end{equation}
we have the following exact sequence:
\begin{equation*}
\xymatrix@R=.8pc@C=.9pc{
	& \Hom_{A_m^\pm}(\tilde{V},\tilde{V}) \ar[r]
	& \Hom_{A_m^\pm}(\tilde{V},S_\infty) \ar[lld]
\\
             \Ext^1_{A_m^\pm}(\tilde{V},V) \ar[r]
           & \Ext^1_{A_m^\pm}(\tilde{V},\tilde{V}) \ar[r] 
           & \Ext^1_{A_m^\pm}(\tilde{V},S_\infty).
}
\end{equation*}
Note that 
\[
\Hom_{A_m^\pm}(\tilde{V},S_\infty)\simeq (V_\infty)^*\simeq \C
\]
and this is spanned by the image of $\mathrm{id}_{\tilde{V}}\in\Hom_{A_m^\pm}(\tilde{V},\tilde{V})$. Thus the map in the upper line is surjective.
It is clear that any exact sequence
\[
0\to S_\infty \to * \to \tilde{V}\to 0
\]
of $A_m^\pm$-modules splits, that is, $\Ext^1_{A_m^\pm}(\tilde{V},S_\infty)=0$. Then we have
\[
\Ext^1_{A_m^\pm}(\tilde{V},V)\simeq \Ext^1_{A_m^\pm}(\tilde{V},\tilde{V}).
\]
On the other hand, applying the functor $\Hom_{A_m^\pm}(-,V)$ to the short exact sequence \eqref{eq-short-exact}, we have the following exact sequence:
\begin{equation*}
\xymatrix@R=.8pc@C=.9pc{
	& \Hom_{A_m^\pm}(\tilde{V},V) \ar[r]
	& \Hom_{A_m^\pm}(V,V) \ar[lld]
\\
		\Ext^1_{A_m^\pm}(S_\infty,V) \ar[r]
	&	\Ext^1_{A_m^\pm}(\tilde{V},V) \ar[r]
	&	\Ext^1_{A_m^\pm}(V,V) \ar[lld]
\\
		\Ext^2_{A_m^\pm}(S_\infty,V).	&	&
}
\end{equation*}
Since $\tilde{V}$ is $\zeta_{\mathrm{cyclic}}$-stable, $\Hom_{A_m^\pm}(\tilde{V},V)=0$.
Moreover, $V$ satisfies the assumption of Lemma \ref{lem-ext^2} by Lemma \ref{lem-4.5}.
Hence the claim follows.
\end{proof}

\begin{lem}\label{lem-alt}
\[
\dim \Ext^1_{A}({V},{V})-\dim \Hom_{A}({V},{V})\equiv \dim V_0 + \dim V_1\quad (\mathrm{mod}\ 2).
\]
\end{lem}
\begin{proof}
It is shown in the proof of \cite[Theorem 7.1]{ncdt-brane} that
\begin{align*}
&\dim \Ext^1_{A}({V},{V})-\dim \Hom_{A}({V},{V})\\
\equiv &\sum_{i\in Q_0}(\dim V_i)^2-\sum_{h\in Q_1}\left(\dim V_{\mathrm{in}(h)}\right)\left(\dim V_{\mathrm{out}(h)}\right)\quad (\mathrm{mod}\, 2).
\end{align*}
\end{proof}

\begin{cor}\label{cor-zariski}
For a $T'$-invariant closed point $x\in \mathfrak{M}_{\zeta^{m,\pm}}(\mathrm{v}_0,\mathrm{v}_1)^{T'}$, the parity of the dimension of the Zariski tangent space to $\mathfrak{M}_{\zeta^{m,\pm}}(\mathrm{v}_0,\mathrm{v}_1)$ at $x$ equals to the parity of $\mathrm{v}_1$.
\end{cor}
\begin{proof}
Under the isomorphism in Theorem \ref{thm-moduli}, $x$ can be regarded as a closed point in the moduli space of $A_m^\pm$-modules. 
Let $\tilde{V}$ denote the $A^\pm_m$-module corresponding to $x$.
Then the Zariski tangent space is isomorphic to $\Ext^1_{A_m^\pm}(\tilde{V},\tilde{V})$. 
Thus the claim is a consequence of Proposition \ref{prop-ext^1} and Lemma \ref{lem-alt}.
\end{proof}

\begin{prop}\label{prop-isolated}
For each ${T'}$-fixed closed point $x\in\M{\zeta}{\vv}^{T'}$, the Zariski tangent space to $\M{\zeta}{\vv}$ at $x$ has no non-trivial ${T'}$-invariant subspace.
\end{prop}
\begin{proof}
Recall that we have the following exact sequence:
\[
0 \to 
\Hom_{A_m^\pm}(V,V) \to
\Ext^1_{A_m^\pm}(S_\infty,V) \to
\Ext^1_{A_m^\pm}(\tilde{V},V) \to
\Ext^1_{A_m^\pm}(V,V) \to 0.
\]
So, it is enough to show that neither 
\begin{equation}\label{eq-left}
\coker\left(\Hom_{A_m^\pm}(V,V) \to
\Ext^1_{A_m^\pm}(S_\infty,V)\right)
\end{equation}
nor 
\begin{equation*}
\Ext^1_{A_m^\pm}(V,V)
\end{equation*}
have non-trivial ${T'}$-invariant subspace.

First, we have the exact sequence
\[
0\to \Hom_{A}(V,V) \to \Hom_{A}(P,V)\to \Hom_{A}\left(\ker(P\to V),V\right)\to \Ext^1_{A}(V,V)
\]
of which the first map coincides with the map in \eqref{eq-left} under the isomorphisms
\[
\Hom_{A_m^\pm}\left(\ker(P\to V),V\right)=\Hom_{A}\left(\ker(P\to V),V\right), \Ext^1_{A_m^\pm}(S_\infty,V)= \Hom_{A_m^\pm}(P,V)=\Hom_{A}(P,V).
\]
Here $P$ denotes the kernel of the canonical map $P_\infty\to S_\infty$ as before.

The claim follows from Lemma \ref{lem_T'1} and Lemma \ref{lem_T'2}.
\end{proof}
\begin{lem}\label{lem_T'1}
\begin{itemize}
We put $I:=\ker(P\to V)$.
\item[(1)] In the case $\zeta=\zeta^{m,+}$, there is no non-trivial ${T'}$-invariant subspace in 
$\ker\left(\Hom_{A}\left(I,V\right)
 \to
\Ext^1_{A}(V,V)\right)$.
\item[(2)] In the case $\zeta=\zeta^{m,-}$, there is no non-trivial ${T'}$-invariant subspace in $\Hom_{A}\left(I,V\right)$.
\end{itemize}
\end{lem}
\begin{proof}
In the proof of Proposition \ref{prop-isolated}, we see that
\begin{itemize}
\item every $T$-weight vector in $P$, $V$ and $I$ is associated to a stone in the empty room configuration, and
\item every $T'$-weight space in $P$ is described as $\C[c]\cdot v_B$ for some stone $B$.
\end{itemize}
Suppose we have a nonzero ${T'}$-invariant element $\phi\in \Hom_{A}\left(I,V\right)$. 
We may assume that there is a positive integer $n$ such that
\[
\phi(v_B)=\begin{cases}
v_{B'} & \text{if $\exists\, [v_{B'}]\in V$ such that $c^n\cdot v_{B'}=v_B\in P$,}\\
0 & \text{otherwise}
\end{cases}
\]
for $v_B\in I$.
\begin{itemize}
\item[(1)]
In the case $\zeta=\zeta^{m,+}$, we will show that the image of $\phi$ under the map $\Hom_{A}\left(I,V\right) \to \Ext^1_{A}(V,V)$ gives a non-trivial extension. 
Since $\phi$ is ${T'}$-invariant it gives self-extensions of $T'$-weight spaces. 
Take a $T'$-weight space $W$ of $P$ such that the restriction of $\phi$ to $W\cap I$ is nontrivial.
Recall that $W$ admits a $\C[c]$-module structere.
It is enough to show that the self extension of $[W]\in V$ is not trivial as a $\C[c]$-module. 
We have
\[
W\simeq \C[c]/(c^m),\quad  W\cap I\simeq (c^{k})/(c^m)
\]
for some positive integers $m$ and $k$ such that $m-n\geq k\geq n$. 
Then the ${T'}$-weight space in the extension associated to $\phi$ is isomorphic to $\C[c]/(c^{k+n})\oplus \C[c]/(c^{k-n})$, which is not isomorphic to $\C[c]/(c^{k})\oplus \C[c]/(c^{k})$.
\item[(2)]
In the case $\zeta=\zeta^{m,-}$, take a stone $B$ from the ridge of the empty configuraion such that $v_B\notin I$.
Let $\alpha$ be the positive integer such that such that $(b_2a_2)^{\alpha-1}\cdot v_{B}\notin I$ and $(b_2a_2)^{\alpha}\cdot v_{B}\in I$ and 
put $v_1:=(b_2a_2)^{\alpha}\cdot v_{B}$. 
Note that $\phi(v_1)=0$.

Let $\beta$ be the positive integer such that $(b_1a_1)^{\beta-1}\cdot(b_2a_2)^{\alpha-1}\cdot v_{B}\notin I$ and $(b_1a_1)^\beta\cdot(b_2a_2)^{\alpha-1}\cdot v_{B}\in I$.

We put
\[
v_2:=c^n\cdot (b_1a_1)^{\beta-1}\cdot(b_2a_2)^{\alpha-1}\cdot v_{B}=c^{n-1}\cdot (b_1a_1)^{\beta}\cdot v_1\in I.
\]
Then we have
\[
0\neq [(b_1a_1)^{\beta-1}\cdot(b_2a_2)^{\alpha-1}\cdot v_{B}]=\phi(v_2)
=c^{n-1}\cdot (b_1a_1)^{\beta}\cdot \phi(v_1)=0.
\]
This is a contradiction.
\end{itemize}
\end{proof}
\begin{lem}\label{lem_T'2}
There is no non-trivial ${T'}$-invariant subspace in $\Ext^1_{A_m^\pm}(V,V)$.
\end{lem}
\begin{proof}
Using the Koszul complex (see the proof of Lemma \ref{lem-Koszul}), $\Ext^1_{A_m^\pm}(V,V)$ is given as the second cohomology of the following complex:
\begin{align*}
&\bigoplus_{i\in Q_0} \Hom\left(V_i,V_i\right)
 \to \bigoplus_{b\in Q_1}\Hom\left(V_{\mr{out}(b)},V_{\mr{in}(b)}\right)\\
& \to \bigoplus_{a\in Q_1}\Hom\left(V_{\mr{in}(a)},V_{\mr{out}(a)}\right)
 \to \bigoplus_{i\in Q_0} \Hom\left(V_i,V_i\right).
\end{align*}
For $P\in\M{\zeta}{\vv}^{T'}$, the second term has no non-trivial ${T'}$-invariant subspace, and hence neither does $\Ext^1_{A_m^\pm}(V,V)$.
\end{proof}
\begin{NB}
\begin{lem}\label{lem_T'}
There is no non-trivial ${T'}$-invariant subspace in $\Hom_{A}\left(\ker(P\to V),V\right)$.
\end{lem}
\begin{proof}
We put $I:=\ker(P\to V)$.
In the proof of Proposition \ref{prop-isolated}, we see that
\begin{itemize}
\item every $T$-weight vector in $P$, $V$ and $I$ is associated to a stone in the empty room configuration, and
\item every $T'$-weight space in $P$ is described as $\C[c]\cdot v_B$ for some stone $B$.
\end{itemize}
Assume we have a nonzero ${T'}$-invariant element $\phi\in \Hom_{A}\left(I,V\right)$. 
We may also assume that there is a positive integer $n$ such that
\[
\phi(v_B)=\begin{cases}
v_{B'} & \text{if $\exists\, [v_{B'}]\in V$ such that $c^n\cdot v_{B'}=v_B\in P$,}\\
0 & \text{otherwise}
\end{cases}
\]
for $v_B\in I$.

We take a stone $B$ and $B'$ such that 
\begin{itemize}
\item $(b_2a_2+a_2b_2)\cdot v_{B'}\in v_B$,
\item $v_{B'}\in I$ and $v_{B}\notin I$, and
\item there is no stone $B''$ such that $c\cdot v_{B''}\in v_B$.
\end{itemize}
In particular, we have $\phi(v_B)=0$.
Let $\beta$ be the minimal positive integer such that $(b_1a_1+a_1b_1)^\beta\cdot v_{B'}\in I$.

We put
\[
v:=c^n\cdot (b_1a_1+a_1b_1)^{\beta-1}\cdot v_{B'}=c^{n-1}\cdot (b_1a_1+a_1b_1)^{\beta}\cdot v_B\in I.
\]
Then we have
\[
0\neq (b_1a_1+a_1b_1)^{\beta-1}\cdot v_{B'}=\pi\bigl(c^n\cdot (b_1a_1+a_1b_1)^{\beta-1}\cdot v_{B'}\bigr)=\bigl(c^{n-1}\cdot (b_1a_1+a_1b_1)^{\beta} \bigr)=0
\]
This is a contradiction.
\end{proof}
\end{NB}
\begin{thm}\label{thm-sign}
\[
\mca{Z}_{\zeta}(\q)=\sum_{\vv\in\left(\Z_{\geq0}\right)^2}(-1)^{\mathrm{v}_1}\left|\,\M{\zeta}{\vv}^{T'}\right|\cdot\q^\vv.
\]
\end{thm}
\begin{proof}
As we mentioned in \S \ref{defofncdt}, this is a consequence of Behrend-Fantechi's result \cite[Theorem3.4]{behrend-fantechi}, Proposition \ref{prop-parameterization}, Corollary \ref{cor-zariski} and Proposition \ref{prop-isolated}.
\end{proof}

\bibliographystyle{amsalpha}
\bibliography{bib-ver6}

\providecommand{\bysame}{\leavevmode\hbox to3em{\hrulefill}\thinspace}
\providecommand{\MR}{\relax\ifhmode\unskip\space\fi MR }
\providecommand{\MRhref}[2]{%
  \href{http://www.ams.org/mathscinet-getitem?mr=#1}{#2}
}
\providecommand{\href}[2]{#2}
\begin{thebibliography}{MNOP06}

\bibitem[BB07]{superrigid-dt}
K.~Behrend and J.~Bryan, \emph{Super-rigid {Donaldson-Thomas} invariants},
  Math. Res. Let. \textbf{14} (2007).

\bibitem[BD]{berenstein-douglas}
D.~Berenstein and M.~Douglas, \emph{Seiberg duality for quiver gauge theories},
  arXiv:hep-th/0207027v1.

\bibitem[Beh09]{behrend-dt}
K.~Behrend, \emph{{Donaldson-Thomas} invariants via microlocal geometry}, Ann.
  of Math. \textbf{170} (2009), no.~3, 1307--1338.

\bibitem[BF08]{behrend-fantechi}
K.~Behrend and B.~Fantechi, \emph{Symmetric obstruction theories and {Hilbert}
  schemes of points on threefolds}, Alg. Number Theory \textbf{2} (2008),
  no.~3, 313--345.

\bibitem[Boc08]{graded3CY}
R.~Bocklandt, \emph{Graded {Calabi} {Yau} algebras of dimension $3$}, J. Pure
  Appl. Algebra \textbf{212} (2008), no.~1, 14--32.

\bibitem[Bri02]{bridgeland-flop}
T.~Bridgeland, \emph{Flops and derived categories}, Invent. Math. \textbf{147}
  (2002), no.~3, 613--632.

\bibitem[BTL]{BTL}
T.~Bridgeland and V.~Toledano-Laredo, \emph{Stability conditions and stokes
  factors}, arXiv:0801.3974v1.

\bibitem[CJ09]{chuang-jafferis}
W.~Chuang and D.~Jafferis, \emph{{Wall Crossing of BPS States on the Conifold
  from Seiberg Duality and Pyramid Partitions}}, Commun. Math. Phys.
  \textbf{292} (2009), 285--301.

\bibitem[DWZ08]{quiver-with-potentials}
H.~Derksen, J.~Weyman, and A.~Zelevinsky, \emph{Quivers with potentials and
  their representations {I}: Mutations}, Selecta Math. \textbf{14} (2008).

\bibitem[Got90]{gottsche}
L.~Gottsche, \emph{The betti numbers of the {Hilbert} scheme of points on a
  smooth projective surface}, Math. Ann. \textbf{286} (1990), 193--207.

\bibitem[He98]{He}
M.~He, \emph{Espaces de modules de syst\`emes coh\'erents}, Internat. J. Math.
  \textbf{9} (1998), no.~5, 545--598.

\bibitem[HL97]{HL}
D.~Huybrechts and M.~Lehn, \emph{The geometry of moduli spaces of sheaves},
  Aspects of Mathematics, E31, Friedr. Vieweg \& Sohn, Braunschweig, 1997.

\bibitem[JM]{jafferis-moore}
D.~Jafferis and D.~Moore, \emph{Wall crossing in local {Calabi-Yau} manifolds},
  arXiv:0810.4909v1.

\bibitem[Joy08]{joyce-4}
D.~Joyce, \emph{Configurations in abelian categories {IV}. {Invariants} and
  changing stability conditions}, Advances in Math \textbf{217} (2008), no.~1,
  125--204.

\bibitem[Kin94]{king}
A.D. King, \emph{Moduli of renresentations of finite dimensional algebras}, J.
  Algebra. \textbf{45} (1994), no.~4, 515--530.

\bibitem[KS]{ks}
M.~Kontsevich and Y.~Soibelman, \emph{Stability structures, motivic
  {Donaldson-Thomas} invariants and cluster transformations},
  arXiv:0811.2435v1.

\bibitem[LP93]{LePotier}
J.~Le~Potier, \emph{Syst\`emes coh\'erents et structures de niveau},
  Ast\'erisque (1993), no.~214, 143 pp.

\bibitem[MNOP06]{mnop}
D.~Maulik, N.~Nekrasov, A.~Okounkov, and R.~Pandharipande,
  \emph{{Gromov-Witten} theory and {Donaldson-Thomas} theory, {I}}, Comp. Math.
  \textbf{142} (2006), 1263--1285.

\bibitem[MR10]{ncdt-brane}
S.~Mozgovoy and M.~Reineke, \emph{On the noncommutative {Donaldson-Thomas}
  invariants arising from brane tilings}, Advances in mathematics \textbf{223}
  (2010).

\bibitem[Nag]{3tcy}
K.~Nagao, \emph{Derived categories of small $3$-dimensional toric {Calabi-Yau}
  varieties and curve counting invariants}, arXiv:0809.2994v3.

\bibitem[NYa]{ny-perv1}
H.~Nakajima and K.~Yoshioka, \emph{Perverse coherent sheaves on blow-up. {I}. a
  quiver description}, arXiv:0802.3120v2.

\bibitem[NYb]{ny-perv2}
\bysame, \emph{Perverse coherent sheaves on blow-up. {II}. wall-crossing and
  {Betti} numbers formula}, J. of Algebraic Geom., posted on March 23, 2010,
  PII S 1056-3911(10)00534-5 (to appear in print).

\bibitem[PT09]{pt1}
R.~Pandharipande and R.P. Thomas, \emph{Curve counting via stable pairs in the
  derived category}, Invent. Math. \textbf{178} (2009), no.~2.

\bibitem[Rud97]{rudakov}
A.~Rudakov, \emph{Stability for an {Abelian} category}, J. Algebra.
  \textbf{197} (1997), no.~1, 231--245.

\bibitem[Sim94]{S:1}
C.T. Simpson, \emph{Moduli of representations of the fundamental group of a
  smooth projective variety. {I}}, Inst. Hautes \'Etudes Sci. Publ. Math.
  (1994), no.~79, 47--129.

\bibitem[ST]{thomas_stoppa}
J.~Stoppa and R.P. Thomas, \emph{Hilbert schemes and stable pairs: {GIT} and
  derived category wall crossings}, arXiv:0903.1444v3.

\bibitem[Sze08]{szendroi-ncdt}
B.~Szendroi, \emph{Non-commutative {Donaldson-Thomas} invariants and the
  conifold}, Geom. Topol. \textbf{12(2)} (2008), 1171--1202.

\bibitem[Tho00]{thomas-dt}
R.~P. Thomas, \emph{A holomorphic {Casson} invariant for {Calabi-Yau}
  $3$-folds, and bundles on {$K3$} fibrations}, J. Differential Geom.
  \textbf{54} (2000), no.~2, 367--438.

\bibitem[Tod10]{toda-dtpt}
Y.~Toda, \emph{Curve counting theories via stable objects {I}. {DT/PT}
  correspondence}, {J. Amer. Math. Soc.} \textbf{23} (2010), 1119--1157.

\bibitem[VdB04]{vandenbergh-3d}
M.~Van~den Bergh, \emph{Three-dimensional flops and noncommutative rings}, Duke
  Math. J. \textbf{122} (2004), no.~3, 423--455.

\bibitem[Yos]{yoshioka_perverse}
K.~Yoshioka, \emph{Perverse coherent sheaves and {Fourier-Mukai} transforms on
  surfaces}, arXiv:1003.2522v1.

\bibitem[Yos03]{Ytwisted2}
\bysame, \emph{Twisted stability and {F}ourier-{M}ukai transform. {II}},
  Manuscripta Math. \textbf{110} (2003), no.~4, 433--465.

\bibitem[You]{young-conifold}
B.~Young, \emph{Computing a pyramid partition generating function with dimer
  shuffling}, J. Combin. Theory Ser. A \textbf{116}, no.~2.

\end{thebibliography}

\end{document}